\documentclass [11pt, twoside] {uwthesis}[2012/06/19]
\usepackage{hyperref}
\usepackage[stretch=10,shrink=10]{microtype}
\usepackage{verbatim}
\usepackage{amsmath,amsthm,amsfonts}
\usepackage{eucal}
\usepackage{tikz} 
\usepackage{ifthen}
\usepackage{theoremref}
\usepackage{changepage}
\usepackage{amssymb}
\usetikzlibrary{calc}
\usetikzlibrary{decorations.pathmorphing}
\usepackage{enumerate}
\usepackage{constants}
\renewconstantfamily{normal}{symbol=c}
\tikzset{vert/.style={circle,fill,inner sep=0,
    minimum size=0.12cm,draw}}
\newtheorem{thm}{Theorem}[chapter]
\newtheorem{lem}[thm]{Lemma}
\newtheorem{cor}[thm]{Corollary}

\newtheorem{prop}[thm]{Proposition}
\theoremstyle{definition}
\newtheorem{definition}[thm]{Definition}

\newtheorem{claim}[thm]{Claim}
\newtheorem{rmk}[thm]{Remark}

\newcommand{\defeq}{\triangleq}
     \newcommand{\leb}{\text{Leb}}
\newcommand{\yulelaw}[1]{\widetilde{P}_{#1}}
  \newcommand{\tr}{\mathop{\mathrm{tr}}\nolimits}
\newcommand{\divides}{\mid}
  \newcommand{\floor}[1]{\lfloor #1 \rfloor}
  \newcommand{\limn}{\lim_{n\to\infty}}
  \newcommand{\eps}{\epsilon}
  \newcommand{\step}[2]{\par\medskip\par\noindent \textbf{Step #1}.
  \textit{#2}} 
\renewcommand{\subset}{\subseteq}
  \newcommand{\Poi}{\mathrm{Poi}}
  \newcommand{\Bin}{\mathrm{Bin}}
  \newcommand{\Exp}{\mathrm{Exp}}
  \newcommand{\Geo}{\mathrm{Geo}}
\newcommand{\1}{\mathbf{1}}
\newcommand{\eqd}{\sim}
  \newcommand{\toPr}{\,{\buildrel pr \over \longrightarrow}\,} 
  \newcommand{\toL}{\,{\buildrel \mathcal{L} \over \longrightarrow}\,} 
\newcommand{\NN}{\mathbb{N}}
\newcommand{\CC}{\mathbb{C}}
\newcommand{\RR}{\mathbb{R}}
\newcommand{\HH}{\mathbb{H}}
\newcommand{\ZZ}{\mathbb{Z}}
  \newcommand{\GGG}{\mathfrak{G}}
\newcommand{\abs}[1]{\lvert #1 \rvert}
\newcommand{\dpoch}[2]{[#1]_{#2}}
\newcommand{\Aa}{\mathcal{A}}
\newcommand{\Ss}{\mathcal{S}}
\newcommand{\Tt}{\mathcal{T}}
\newcommand{\Bb}{\mathcal{B}}
\newcommand{\Ii}{\mathcal{I}}
\newcommand{\Jj}{\mathcal{J}}
\newcommand{\Kk}{\mathcal{K}}
\newcommand{\Ff}{\mathcal{F}}
\newcommand{\Hh}{\mathcal{H}}
\newcommand{\Ww}{\mathcal{W}}
\newcommand{\E}{\mathbf{E}}
\newcommand{\var}{\mathbf{Var}}
\newcommand{\CNBW}[2][\infty]{\mathrm{CNBW}_{#2}^{(#1)}}

\renewcommand{\P}{\mathbf{P}}
\newcommand{\dtv}{d_{TV}}
  \newcommand{\cov}{\mathop{\mathbf{Cov}}\nolimits}
\newcommand{\Y}{\mathbf{Y}}
     \newcommand{\ip}[2]{\left\langle #1,#2 \right\rangle }
    \newcommand{\labeledarrow}[3]{\tikz[baseline]{
    \node[anchor=base west] 
    (n1) at (0,0)
    {\ensuremath{#1}}; \path (n1.base east)+(1.1,0) node[anchor=base west](n2) 
    {\ensuremath{#2}};
    \draw[->](n1.mid east)--node[auto,font=\scriptsize,yshift=-0.07cm]
      (lb) {\ensuremath{#3}}(n2.mid west);  
      \pgfresetboundingbox
      \useasboundingbox (0.08,0) rectangle ($(n2.north east |- lb.north)
         +(-0.08cm,-0.05cm)$);}}

\begin{document}
 
%
%

\prelimpages
 
%
%
\Title{Eigenvalue fluctuations for random regular graphs}
\Author{Tobias Lee Johnson}
\Year{2014}
\Program{Department of Mathematics}

\Chair{Ioana Dumitriu}{Associate Professor}{Department of Mathematics}
\Chair{Soumik Pal}{Associate Professor}{Department of Mathematics}
\Signature{Sara Billey}

\copyrightpage



{\Degreetext{A dissertation
  submitted in partial fulfillment of the\\ requirements for the degree of}
 \titlepage
 }
\setcounter{footnote}{0}

 
%
%

%
%

\setcounter{page}{-1}
\abstract{%
One of the major themes of random matrix theory is that many asymptotic properties
of traditionally studied distributions of random matrices are \emph{universal}.
We probe the edges of universality by studying the spectral properties of random regular graphs.
Specifically, we prove limit theorems for the fluctuations
of linear spectral statistics of random regular graphs. We find both universal and non-universal behavior.
Our most important tool is Stein's method for Poisson approximation, which we develop for use
on random regular graphs.
}
 
%
%
\tableofcontents

%
%
\acknowledgments{
  I owe a debt to many people for their help and friendship.
  I thank Ioana for introducing me to research and to random regular graphs.
  I'm grateful to her for supporting me and for not taking my curmudgeonliness
  too seriously. I'm grateful to Soumik for teaching me probability and shaping my view
  of it. His suggestion that I learn how to use Stein's method and try to apply it
  to cycle counts of random regular graphs has shaped my mathematical career more than I 
  would ever have  imagined. I also thank Chris Hoffman, who has been
  a third advisor to me this year, introducing me to new areas of math and collaborating with me.
  I can't thank Elliot Paquette and Matt Junge enough. It's been wonderful to have them as friends
  and collaborators. Most of all, 
  I thank my friends and my wife Lindsay.
}

%
%

%

%
%

\textpages
 
 
\chapter {Introduction}

\section{How universal is universality?}
Random matrix theory traditionally studies certain random matrices of interest to
physicists and statisticians. 
The central question of classical random matrix theory is to prove that the eigenvalues
of random matrices' show \emph{universal} behavior as the size of the random matrices grow.
Universality is not a precise concept. The classical central limit theorem gives
an example of it:  with only light conditions
on a collection of random variables (being i.i.d.\ with finite variances), their centered and
normalized sums converge in law to Gaussian. 

The most basic symmetric random matrix model
is the \emph{Gaussian Orthogonal Ensemble}, abbreviated GOE.
Let $G$ be an $n\times n$ matrix whose entries are independent and distributed as $N(0,2)$. 
Define $X$ as $(G+G^T)/2$, a random symmetric matrix with independent
entries on and above the diagonal. The random matrix $X$ has centered Gaussian entries with
variance~$1$ above the diagonal and variance~$2$ on the diagonal, and it 
is said to be drawn from the GOE. 
Any $n\times n$ random matrix with centered independent entries on and above the diagonal and variance~$1$
entries above the diagonal is called a \emph{Wigner matrix}.
(The word ``ensemble'' does not have any precise meaning,
but it is usually refers to a collection of probability distributions on $n\times n$ matrices,
as $n$ ranges from $1$ to infinity. Each distribution typically obeys some sort of invariance.
For instance, if $O$ is an arbitrary orthogonal matrix and $X$ is drawn from the GOE, then
$OX$ has the same distribution as $X$.)

An example of universality for random matrices is that the eigenvalues of $n\times n$ Wigner matrices
show the same limiting behavior as those of matrices from the GOE as $n\to\infty$.
Most results along these lines were confirmed only recently, in a series
    of papers including \cite{TV1,TV2,ESY1,ESY2,ESY3,ESY4,ESY5}.
    
  The adjacency matrix of a random regular graph is similar to a Wigner matrix, but its entries
  are uncentered and lightly dependent. How does this affect the adjacency matrix's spectral properties?
   To put it another way, how universal is universality
  of random matrices? This is our main motivation for investigating properties of eigenvalues
  of random regular graphs from the perspective of random matrix theory.
\section{Stein's method applied to random regular graphs}
Graph eigenvalues have a close connection to the graph's structural properties
(see \cite{Chung,Spielman}). We exploit this by determining spectral properties
of random regular graphs by looking at the distribution of their cycle counts.
The main novelty of our approach is the use of \emph{Stein's method}, which to our knowledge
had never been applied to random regular graphs before.
Stein's method is a collection of techniques for distributional approximation. Stein's method
naturally gives not just asymptotic results but also quantitative error bounds on the approximations.
This was essential for the eigenvalue fluctuation results described in this thesis.

Stein's method  was originally developed by Charles
Stein for normal approximation; its first published
use is \cite{St1}. Louis Chen adapted the method 
for Poisson approximation \cite{Chen1}.  
Because of this, Stein's method is sometimes called the
Stein-Chen or Chen-Stein method when applied to Poisson approximation.
Now that Stein's method is understood in a more general and applied
to a wide range of distributions, it is more typical to see it called
just Stein's method, regardless of the type of approximation.
The survey
paper \cite{Ross} gives a broad introduction to Stein's method, 
and \cite{CDM} and
\cite{BHJ} focus specifically on using it for Poisson approximation, as we do in
this thesis.

The classical scenario for Poisson approximation is for sums of increasingly many, increasingly
unlikely independent indicators: in other words, the convergence of $\Bin(n, \lambda/n)$
to $\Poi(\lambda)$ as $n\to\infty$.
There are several approaches to Stein's method for Poisson approximation, each allowing
this approximation to hold in the presence of some dependence.
The most straightforward is the local approach: each indicator is independent of all others
but a small ``neighborhood''. This was the original approach in \cite{Chen1}, and it is generalized
and put in a very usable form in \cite{AGG}. This approach does not seem to work in the context
of random regular graphs, where nearly everything is lightly dependent on everything else.
Another approach is size-bias coupling. This theory  is developed at length for Poisson approximation
in \cite{BHJ}, though it is not viewed through the lens of size-biasing there.
See\cite{Ross} and \cite{AGK} for how it fits into this framework. We use this method on the
\emph{permutation model} of random regular graph (see Section~\ref{sec:models} for its definition).
Another technique is the method of exchangeable pairs; see \cite{CDM} and \cite{Ross} for good
expositions. This technique is perhaps the most flexible and the most finicky of the three.
We use it for Poisson approximation in the \emph{uniform model} of random regular graph, defined
in Section~\ref{sec:models}. This technique has some clear similarities to a combinatorial
technique called  \emph{the method of switchings}, and we make some rigorous connections between the two.
\section{The results of this thesis}
  Consider an $n\times n$ Wigner random matrix $X_n$ (a symmetric matrix with independent, mean zero,
  variance one entries above the diagonal). Choose an interval in the real line, 
  and let $N_n$ denote the
  number of eigenvalues of $n^{-1/2}X_n$ lying in this interval. A fundamental result in random matrix
  theory is that $N_n /n$ converges in probability to a deterministic value as $n$ tends to infinity. This
  value is the measure of the interval
  under Wigner's semicircle law, the measure on $[-2,2]$ given by the density $\frac{1}{2\pi}
  \sqrt{4-x^2}\,dx$. This measure is a universal limit, in the sense
  that it does not depend on the distributions of the individual matrix entries, besides their
  means and variances.
  
  The analogue of this result for random regular graphs appears in \cite{McK}:
  Let $\lambda_1,\ldots,\lambda_n$ be the eigenvalues of a random $d$-regular graph on $n$ vertices. 
  If $f\colon \RR\to\RR$ is an indicator
    on an interval or is bounded and continuous, then as $n\to\infty$,
    \begin{align*}
      \frac{1}{n}\sum_{i=1}^n f(\lambda_i) \toPr \int_{-2\sqrt{d-1}}^{2\sqrt{d-1}} f(x)p_d(x)\,dx.
    \end{align*}
  The limiting measure $p_d(x)\,dx$ is not the semicircle law, but a different
  measure known now as the
  Kesten--McKay law. Its density is given on $\abs{x}\leq 2\sqrt{d-1}$ by
  \begin{align}
    p_d(x)=\frac{d\sqrt{4(d-1)-x^2}}{2\pi(d^2-x^2)}.\label{eq:kmlaw}
  \end{align}
  The expression $\sum f(\lambda_i)$ is called a \emph{linear eigenvalue statistic}.

  The topic of this thesis is the second-order behavior of these linear statistics.
  We will show that when the degree of the random graphs is held fixed,
  their fluctuations converge to compound
  Poisson distributions, in contrast to the Gaussian limit known for Wigner matrices.
  If the degree grows with the size of the graph, however, the limit of the fluctuations is
  Gaussian, in line with the universal behavior.
  We show that this holds in two models of random regular graphs,
  defined in Section~\ref{sec:models}.
  
  The path to these results is through an analysis of the distribution of cycle counts in these
  models by Stein's method. These results are interesting in their own right, and they make up
  Chapter~\ref{chap:poisson}. In Chapter~\ref{chap:fluctuations}, we apply them to prove
  the eigenvalue fluctuation results.
  
  In Chapter~\ref{chap:GFF}, we consider a process of growing random regular graphs.
  The eigenvalue fluctuations are then a stochastic process whose marginals are given
  by the results of Chapter~\ref{chap:fluctuations}. This is analogous to a \emph{corners
  process} in random matrix theory; see \cite{BG} for a good introduction. The idea
  is to think of a sequence of random matrices as the principal minors of an infinite random matrix.
  One can then consider not just the marginal distribution of the eigenvalues of each random matrix,
  but the joint distribution of eigenvalues of a matrix and its minors.
  The limiting fluctuations of some of these processes can be expressed in terms
  of the Gaussian free field \cite{Bor1,Bor2,BG}. We show that the same holds for the
  eigenvalues of the growing random regular graphs.
   
Most of this thesis is joint work.
Chapters~\ref{chap:poisson} and~\ref{chap:fluctuations} are a synthesis of \cite{DJPP},
\cite{JP},
and \cite{Joh}. The results on the permutation model are from \cite{DJPP}, which
is joint with Ioana Dumitriu, Elliot Paquette, and Soumik Pal, and from \cite{JP}, which
is joint with Pal. The results on the uniform
model are from \cite{Joh}. (See Section~\ref{sec:models} for the definitions of these two
models of random regular graphs). 
\thref{thm:bestpermutationpoiapprox}, a version
of \cite[Corollary~24i]{JP} with an improved rate, appears only in this thesis.

Chapter~\ref{chap:GFF} is mostly taken from \cite{Joh}, which is joint work with Pal.
Section~\ref{sec:GFFconvergence} is new to this thesis and was also done jointly
with Pal. (The exception is Section~\ref{sec:GFFbackground2}, an extended introduction
to the Gaussian free field. It and any errors contained in it are mine alone.)
The main result here is \thref{thm:GFFconvergence}, which shows
the convergence of eigenvalue fluctuations to the Gaussian free field
in a more explicit form than in \cite{JP}.

\section{Models of random regular graphs}
\label{sec:models}
In Chapters~\ref{chap:poisson} and~\ref{chap:fluctuations}, we
will present results on two models of random regular graphs, the \emph{permutation mode} and
the \emph{uniform model}. 
Traditionally, combinatorialists were most concerned with the uniform model of random regular
  graphs. The permutation model is typically easier to work with, however, and it is the setting
  for many spectral results on random regular graphs (for example,
  \cite{BrSh, Fri91, Fri}). There seems to have been a sense that that the two models
  had basically the same properties, besides the permutation model having loops
  and multiple edges. The contiguity result in \cite{GJKW} justifies this somewhat.

We now review the definitions of these two models and of our sequence of growing
graphs.
\subsection{The uniform model}
A random $d$-regular graph on $n$ vertices drawn from the uniform model
is just a graph chosen uniformly from the set of all $d$-regular graphs
  (i.e., graphs where every vertex has degree exactly $d$)
  on $n$ vertices
  without loops or multiple edges. Such graphs only exist when $nd$ is even.
\subsection{The permutation model}
\label{sec:permutationmodel}
  The permutation model is given by choosing
  $d/2$ independent, uniformly random permutations on $n$ vertices, making a graph from
  the cycle structure of each permutation, and overlaying them. It exists only
  for even values of $d$. For a more formal definition, let $\pi_1,\ldots,\pi_{d/2}$ be 
  independent, uniformly random permutations on $n$ vertices. Define a graph on vertices
  $\{1,\ldots,n\}$ by making an edge between vertices $x$ and $y$ for every $k$ 
  such that $\pi_k(x)=y$. This model allows loops and multiple edges.
  We consider a loop at vertex~$x$ as counting as two edges when computing the degree of $x$,
  so that the graph really is $d$-regular. We also count a loop at vertex~$i$ as increasing the
  graph's adjacency matrix by $2$ at position $(i,i)$. The adjacency matrix of a graph from
  this model is then a sum of independent permutation matrices.

\subsection{Growing random regular graphs}
\label{sec:growingrrg}
A \emph{tower of random permutations} is a sequence of random permutations $(\pi^{(n)}, n\in \NN)$ such that
\begin{enumerate}[(i)]
\item $\pi^{(n)}$ is a uniformly distributed random permutation of $\{1,\ldots,n\}$, and
\item for each $n$, if $\pi^{(n)}$ is written as a product of cycles then $\pi^{(n-1)}$ is derived from $\pi^{(n)}$ by deletion of the element $n$ from its cycle. 
\end{enumerate}
The stochastic process that grows $\pi^{(n)}$ from $\pi^{(n-1)}$ by sequentially inserting an element $n$ randomly is called the Chinese Restaurant Process.
 We give a further review of it in Section~\ref{sec:CRPprimer}.

Now suppose we construct towers of random permutations $(\pi_d^{(n)},\,n\geq 1)$, independent for each $d$.
For any $n$ and $d$, we can define a random $2d$-regular graph $G(n, 2d)$
from $\{ \pi^{(n)}_j,\, 1\le j \le d \}$ as in Section~\ref{sec:permutationmodel}.
Marginally, $G(n,2d)$ is then a random graph from the permutation model.
We will often keep $d$ fixed and consider $n$ as a growing parameter, referring to $G(n,2d)$
as $G_n$. Here and later, $G_0$ will represent the empty graph. 

We construct a continuous-time version of this by inserting new vertices
into $G_n$ with rate $n+1$.  Formally,
define independent \label{page:poissonization}
times $T_i\eqd\Exp(i)$, and let 
\begin{align*}
  M_t=\max\Big\{m\colon\ \sum_{i=1}^mT_i\leq t\Big\},
\end{align*}
and define the continuous-time Markov chain
$G(t)=G_{M_t}$.\label{page:chaindef}
When we vary $d$ as well as $n$, we will also refer to this as $G(t,2d)$.

\chapter{Poisson approximation for cycle counts in random regular graphs}
\label{chap:poisson}
\newcommand{\Cy}[1]{C_{#1}}

Let $\Cy{k}$ denote the number of cycles of length
$k$ in a random graph $G_n$.
The distribution of these random variables has been studied since 
\cite{Bol80,Wor81}, where it was proven that if $G$ is a uniform
$d$-random regular graph on $n$ vertices, then
 $(\Cy{3},\ldots,\Cy{r})$ converges in law
to a vector of independent Poisson random variables as $n$ tends to infinity,
with $r$ held fixed.

The strongest results on the cycle counts of a random regular graph came
in \cite{MWW}, where the Poisson approximation was shown to hold even
as $d=d(n)$ and $r=r(n)$ grow with $n$, so long as $(d-1)^{2r-1}=o(n)$.
This is a natural boundary: in this asymptotic regime, 
all cycles in $G_n$ of length $r$ or less have disjoint edges,
asymptotically almost surely. If $(d-1)^{2r-1}$ grows any
faster, this fails.  This led the authors in \cite{MWW} to speculate that the
Poisson approximation failed beyond this threshold.
Surprisingly, this is not the case. 
We will show that the Poisson approximation holds slightly beyond this threshold.
We also give a quantitative bound on the accuracy
of the approximation, which was our original motivation and is
the necessary ingredient for our results on linear eigenvalue statistics.

We will give results on both the permutation model and the uniform model
of random regular graphs. 
We use Stein's method in both cases, but we use different techniques for the two
models: size-biased couplings
for the permutation model and exchangeable pairs for the uniform model.
We provide background and references on these techniques in the following section.

Before we go any further, we present the main results of this section.
Rather than showing that the cycle counts are approximately Poisson, we will
make a more general statement about process made up of the cycles themselves.
To state our results, we must explain exactly what we mean by a cycle in a graph.

We start by discussing the permutation model.
  Let $G_n$ be a random $2d$-regular graph on $n$ vertices from the permutation
  model, formed from the independent permutations
  $\pi_1,\ldots,\pi_d$ as described in Section~\ref{sec:models}.
  This graph can be considered as a directed,\label{permmodeldiscussion} 
  edge-labeled
  graph in a natural way.
  If $\pi_l(i)=j$, then
  by definition $G_n$ contains an edge between $i$ to $j$.  When convenient,
  we consider this edge to be directed from $i$ to $j$ and to be
  labeled by $\pi_l$.
  
  Consider a walk on $G_n$, viewed in this way, and imagine writing
  down the label of each edge as it is traversed, putting
  $\pi_i$ or $\pi_i^{-1}$ according to the direction
  we walk over the edge.
  We call a walk \emph{closed} if it starts and ends at the same vertex,
  and we call a closed walk a cycle
  it never visits a vertex twice (besides the first and last one), 
  and it never traverses an
  edge more than once in either direction.  
  Thus the word $w=w_1\cdots w_k$ formed as a cycle is traversed
  is cyclically reduced, i.e., $w_i\neq w_{i+1}^{-1}$ for all $i$, 
  considering $i$ modulo
  $k$.  For example,
  following an edge and then immediately 
  backtracking does not form a $2$-cycle,
  and the word formed by this walk is 
  $\pi_i\pi_i^{-1}$ or $\pi_i^{-1}\pi_i$ for some
  $i$, which is not cyclically reduced.  We consider two cycles
  equivalent if they are both walks on an identical set of edges;
  that is, we ignore the starting vertex and the direction of the walk.
  We will often denote the length of a cycle $\alpha$ by $\abs{\alpha}$.
  
  \begin{definition}
    Let $\Ii_k$ be the set of all $k$-cycles in the complete graph on $n$ vertices with
    edges labeled by $\pi_1^{\pm 1},\ldots,\pi_d^{\pm 1}$, where the word formed as the
    cycle is traversed is cyclically reduced. Let $a(d,k)$ be number of cyclically
    reduced words of length~$k$ in this alphabet.
  \end{definition}
  Observe that $\abs{\Ii_k} = \dpoch{n}{k}a(d,k)/2k$, where 
  $\dpoch{n}{k}=n(n-1)\cdots (n-k+1)$. By an inclusion-exclusion argument 
  \cite[Lemma~41]{DJPP},
  \begin{align}
    a(d, k) = \begin{cases}
      (2d-1)^k - 1 + 2d & \text{if $k$ is even,}\\
      (2d-1)^k + 1 & \text{if $k$ is odd.}
    \end{cases}\label{eq:adk}
  \end{align}

We are now ready to state the main Poisson approximation result for the
permutation model.
\newcommand{\I}{\mathbf{I}}
\newcommand{\Z}{\mathbf{Z}}
\begin{thm}[Theorem~14 in \cite{JP}]\thlabel{thm:processapprox}
Let $G_n$ be a random $2d$-regular graph
on $n$ vertices from the permutation model. Let 
$\Ii=\bigcup_{k=1}^r\Ii_k$ for some integer $r$.
For any cycle $\alpha\in\Ii$, let 
$I_\alpha=1\{\text{$G_n$ contains $\alpha$}\}$, and let $\I=(I_\alpha,\,\alpha
\in\Ii)$.
Let $\Z=(Z_\alpha,\,\alpha\in\Ii)$ be a vector whose coordinates
are independent Poisson random variables with $\E Z_\alpha=1/[n]_k$
for $\alpha\in\Ii_k$.  Then for all $d\geq 2$ and $n,r\geq 1$,
\begin{align*}
  \dtv(\I, \Z)
  &\leq \frac{c(2d-1)^{2r-1}}{n} 
\end{align*}
for some absolute constant $c$.
\end{thm} 
In the uniform model, there are no edge labels, and a cycle is simply
a closed walk repeating no vertices. Again, we consider two walks equivalent if they are 
walks on the same set of edges.
\begin{thm}[Corollary~8 in \cite{Joh}]\label{thm:uniformprocesspoiapprox}
  Let $G_n$ be a random $d$-regular graph on $n$ vertices from the uniform model,
  and let $\Ii$ be the collection of all cycles of length $r$
  or less in the complete graph $K_n$.
  For any cycle $\alpha\in\Ii$, let 
$I_\alpha=1\{\text{$G_n$ contains $\alpha$}\}$, and let $\I=(I_\alpha,\,\alpha
\in\Ii)$.
Let $\Z=(Z_\alpha,\,\alpha\in\Ii)$ be a vector whose coordinates
are independent Poisson random variables with $\E Z_\alpha=(d-1)^{\abs{\alpha}}/[n]_{\abs{\alpha}}$.
  For some absolute constant $c$, for all $n$
  and $d,r\geq 3$,
  \begin{align*}
    \dtv(\I,\,\Z) &\leq
      \frac{c(d-1)^{2r-1}}{n}.
  \end{align*}
\end{thm}

These theorems immediately imply that the vectors of cycle counts
of length $r$ or less in the permutation and uniform models
are also within $O((2d-1)^{2r-1}/n)$ and $O((d-1)^{2r-1}/n)$, respectively, of vectors
of independent Poissons. In fact, we can do slightly better:

\begin{thm}\thlabel{thm:bestpermutationpoiapprox}
  Let $G_n$ be a random $2d$-regular graph on $n$ vertices from the permutation model with
  cycle counts $(\Cy{k},\,k\geq 1)$. Let $Z_k,\,k\geq 1$ be independent Poisson random variables
  with $\E Z_k = a(d,k)/2k$. For any $d\geq 2$ and $n,r\geq 1$,
  \begin{align*}
    \dtv\bigl((\Cy{1},\ldots,\Cy{r}),\,(Z_1,\ldots,Z_r)\bigr)
      &\leq \frac{cr^2(2d-1)^r\log(2d-1)}{n},
  \end{align*}
  for some absolute constant~$c$.
\end{thm}

\begin{thm}[Theorem~11 in \cite{DJPP}]\thlabel{thm:bestpoiapprox}
  Let $G_n$ be a random $d$-regular graph on $n$ vertices from the uniform model
  with cycle counts $(\Cy{k},\,k\geq 3)$.
  Let $(Z_k,\,k\geq 3)$ be independent Poisson random variables
  with $\E Z_k=(d-1)^k/2k$.  For any $n\geq 1$ and $r,d\geq 3$,
  \begin{align*}
    \dtv\big((\Cy{3},\ldots,\Cy{r}),\,(Z_3,\ldots,Z_r)\big)
      &\leq\frac{c\sqrt{r}(d-1)^{3r/2-1}}{n}
   \end{align*}
   for some absolute constant~$c$.
\end{thm}

\section{Background on Stein's method}

\subsection{Size-bias couplings}
To give some intuition behind size-bias couplings, let us go to the archetypal setting
for Poisson approximation. Let $I_1,\ldots,I_n$ be independent Bernoulli random variables,
equal to $1$ with probability $1/n$ and $0$ with probability $(n-1)/n$.
Let $X$ be the sum of these, which makes its distribution $\Bin(n,1/n)$.
Define $X'$ to be $X-I_N+1$, where $N$ is uniformly chosen from $\{1,\ldots,n\}$,
independently of everything else. In other words, $X'$ is given by taking one of the
indicators at random and forcing it to be $1$. It is not hard to show that $X'$ is a size-biased
version of $X$, meaning that
\begin{align*}
  \P[X'=k] &= \frac{k}{\E X}\P[X=k].
\end{align*}
For large~$n$, we have $X'\approx 1+X$.

This definition of a size-biased version of $X$ defined on the same probability space
is an example of a more general construction; see \cite[Section~3.4.1]{Ross}.
In general, if a random variable $X$ can be coupled with $X'$,
a size-biased version of itself, and $X'$ is close to  $X+1$ in $L^1$, then $X$ is approximately
Poisson. A precise verison of this statement is \cite[Theorem~4.13]{Ross}.

We will use a formulation of this idea from \cite{BHJ}. This formulation never
explicitly make a size-biased version of the random variable to be approximated, but
its idea is exactly the same.
Recall the definition of $(I_\beta,\,\beta\in\Ii)$ from Theorem~\ref{thm:processapprox}.
For each $\alpha\in\Ii$, let $(J_{\beta\alpha},\,\beta\in\Ii)$ 
be distributed
as $(I_\beta,\,\beta\in\Ii)$ conditioned on $I_\alpha=1$.
The goal is to construct a coupling of $(I_\beta,\,\beta\in\Ii)$
and 
$(J_{\beta\alpha},\,\beta\in\Ii)$ so that the two random vectors
are ``close together''.  We hope that for each $\alpha\in\Ii$,
the cycles in  $\Ii\setminus\{\alpha\}$ can be partitioned
into two sets $\Ii_\alpha^{-}$ and $\Ii_\alpha^+$ such that
\begin{align}
  J_{\beta\alpha}&\leq I_\beta\quad\text{if $\beta\in\Ii_\alpha^-$,}
  \label{eq:monotone-}\\
  J_{\beta\alpha}&\geq I_\beta\quad\text{if $\beta\in\Ii_\alpha^+$.}
  \label{eq:monotone+}
\end{align}
If this is the case, then one can approximate $(I_\beta,\,\beta\in\Ii)$ by a 
Poisson process
by calculating $\cov(I_\alpha, I_\beta)$ for every $\alpha,\beta\in\Ii$,
according to the following proposition.
\begin{prop}[Corollary~10.J.1 in \cite{BHJ}]\label{prop:BHJprop}
  Suppose that $\I=(I_\alpha,\,\alpha\in\Ii)$ is a vector
  of 0-1 random variables with $\E I_\alpha=p_\alpha$.
  Suppose that $(J_{\beta\alpha},\,\beta\in\Ii)$ is distributed
  as described above, and that for each $\alpha$ there exists
  a partition and a coupling of $(J_{\beta\alpha},\,\beta\in\Ii)$
  with $(I_\beta,\,\beta\in\Ii)$ such that \eqref{eq:monotone-}
  and \eqref{eq:monotone+}
  are satisfied.
  
  Let $\Y=(Y_\alpha,\,\alpha\in\Ii)$ be a vector of independent
  Poisson random variables with $\E Y_\alpha=p_\alpha$.  Then
  \begin{align}
    \dtv(\I, \Y)&\leq \sum_{\alpha\in\Ii}p_\alpha^2
      +\sum_{\alpha\in\Ii}\sum_{\smash{\beta\in\Ii_\alpha^-}}
      \abs{\cov(I_\alpha,I_\beta)}
      +\sum_{\alpha\in\Ii}\sum_{\smash{\beta\in\Ii_\alpha^+}}\cov(I_\alpha,I_\beta).
      \label{eq:BHJ}
  \end{align}
\end{prop}

By bunching together indicators into bins, we can slightly improve the rates:
\begin{prop}[Theorem~10.K in \cite{BHJ}]\label{prop:BHJbin}
  Assume all the conditions of the previous proposition.
  Suppose that we partition the index set as $\Ii = \bigcup_{k=1}^r\Ii_k$,
  and define
  \begin{align*}
    W_k = \sum_{\alpha\in\Ii_k} I_k,\qquad\qquad Y_k = \sum_{\alpha\in\Ii_k}Y_\alpha.
  \end{align*}
  Let $\lambda_k=\E Y_k$.
  \begin{equation}
    \begin{split}
    \dtv\bigl((W_1,&\ldots,W_r),\;(Y_1,\ldots,Y_r)\bigr)\\ &\leq
      2(1+e^{-1}\log^+\max\lambda_j)\Biggl(
        \sum_{k=1}^r\sum_{\alpha\in\Ii_k}\frac{p_\alpha^2}{\lambda_k} 
        +\sum_{j,k=1}^r\frac{A(j,k)}{\sqrt{\lambda_j\lambda_k}}\Biggl),
  \end{split}\label{eq:BHJbin}
  \end{equation}
  where
  \begin{align*}
    A(j,k)&=
        \sum_{\alpha\in\Ii_k}\Biggl(
          \sum_{\beta\in\Ii_\alpha^-\cap\Ii_j}\abs{\cov(I_\alpha,I_\beta)}
          +\sum_{\beta\in\Ii_\beta^+\cap\Ii_j}\cov(I_\alpha,I_\beta)\Biggl)
  \end{align*}
  
\end{prop}

\subsection{Exchangeable pairs and switchings}
For our Poisson approximation of cycle counts in the uniform model, we will use
a different form of Stein's known as the method of exchangeable pairs.
 As we lay out the background necessary to apply
Stein's method by exchangeable pairs, we will also explain a connection between this method
and  a combinatorial technique for asymptotic enumeration
called the method of switchings.

The method of switchings,  pioneered
by Brendan McKay and Nicholas Wormald,
has been
applied to asymptotically enumerate combinatorial structures
that defy exact counts, including
Latin rectangles \cite{GM} and matrices with prescribed row and column sums
\cite{Mc84,MW,GMW}.  It has seen its biggest use in analyzing regular graphs;
see \cite{KSVW}, \cite{MWW}, \cite{KSV}, and \cite{BSK} for some examples.
A good summary of switchings in random regular graphs
can be found in Section~2.4 of \cite{Wor99}.
    
    The basic idea of the method is to choose
    two families of objects, $A$ and $B$, and investigate only their
    relative sizes.  To do this, one defines a set
    of switchings that
    connect elements of $A$ to elements of $B$.
    If every element of $A$ is connected to
    roughly $p$ objects in $B$,
    and every element in $B$ is connected to roughly
    $q$ objects in $A$, then 
    by a double-counting argument, $|A|/|B|$
    is approximately $q/p$.  When the objects in question
    are elements of a probability space, this gives an estimate of
    the relative probabilities of two events.

Stein's method (sometimes
called the Stein-Chen method when used for Poisson approximation)
is a powerful and elegant tool to compare two probability
distributions.  It was originally developed by Charles
Stein for normal approximation; its first published
use is \cite{St1}. Louis Chen adapted the method 
for Poisson approximation \cite{Chen1}.  Since then,
Stein, Chen, and a score of others have adapted Stein's method to
a wide variety of circumstances.  The survey
paper \cite{Ross} gives a broad introduction to Stein's method, 
and \cite{CDM} and
\cite{BHJ} focus specifically on using it for Poisson approximation.

We will use the technique of exchangeable pairs,
following the treatment in \cite{CDM}.
Suppose we want to bound
the distance of the law of $X$ from the Poisson distribution.
The technique is to introduce an auxiliary randomization to $X$ to
get a new random variable $X'$ so that $X$ and $X'$ are exchangeable
(that is, $(X,X')$ and $(X', X)$ have the same law).
If $X$ and $X'$ have the right relationship---specifically,
if they behave
like two steps in an immigration-death 
process whose stationary distribution
is Poisson---then Stein's method
gives an easy proof that $X$ is approximately Poisson.

    Switchings and Stein's method have bumped into each other
    several times.
For instance,
both techniques have been used to study Latin rectangles \cite{Stein,GM},
and the analysis of random contingency tables in \cite{DS} is similar
to combinatorial work like \cite{GrM}.
Nevertheless, we believe that this is the first explicit connection
between the two techniques.
The essential idea is to use a random switching as the auxiliary randomization
in constructing an exchangeable pair.  

We believe the connection between switchings and Stein's method
may prove profitable to users of both techniques.  
Using Stein's method in conjunction with a switchings argument
allows for a quantitative bound  on the accuracy of the approximation.
Stein's method
can also be used for approximation
by other distributions besides Poisson, and for proving
concentration bounds  (see \cite{Cha}).
On the other hand, Stein's method cannot prove results as sharp
as \cite[Theorem~2]{MWW}, which gives an extremely accurate bound
on the probability that a random graph
has no cycles of length $r$ or less.
The bare-hands switching arguments used there might
be useful to anyone who needs a particularly sharp bound on
a Poisson approximation at a single point (see \cite[Proposition~1.7]{JohPaq}).

Now, we give the background we need on Stein's method of exchangeable pairs.
Recall that the main idea of Stein's method of exchangeable pairs
is to perturb a random variable $X$ to get a new random variable $X'$,
and then to examine
the relationship between the two.
The basic heuristic is that
if $(X,X')$ is exchangeable and
\begin{align*}
  \P[X'=X+1\mid X]&\approx \frac{\lambda}{c},\\
  \P[X'=X-1\mid X]&\approx \frac{X}{c},
\end{align*}
for some constant $c$, then 
$X$ is approximately Poisson with mean $\lambda$.
(When $X$ and $X'$ are exactly Poisson with mean $\lambda$ and are
two steps in an immigration-death
chain whose stationary distribution is that, these
equations hold exactly.)  The following proposition
gives a precise, multivariate version of this heuristic.
\begin{prop}[{\cite[Proposition~10]{CDM}}]\label{prop:steinsmethod}
  Let $W=(W_1,\ldots,W_r)$ be a random vector taking values
  in $\NN^r$, and let the coordinates of
  $Z=(Z_1,\ldots,Z_r)$ be independent Poisson random variables with
  $\E Z_k=\lambda_k$.  Let $W'=(W_1',\ldots,W_r')$ be defined
  on the same space as $W$, with $(W,W')$ an exchangeable pair.
  
  For any choice of $\sigma$-algebra $\Ff$ with
  respect to which $W$ is measurable and any
  choice of constants $c_k$,
  \begin{align*}
    \dtv(W, Z)\leq
      \sum_{k=1}^r\xi_k\Big(\E\big|\lambda_k-c_k\P[\Delta^+_k\mid\Ff]\big|
      +\E\big| W_k-c_k\P[\Delta^-_k\mid\Ff]  \big|
      \Big),
  \end{align*}
  with $\xi_k=\min(1,1.4\lambda_k^{-1/2})$ and
  \begin{align*}
    \Delta^+_k &= \{W_k'=W_k+1,\ \text{$W_j=W'_j$ for $k<j\leq r$}\},\\
    \Delta^-_k &= \{W_k'=W_k-1,\ \text{$W_j=W'_j$ for $k<j\leq r$}\}.
  \end{align*}
\end{prop}
\begin{rmk}
  We have changed the statement of the proposition from
\cite{CDM} in two small ways:
we condition our probabilities on $\Ff$, rather than on $W$, and
we do not require that $\E W_k =\lambda_k$ (though
the approximation will fail if this is far from true).  Neither
change invalidates the proof of the proposition.
\end{rmk}

\begin{rmk}
There is a direct connection between switchings and a certain bare-hands
version of Stein's method.  Though this is not
what we use in this paper,
it is helpful in understanding why Stein's method
and the method of switchings are so similar.
If $(X,X')$ is exchangeable, then as explained in
\cite[Section~2]{St92},
one can  directly investigate
ratios of probabilities of different values of $X$ using the equation
\begin{align*}
  \frac{\P[X=x_1]}{\P[X=x_2]} &= \frac{\P[X'=x_1\mid X=x_2]}{\P[X'=x_2\mid
  X=x_1]}.
\end{align*}
This technique
bears a strong resemblance to the method of switchings: if we think
of $X$ as some property of a random graph (for example, number of cycles)
and $X'$ as that property after a random switching has been applied, then
this formula instructs us to count how many switchings change $X$
from $x_1$ to $x_2$ and vice versa, just as one does
when using switchings for asymptotic enumeration. 
\end{rmk}

\section{Poisson approximation in the permutation model}

We introduce two lemmas.  The first gives a bound on the distance between
Poisson random variables with almost the same means, and 
the second provides a technical bound that we need.
\begin{lem}\label{lem:YZ}
  Let $\Y=(Y_\alpha,\,\alpha\in\Ii)$ and $\Z=(Z_\alpha,\,\alpha\in\Ii)$
  be vectors of independent Poisson random variables.
  Then
  \begin{align*}
    \dtv(\Y,\,\Z)\leq \sum_{\alpha\in\Ii}|\E Y_\alpha-\E Z_\alpha|.
  \end{align*}
\end{lem}
\begin{proof}
  We will apply the Stein-Chen method directly.
  Define the operator $\Aa$ by
    \begin{align*}
      \Aa h(x) = \sum_{\alpha\in\Ii}\E[Z_\alpha]\big(h(x+e_\alpha)-h(x)\big)
        +\sum_{\alpha\in\Ii}x_\alpha\big(h(x-e_\alpha)-h(x)\big)
    \end{align*}
    for any $h\colon \ZZ_+^{\abs{\Ii}}\to \RR$ and $x\in\ZZ_+^{\abs{\Ii}}$.
    This is the Stein operator for the law of $\Z$, and
    $\E\Aa h(\Z)=0$ for any bounded function $h$.
    By Proposition~10.1.2 and Lemma~10.1.3 in \cite{BHJ},
    for any set $A\subset \ZZ_+^{\abs{\Ii}}$, there is a function
    $h$ such that
    \begin{align*}
        \Aa h(x)=1\{x\in A\}-
        \P[\Z\in A],
      \end{align*}
    and this function has the property that
    \begin{align}
        \sup_{\substack{x\in\ZZ_+^{\abs{\Ii}}\\ \alpha\in\Ii}}|h(x+e_\alpha)-h(x)|\leq 1.
          \label{eq:mvsteinbound1}
      \end{align}
    Thus we can bound the total variation distance between
    the laws of $\Y$ and $\Z$ by bounding $\abs{\E\Aa h(\Y)}$
    over all such functions $h$.
    
    We write $\Aa h(\Y)$ as
    \begin{align*}
      \Aa h(\Y) &= \sum_{\alpha\in\Ii}\E[Y_\alpha]\big(h(\Y+e_\alpha)-h(\Y)\big)
        +\sum_{\alpha\in\Ii}Y_\alpha\big(h(\Y-e_\alpha)-h(x)\big)\\
        &\phantom{=}\quad +\sum_{\alpha\in\Ii}\big(\E Z_\alpha - \E Y_\alpha\big)
          \big(h(\Y+e_\alpha)-h(\Y)\big).
    \end{align*}
    The first two of these sums have expectation zero, so
    \begin{align*}
      \abs{\E\Aa h(\Y)} &\leq \sum_{\alpha\in\Ii}\abs{\E Z_\alpha - \E Y_\alpha}
      \E\abs{h(\Y+e_\alpha)-h(\Y)}.
    \end{align*}
    By \eqref{eq:mvsteinbound1},
      $\abs{h(\Y+e_\alpha)-h(\Y)}\leq 1$,
    which proves the lemma.
\end{proof}

\begin{lem}\label{lem:ipbounds}
  Let $a$ and $b$ be $d$-dimensional vectors with nonnegative integer
  components,
  and let $\ip{a}{b}$ denote the standard Euclidean inner product.
  \begin{align*}
      \prod_{i=1}^d\frac{1}{[n]_{a_i+b_i}} - 
        \prod_{i=1}^d\frac{1}{[n]_{a_i}[n]_{b_i}}
        \leq \frac{\ip{a}{b}}{n}\prod_{i=1}^d\frac{1}{[n]_{a_i+b_i}}
    \end{align*}
\end{lem}
\begin{proof}
    We define a family of independent
    random maps $\sigma_i$ and $\tau_i$
    for $1\leq i\leq d$.
    Choose $\sigma_i$ uniformly from all injective maps from $[a_i]$
    to $[n]$, and choose $\tau_i$ uniformly from all injective maps
    from $[b_i]$ to $[n]$.  Effectively, $\sigma_i$ and $\tau_i$
    are random ordered subsets of $[n]$.
    We say that $\sigma_i$ and $\tau_i$ \textit{clash} if their images
    overlap.
    \begin{align*}
      \P[\text{$\sigma_i$ and $\tau_i$ clash for some $i$}] = 
        1-\prod_{i=1}^d\frac{[n]_{a_i+b_i}}{[n]_{a_i}[n]_{b_i}}.
    \end{align*}
  For any $1\leq i\leq d$, $1\leq j\leq a_i$, and $1\leq k\leq b_i$,
  the probability that $\sigma_i(j)=\tau_i(k)$ is $1/n$.
  By a union bound,
  \begin{align*}
    \P[\text{$\sigma_i$ and $\tau_i$ clash for some $i$}]
      &\leq \sum_{i=1}^d\frac{a_ib_i}{n}=\frac{\ip{a}{b}}{n}.
  \end{align*}
  We finish the proof by dividing
  both sides of this inequality by $\prod_{i=1}^d[n]_{a_i+b_i}$.
\end{proof}

\begin{proof}[Proof of Theorem~\ref{thm:processapprox}]
  We will give the proof in three sections: First, we make the coupling
  and show that it satisfies \eqref{eq:monotone-} and \eqref{eq:monotone+}.
  Next, we apply Proposition~\ref{prop:BHJprop} to approximate $\I$
  by $\Y$, a vector of independent Poissons with $\E Y_\alpha=\E I_\alpha$.
  Last, we approximate $\Y$ by $\Z$ to prove the theorem.
  
  If $d>n^{1/2}$ or $r>n^{1/10}$, then  $c(2d-1)^{2r-1}/n>1$ 
  for a sufficiently large choice of $c$, and the theorem holds
  trivially.  Thus we will assume throughout that $d\leq n^{1/2}$
  and $r\leq n^{1/10}$ (the choice of $1/10$ here is completely arbitrary).
  The expression $O(f(d, r, n))$ should be interpreted as a function
  of $d$, $r$, and $n$
  whose absolute value is bounded by $Cf(d,r,n)$ for some absolute constant
  $C$, for all $d$, $r$, and $n$ satisfying $2\leq d\leq n^{1/2}$
  and $r\leq n^{1/10}$.
  
  \step{1}{Constructing the coupling.}
  
  Fix some $\alpha\in\Ii$.  We will construct a random vector
  $(J_{\beta\alpha},\,\beta\in\Ii)$ distributed as
  $(I_\beta,\,\beta\in\Ii)$ conditioned on $I_\alpha=1$.
  We do this by constructing a random graph $G_n'$
  distributed as $G_n$ conditioned to contain the cycle $\alpha$.
  Once this is done, we will define
  $J_{\beta\alpha}=1\{\text{$G_n'$ contains cycle $\beta$}\}$.

Let $\pi_1,\ldots,\pi_d$ be the random permutations that give rise to $G_n$.
We will alter them to form permutations $\pi_1',\ldots,\pi_d'$, and we will
construct $G_n'$ from these.  Let us first consider what distributions
$\pi_1',\ldots,\pi_d'$ should have.  For example, suppose that
$\alpha$ is the cycle
\begin{center}
  \begin{tikzpicture}[scale=1.4,auto]
    \path (0,0) node (s1) {$1$}
          (1,0) node (s2) {$2$}
          (2,0) node (s3) {$3$}
          (3,0) node (s4) {$4$}
          (4,0) node (s5) {$1.$};
    \draw[->] (s1)-- node {$\pi_3$} (s2);
    \draw[<-] (s2)-- node {$\pi_1$} (s3);
    \draw[->] (s3)-- node {$\pi_3$} (s4);
    \draw[->] (s4)-- node {$\pi_1$} (s5);
  \end{tikzpicture}
\end{center}
Then $\pi_1'$ should be distributed as a uniform random $n$-permutation
conditioned to make $\pi_1'(3)=2$ and $\pi_1'(4)=1$, and
$\pi_3'$ should be distributed as a uniform random $n$-permutation
conditioned to make $\pi_3'(1)=2$ and $\pi_3'(3)=4$, while
$\pi_2'$ should just be a uniform random $n$-permutation.
A random graph constructed from $\pi_1'$, $\pi_2'$, and $\pi_3'$
will be distributed as $G_n$ conditioned to contain $\alpha$.

We now describe the construction of $\pi_1',\ldots,\pi_d'$.
Suppose $\alpha$ is the cycle
  \begin{align}\label{eq:alpha}
  \begin{tikzpicture}[baseline, auto]
    \begin{scope}[anchor=base west]
      \path (0,0) node (s0) {$s_0$}
            (s0.base east)+(1.1,0) node (s1) {$s_1$}
            (s1.base east)+(1.1,0) node (s2) {$s_2$}
            (s2.base east)+(1.1,0) node (s3) {$\cdots$}
            (s3.base east)+(1.1,0) node (s4) {$s_k=s_0,$};
    \end{scope}      
    \foreach\i [remember=\i as \lasti (initially 0)] in {1, ..., 3}
      \draw (s\lasti.mid east)-- node {$w_\i$} (s\i.mid west);
    \draw (s3.mid east)--node {$w_k$} (s4.mid west);
  \end{tikzpicture}
  \end{align}
  with each edge directed according to whether $w_i(s_{i-1})=s_{i}$
  or $w_i(s_{i})=s_{i-1}$.
    Fix some $1\leq l\leq d$, and suppose that the edge-label $\pi_l$
    appears $M$ times in the cycle $\alpha$. Let
    $(a_m,b_m)$ for $1\leq m\leq M$ be these directed edges.
    We must construct $\pi_l'$ to have the uniform
    distribution conditioned on $\pi_l'(a_m)=b_m$
    for  $1\leq m\leq M$.
    
    We define a sequence of random transpositions by the following
    algorithm: Let $\tau_1$ swap $\pi_l(a_1)$ with $b_1$.
    Let $\tau_2$ swap $\tau_1\pi_l(a_2)$ with $b_2$, and so on.
    We then define $\pi_l'=\tau_M\cdots\tau_1\pi_l$.
    This permutation satisfies $\pi_l'(a_m)=b_m$ for $1\leq m\leq M$,
    and it is distributed uniformly, subject to the
    given constraints, which can be 
    proven by induction on each swap.
      We now define $G_n'$ from the permutations $\pi_1',\ldots,\pi_d'$
      in the usual way.  It is defined on the same probability space as
      $G_n$, and it is distributed as $G_n$ conditioned to contain $\alpha$,
      giving us a random vector $(J_{\beta\alpha},\,\beta\in\Ii)$
      coupled with $(I_\beta,\,\beta\in\Ii)$.

      Now, we will give a partition $\Ii^-\cup\Ii^+=\Ii\setminus\{\alpha\}$
    satisfying \eqref{eq:monotone-} and \eqref{eq:monotone+}.
    Suppose that
    $G_n$ contains an edge $\labeledarrow{s_i}{v}{w_{i+1}}$
    with $v\neq s_{i+1}$,
    or an edge $\labeledarrow{v}{s_{i+1}}{w_{i+1}}$ with $v\neq s_i$.
    The graph $G_n'$ cannot contain this edge, since it contains
    $\alpha$.  In fact,
    edges of this form are the \emph{only} ones found in $G_n$ but not
    $G_n'$:
      \begin{lem}\label{lem:coupling}
      Suppose there is an edge 
      \labeledarrow{i}{j}{\pi_l}
      contained in $G_n$ but not in $G_n'$.  Then $\alpha$
      contains either an edge \labeledarrow{i}{v}{\pi_l}
      with $v\neq j$,
      or $\alpha$ contains an edge \labeledarrow{v}{j}{\pi_l} with $v\neq i$.
      \end{lem}
      \begin{proof}
        Suppose
        $\pi_l(i)=j$, but $\pi_l'(i)\neq j$.  Then $j$ must
        have been swapped 
        when making $\pi'_l$, which can happen only if
        $\pi_l(a_m)=j$ or $b_m=j$ for some $m$.  In the first case,
        $a_m=i$ and
        $\alpha$ contains the edge $\labeledarrow{i}{b_m}{\pi_l}$ with
        $b_m\neq j$,
        and in the second $\alpha$ contains the edge
        $\labeledarrow{a_m}{j}{\pi_l}$ with $a_m\neq i$.
    \end{proof}
    Define
      $\Ii_\alpha^-$ as all cycles in $\Ii$ that contain an edge
          $\labeledarrow{s_i}{v}{w_{i+1}}$ with $v\neq s_{i+1}$
          or an edge $\labeledarrow{v}{s_{i+1}}{w_{i+1}}$
          with $v\neq s_i$, and define $\Ii_\alpha^+$
          to be the rest of $\Ii\setminus\{\alpha\}$.
    Since $G_n'$ cannot contain any cycle in $\Ii_\alpha^{-}$,
    we have $J_{\beta\alpha}=0$ for all $\beta\in\Ii_\alpha^-$,
    satisfying \eqref{eq:monotone-}.
    For any $\beta\in\Ii_\alpha^+$, Lemma~\ref{lem:coupling}
    shows that if $\beta$ appears in $G_n$, it must also
    appear in $G_n'$.  Hence $J_{\beta\alpha}\geq I_\beta$, and
    \eqref{eq:monotone+} is satisfied.

      \step{2}{Approximation of $\I$ by $\Y$.}
        
    The conditions of Proposition~\ref{prop:BHJprop} are satisfied, and
    we need only bound the sums in \eqref{eq:BHJ}.
    Let $p_\alpha=\E I_\alpha$, the probability that cycle $\alpha$
    appears in $G_n$.  Recall that this equals
    $\prod_{i=1}^d 1/[n]_{e_i}$, where
    $e_i$ is the number of times $\pi_i$ and $\pi_i^{-1}$ appear in the word
    of $\alpha$.  This means that
    \begin{align}
      \frac{1}{n^k}\leq p_\alpha\leq\frac{1}{[n]_k},\label{eq:pbound}
    \end{align}
    where $k=\abs{\alpha}$, the length of cycle $\alpha$.
    
    We bound the first sum in \eqref{eq:BHJ} by
    \begin{align}
      \sum_{\alpha\in\Ii}p_\alpha^2
      =\sum_{k=1}^r\sum_{\alpha\in\Ii_k}p_\alpha^2
      &\leq \sum_{k=1}^r\sum_{\alpha\in\Ii_k}\frac{1}{[n]_k^2}\nonumber\\
      &= \sum_{k=1}^r \left(\frac{[n]_ka(d,k)}{2k}\right)\left(
        \frac{1}{[n]_k^2}\right)\nonumber\\
      &\leq \sum_{k=1}^r \frac{2d(2d-1)^{k-1}}{2k[n]_k}=
      O\left(\frac{d}{n}\right).\label{eq:psquaredcov} 
    \end{align}

    To bound the second sum in \eqref{eq:BHJ}, 
    we investigate the size of $\Ii_\alpha^-$.
    Suppose that $\alpha\in\Ii_k$, and $\alpha$ has the form
    given in \eqref{eq:alpha}.  Any $\beta\in\Ii_\alpha^-$ must
    contain an edge
    \labeledarrow{s_i}{v}{w_{i+1}} with $v\neq s_{i+1}$,
    or an edge \labeledarrow{v}{s_{i+1}}{w_{i+1}} with $v\neq s_{i}$, and
    there are at most $2k(n-1)$ edges of this form.
    For any given edge, there are at most $[n-2]_{j-2}(2d-1)^{j-1}$
    cycles in $\Ii_j$ that contain that edge, for any $j\geq 2$.
    Thus for any $\alpha\in\Ii_k$, the number of cycles of
    length $j\geq 2$ in $\Ii_\alpha^-$ is at most
    $2k[n-1]_{j-1}(2d-1)^{j-1}$, and this bound
    also holds for $j=1$.
    
    For any $\beta\in\Ii_\alpha^-$, it holds that $\E[I_\alpha I_\beta]=0$,
    so that $\cov(I_\alpha,I_\beta)=-p_\alpha p_\beta$.  Putting this all
    together and applying \eqref{eq:pbound}, we have
    \begin{align}
      \sum_{\alpha\in\Ii}\sum_{\smash{\beta\in\Ii_\alpha^-}}
      \abs{\cov(I_\alpha,I_\beta)}
      &= \sum_{k=1}^r\sum_{\alpha\in\Ii_k}\sum_{j=1}^r\sum_{\beta\in\Ii_\alpha^-\cap\Ii_j}
      p_\alpha p_\beta\nonumber\\
      &\leq \sum_{k=1}^r\abs{\Ii_k}\frac{1}{[n]_k}
        \sum_{j=1}^r\abs{\Ii_\alpha^-\cap\Ii_j}\frac{1}{[n]_j}
        \nonumber\\
      &\leq\sum_{k=1}^r\frac{a(d,k)}{2k}\sum_{j=1}^r\frac{2k(2d-1)^{j-1}}{n}
      \nonumber\\
      &=\sum_{k=1}^ra(d,k)O\left(\frac{(2d-1)^{r-1}}{n}\right)
      =O\left(\frac{(2d-1)^{2r-1}}{n}\right).\label{eq:-cov}
    \end{align}

    The final sum in \eqref{eq:BHJ} is the most difficult
    to bound. We partition\label{pluspartition}
    $\Ii_\alpha^+$  into sets
    $\Ii_\alpha^+=\Ii_\alpha^0\cup\cdots\cup\Ii_\alpha^{\abs{\alpha}-1}$,
    where $\Ii_\alpha^l$ is all cycles in $\Ii_\alpha^+$
    that share exactly $l$ labeled edges with $\alpha$.
    For any $\beta\in\Ii_\alpha^+$,
    \begin{align*}
      \E[I_\alpha I_\beta] = \P[\text{$G$ contains $\alpha$ and $\beta$}]
        =\prod_{i=1}^d\frac{1}{[n]_{e_i}},
    \end{align*}
    where $e_i$ is the number of $\pi_i$-labeled edges
    in $\alpha\cup\beta$.
    Thus for $\beta\in\Ii_\alpha^l$,
    \begin{align}
      \frac{1}{n^{\abs{\alpha}+\abs{\beta}-l}}\leq
      \E[I_\alpha I_\beta]\leq \frac{1}{[n]_{\abs{\alpha}+\abs{\beta}-l}}.
      \label{eq:covbound}
    \end{align}
    
    We start by seeking estimates on the size
    of $\Ii_\alpha^l$ for $l\geq 1$.
    Fix some choice of $l$ edges
    of $\alpha$.  We start by counting the cycles in $\Ii_\alpha^l$
    that share exactly these edges with $\alpha$.
    We illustrate this in Figure~\ref{fig:graphassembly}.
    Call the graph consisting of these edges
    $H$, and suppose that $H$ has $p$ components.
    Since it is a forest, $H$ has $l+p$ vertices.

           \begin{figure}
         \begin{center}
        \begin{tikzpicture}[scale=1.75,vert/.style={circle,fill,inner sep=0,
              minimum size=0.15cm,draw}, H/.style={dashed},>=stealth]
          \begin{scope}
          \node[vert,label=90:{$1$}] (s1) at (270:1cm) {};
            \node[vert,label=57:{$2$}] (s2) at (237:1) {};
            \node[vert,label=18:{$3$},H] (s3) at (205:1) {};
            \node[vert,label=358:{$4$},H] (s4) at (171:1) {};
            \node[vert,label=319:{$5$},H] (s5) at (139:1) {};
            \node[vert,label=278:{$6$}] (s6) at (106:1) {};
            \node[vert,label=260:{$7$},H] (s7) at (74:1) {};
            \node[vert,label=221:{$8$},H] (s8) at (41:1) {};
            \node[vert,label=182:{$9$},H] (s9) at (8:1) {};
            \node[vert,label=158:{$10$},H] (s10) at (335:1) {};
            \node[vert,label=118:{$11$}] (s11) at (302:1) {};
            \draw[thick,->] (s1) to node[auto] {$\pi_1$} (s2);
            \draw[thick,->] (s2) to node[auto] {$\pi_1$} (s3);
            \draw[thick,<-,H] (s3) to node[auto,black] {$\pi_2$} (s4);
            \draw[thick,->,H] (s4) to node[auto,black] {$\pi_3$} (s5);
            \draw[thick,->] (s5) to node[auto] {$\pi_1$} (s6);
            \draw[thick,<-] (s6) to node[auto,black] {$\pi_2$} (s7);
            \draw[thick,<-,H] (s7) to node[auto,black] {$\pi_1$} (s8);
            \draw[thick,<-] (s8) to node[auto,black] {$\pi_2$} (s9);
            \draw[thick,->,H] (s9) to node[auto,black] {$\pi_1$} (s10);
            \draw[thick,->] (s10) to node[auto,black] {$\pi_3$} (s11);
            \draw[thick,->] (s11) to node[auto,black] {$\pi_3$} (s1);
            \draw (0,-1.4) node[text width=5.8cm,anchor=base,execute at begin node=\setlength{\baselineskip}{13pt}]
            {
              The cycle $\alpha$, with $H$ dashed.  The subgraph
              $H$ has components $A_1,\ldots,A_p$.  In this example,
              the number of components of $H$ is $p=3$, 
              the size of $\alpha$ is $k=11$, and the number
              of edges in $H$ is $l=4$.\\[0.1cm]
              In this example, we will construct a cycle $\beta$ of length
              $j=10$
              that overlaps with $\alpha$ at $H$.
            };
          \end{scope}
          \begin{scope}[xshift=3.7cm]
            \path (-1.476,0) node[vert,label=below:{$3$}] (s3) {}
                  -- ++(0.563,0) node[vert,label=below:$4$] (s4) {}
                  -- ++(0.563,0) node[vert,label=below:$5$] (s5) {}
                  -- ++(0.35,0) node[vert,label=below:$10$] (s10) {}
                  -- ++(0.563,0) node[vert,label=below:$9$] (s9){}
                  -- ++(0.35,0) node[vert,label=below:$7$] (s7) {}
                  -- ++(0.563,0) node[vert,label=below:$8$] (s8){};
            \draw[thick,<-] (s3) to node[auto] {$\pi_2$} (s4);
            \draw[thick,->] (s4) to node[auto] {$\pi_3$}(s5);
            \draw[thick,->] (s9) to node[auto,swap] {$\pi_1$}(s10);
            \draw[thick,<-] (s7) to node[auto] {$\pi_1$}(s8);
            \draw (0,-1.4) node[text width=5.8cm,anchor=base,execute at begin node=\setlength{\baselineskip}{13pt}]
            {
              \textbf{Step 1.} We lay out the components $A_1,\ldots,A_p$.
              We can order and orient $A_2,\ldots,A_p$ however we would like,
              for a total of $(p-1)!2^{p-1}$ choices.
              Here, we have ordered the components $A_1, A_3, A_2$, 
              and we have reversed the orientation of $A_3$.
            };
          \end{scope}
          \begin{scope}[yshift=-5cm,new/.style={}]
            \path (-1.476,0) node[vert,label=below:{$3$}] (s3) {}
                  -- ++(0.563,0) node[vert,label=below:$4$] (s4) {}
                  -- ++(0.563,0) node[vert,label=below:$5$] (s5) {}
                  -- ++(0.35,0) node[vert,label=below:$10$] (s10) {}
                  -- ++(0.563,0) node[vert,label=below:$9$] (s9){}
                  -- +(-0.1065,0.553) node[vert,new] (a1){}
                  -- +(0.4565,0.553) node[vert,new] (a2) {}
                  -- ++(0.35,0) node[vert,label=below:$7$] (s7) {}
                  -- ++(0.563,0) node[vert,label=below:$8$] (s8){}
                  -- (0,-0.5) node[vert,new] (a3) {};
            \draw[thick,<-] (s3) to node[auto] {$\pi_2$} (s4);
            \draw[thick,->] (s4) to node[auto] {$\pi_3$}(s5);
            \draw[thick,new] (s5) to [out=90,in=90] (s10);
            \draw[thick,->] (s9) to node[auto,swap] {$\pi_1$}(s10);
            \draw[thick,new] (s9) to (a1);
            \draw[thick,new] (a1) to (a2);
            \draw[thick,new] (a2) to (s7);
            \draw[thick,<-] (s7) to node[auto] {$\pi_1$}(s8);
            \draw[thick,new] (s8) to[out=225,in=0] (a3);
            \draw[thick,new] (a3) to[out=180,in=305] (s3);
            
            \draw (0,-0.95) node[anchor=base,text width=5.8cm,execute at begin node=\setlength{\baselineskip}{13pt}]
            {
              \textbf{Step 2.} Next, we choose how many edges will go in each
              gap between components.  Each gap must contain at least
              one edge, and we must add a total of $j-l$ edges, giving
              us $\binom{j-l-1}{p-1}$ choices.
              In this example, we have added one edge after $A_1$,
              three after $A_3$, and two after $A_2$.
            };
          \end{scope}
          \begin{scope}[yshift=-5cm,xshift=3.7cm,new/.style={}]
            \path (-1.476,0) node[vert,label=below:{$3$}] (s3) {}
                  -- ++(0.563,0) node[vert,label=below:$4$] (s4) {}
                  -- ++(0.563,0) node[vert,label=below:$5$] (s5) {}
                  -- ++(0.35,0) node[vert,label=below:$10$] (s10) {}
                  -- ++(0.563,0) node[vert,label=below:$9$] (s9){}
                  -- +(-0.1065,0.553) node[vert,label={[new]
                                       above left:$23$}] (a1){}
                  -- +(0.4565,0.553) node[vert,
                            label={[new]above right:$1$}] (a2) {}
                  -- ++(0.35,0) node[vert,label=below:$7$] (s7) {}
                  -- ++(0.563,0) node[vert,label=below:$8$] (s8){}
                  -- (0,-0.5) node[vert,label={[new]below:$15$}] (a3) {};
            \draw[thick,<-] (s3) to node[auto] {$\pi_2$} (s4);
            \draw[thick,->] (s4) to node[auto] {$\pi_3$}(s5);
            \draw[thick] (s5) to [out=90,in=90] node[auto,new]
            {$\pi_1$} (s10);
            \draw[thick,->] (s9) to node[auto,swap] {$\pi_1$}(s10);
              \draw[thick,->] (s9) to (a1);
            
            \draw (s9) +(-0.21,0.32) node[new] {$\pi_2$};
            \draw[thick,->] 
                (a1) to node[auto,new] {$\pi_3$} (a2);
            \draw[thick,<-] (a2) to (s7);
            \draw (s7) +(0.21,0.32) node[new] {$\pi_2$};
            \draw[thick,<-] (s7) to node[auto] {$\pi_1$}(s8);
            \draw[thick,<-] (s8) to[out=225,in=0] node[auto,new] {$\pi_2$}(a3);
            \draw[thick,<-] (a3) to[out=180,in=305] node[auto,new]{$\pi_1$} (s3);
            
            \draw (0,-0.95) node[anchor=base, text width=5.8cm,execute at begin node=\setlength{\baselineskip}{13pt}]
            {
              \textbf{Step 3.} We can choose the new vertices in
              $[n-p-l]_{j-p-l}$ ways, and we can direct and give labels
              to the new edges in at most $(2d-1)^{j-l}$ ways.
            };
          \end{scope}
        \end{tikzpicture}          
         \end{center}
         \caption{Assembling an element $\beta\in\Ii_\alpha^l$ that
         overlaps with $\alpha$ at a given subgraph $H$.}
         \label{fig:graphassembly}
       \end{figure}   
       Let $A_1,\ldots, A_p$ be the components of $H$.  We can
       assemble any  element $\beta\in\Ii_\alpha^l$ 
       that overlaps with $\alpha$ in $H$
       by stringing together these components in some order, with
       other edges in between.
       Each component can appear in
       $\beta$ in one of two orientations.  Since the vertices in $\beta$
       have no fixed ordering, we can
       assume without
       loss of generality that
       $\beta$ begins with component $A_1$ with a fixed orientation.
       This leaves $(p-1)!2^{p-1}$ choices for the order
       and orientation of $A_2,\ldots,A_p$ in $\beta$.
       
       Imagine now the components laid out in a line,
       with gaps between them, and count the number
       of ways to fill the gaps.  Suppose that $\beta$
       is to have length $j$.
       Each of the
       $p$ gaps
       must contain at least one edge, and the total number
       of edges in all the gaps is $j-l$.
       Thus the total number of possible gap sizes
       is the number of compositions of $j-l$ into $p$ parts,
       or $\binom{j-l-1}{p-1}$.
       
       Now that we have chosen the number of edges
       to appear in each gap, we choose the edges themselves.
       We can do this by giving
       an ordered list $j-p-l$ vertices to go in the gaps,
       along with a label and an orientation
       for each of the $j-l$ edges this gives.
       There are $[n-p-l]_{j-p-l}$ ways to choose the vertices.
       We can give each new edge any orientation and label subject
       to the constraint that the word of the cycle we
       construct must be reduced.
       This means we have at most $2d-1$ choices for the orientation
       and label of each new edge, for a total of at most $(2d-1)^{j-i}$.

       All together, there are at most
       $(p-1)!2^{p-1}\binom{j-l-1}{p-1}[n-p-l]_{j-p-l}(2d-1)^{j-l}$
       elements of $\Ii_j$ that overlap with the cycle $\alpha$ at the
       subgraph $H$.  We now calculate the number of different ways to choose
       a subgraph $H$ of $\alpha$ with $l$ edges and $p$ components.
       Suppose $\alpha$ is given as in \eqref{eq:alpha}.
       We first choose a vertex $s_{i_0}$.  Then, we can
       specify which edges to include in $H$ by giving a sequence
       $a_1,b_1,\ldots,a_p,b_p$ instructing us
       to include in $H$ the first $a_1$ edges after $s_{i_0}$,
       then to exclude the next $b_1$, then to include the next $a_2$,
       and so on.  Any sequence for which $a_i$ and $b_i$
       are positive integers, $a_1+\cdots+ a_p=l$,
       and $b_1+\cdots+b_p=k-l$ gives us a valid choice of $l$ edges of $\alpha$
       making up $p$ components.  This counts each subgraph $H$ a total of
       $p$ times, since we could begin with any component of $H$.
       Hence the number of subgraphs $H$ with $l$ edges
       and $p$ components is $(k/p)\binom{l-1}{p-1}\binom{k-l-1}{p-1}$.
       This gives us the bound
       \begin{align*}
         |\Ii_\alpha^l\cap\Ii_j|&\leq \sum_{p=1}^{l\wedge (j-l)}
              (k/p)\binom{l-1}{p-1}\binom{k-l-1}{p-1}
           (p-1)! \;\times \\
           &\qquad\qquad 2^{p-1}\binom{j-l-1}{p-1}[n-p-l]_{j-p-l}(2d-1)^{j-l}.
       \end{align*}
       We apply the bounds
       \begin{align*}
         \binom{l-1}{p-1}&\leq \frac{r^{p-1}}{(p-1)!},\\
         \binom{k-l-1}{p-1},\,\binom{j-l-1}{p-1}&\leq (er/(p-1))^{p-1},
       \end{align*}
       to get
       \begin{align*}
         |\Ii_\alpha^l\cap\Ii_j|
           &\leq k(2d-1)^{j-l}[n-1-l]_{j-1-l}\left(1+
              \sum_{p=2}^{i\wedge (k-i)}\frac{1}{p}
              \left(\frac{2e^2r^3}{(p-1)^2}\right)^{p-1}
              \frac{1}{[n-1-l]_{p-1}}\right).
       \end{align*}
       Since $r\leq n^{1/10}$,
       the sum in the above equation is bounded by an absolute constant.
       Applying this bound and \eqref{eq:covbound}, 
       for any $\alpha\in\Ii_k$ and $l\geq 1$,
       \begin{align}
         \sum_{\beta\in\Ii_\alpha^l}
         \cov(I_\alpha,I_\beta)
           &\leq \sum_{j=l+1\vphantom{\Ii_\alpha^l}}^r
             \sum_{\beta\in\Ii_\alpha^l\cap\Ii_j}
             \frac{1}{[n]_{k+j-l}}\label{eq:covstartsum}\\
           &\leq \sum_{j=l+1}^rO\left(\frac{k(2d-1)^{j-l}}{n^{k+1}}\right)\nonumber\\
           &=O\left(\frac{k(2d-1)^{r-l}}{n^{k+1}}\right).\nonumber
       \end{align}
       Therefore
       \begin{align}
         \sum_{\alpha\in\Ii}\sum_{l\geq 1}
         \sum_{\beta\in\Ii_\alpha^l}\cov(I_\alpha,I_\beta)
         &= \sum_{k=1}^r\sum_{\alpha\in\Ii_k}\sum_{l=1}^{k-1}
         \sum_{\beta\in\Ii_\alpha^l}\cov(I_\alpha,I_\beta)\nonumber\\
         &\leq  \sum_{k=1}^r\sum_{\alpha\in\Ii_k}\sum_{l=1}^{k-1}
          O\left(\frac{k(2d-1)^{r-l}}{n^{k+1}}\right)\nonumber\\
         &= \sum_{k=1}^r\frac{[n]_ka(d,k)}{2k}
         O\left(\frac{k(2d-1)^{r-1}}{n^{k+1}}\right)\nonumber\\
         &= \sum_{k=1}^r O\left(\frac{(2d-1)^{r+k-1}}{n}\right)\nonumber\\
         &= O\left(\frac{(2d-1)^{2r-1}}{n}\right).\label{eq:+cov}
       \end{align}
         
       Last, we must bound $\sum_{\alpha\in\Ii}\sum_{\beta\in\Ii_\alpha^0}\cov(I_\alpha,
       I_\beta)$.
       For any word $w$, let 
       $e^w_i$ be the number
       of appearances of $\pi_i$ and $\pi_i^{-1}$ in $w$.
       Let $\alpha$ and $\beta$ be cycles with words $w$
       and $u$ respectively, and let $k=\abs{\alpha}$ and $j=\abs{\beta}$.
       Suppose that $\beta\in \Ii^0_\alpha$.  Then
    \begin{align*}
      \cov(I_\alpha,I_\beta)&= \prod_{i=1}^d\frac{1}{[n]_{e^w_i+e^u_i}} - 
        \prod_{i=1}^d\frac{1}{[n]_{e^w_i}[n]_{e^u_i}}\\
        &\leq \frac{\ip{e^w}{e^u}}{n}\prod_{i=1}^d\frac{1}{[n]_{e^w_i+e^u_i}}
        \leq \frac{\ip{e^w}{e^u}}{n[n]_{k+j}}
    \end{align*} 
    by Lemma~\ref{lem:ipbounds}.
    For any pair of words $w\in\Ww_k$ and $u\in\Ww_j$, there are
    at most
    $[n]_{k}[n]_j$ pairs of cycles $\alpha,\beta\in\Ii$ with words
    $w$ and $u$, respectively.  Enumerating over all
    $w\in\Ww_k$ and $u\in\Ww_j$, we count each pair
    of cycles $\alpha,\beta$ exactly $4kj$ times.  Thus
    \begin{align*}
      \sum_{\alpha\in\Ii_k}\sum_{\beta\in\Ii_\alpha^0\cap\Ii_j}\cov(I_\alpha,I_\beta)
      &\leq \frac{[n]_{k}[n]_j}{4kjn[n]_{k+j}}\sum_{w\in\Ww_k}\sum_{u\in\Ww_j}\ip{e^w}
        {e^u}\\
      &\leq \frac{1+O(r^2/n)}
        {4kjn}\ip{\sum_{w\in\Ww_k}e^w}{\sum_{u\in\Ww_j}e^u}.
    \end{align*}
    The vector $\sum_{w\in\Ww_k}e^w$ has every entry equal by symmetry,
    as does $\sum_{u\in\Ww_j}e^u$.  Thus each entry of
    $\sum_{w\in\Ww_k}e^w$ is $ka(d,k)/d$, and each entry
    of $\sum_{u\in\Ww_j}e^u$
    is $ja(d,j)/d$.  The inner product in the above equation
    comes to $kja(d,k)a(d,j)/d$, giving us
    \begin{align}
      \sum_{\alpha\in\Ii_k}\sum_{\beta\in\Ii_\alpha^0\cap\Ii_j}\cov(I_\alpha,I_\beta)
      &\leq \frac{a(d,k)a(d,j)(1+O(r^2/n))}{4dn}\nonumber\\
      &= O\left(\frac{(2d-1)^{j+k-1}}{n}\right).\label{eq:jk0cov}
    \end{align}
    Summing over all $1\leq k,j\leq r$,
    \begin{align}
      \sum_{\alpha\in\Ii}\sum_{\beta\in\Ii^0_\alpha}\cov(I_\alpha,I_\beta)
      &= \left(\frac{(2d-1)^{2r-1}}{n}\right).\label{eq:0cov}
    \end{align}
  
       We can now combine equations \eqref{eq:psquaredcov}, \eqref{eq:-cov},
       \eqref{eq:+cov}, and \eqref{eq:0cov} with 
       Proposition~\ref{prop:BHJprop} to show that
       \begin{align}
         \dtv(\I,\,\Y)=O\left(\frac{(2d-1)^{2r-1}}{n}\right).\label{eq:IY}
       \end{align}
       \step{3}{Approximation of $\Y$ by $\Z$.}
               
       By Lemma~\ref{lem:YZ} and \eqref{eq:pbound},
       \begin{align*}
         \dtv(\Y,\,\Z)\leq\sum_{\alpha\in\Ii}\abs{\E Y_\alpha-\E Z_\alpha}
         &\leq \sum_{k=1}^r\sum_{\alpha\in\Ii_k}\left(
         \frac{1}{[n]_k}-\frac{1}{n^k}\right)\\
         &=\sum_{k=1}^r\frac{a(d,k)}{2k}\left(1-\frac{[n]_k}{n^k}\right).
       \end{align*}
       Since $[n]_k\geq n^k(1-k^2/2n)$,
       \begin{align}
         \dtv(\Y,\,\Z)&\leq\sum_{k=1}^r\frac{a(d,k)k}{4n}
         =O\left(\frac{r(2d-1)^r}{n}\right).\label{eq:YZdist}
       \end{align}
       Together with \eqref{eq:IY}, this bounds the total variation
       distance between the laws of $\I$ and $\Z$ and proves the theorem.
\end{proof}

\begin{proof}[Proof of]
  Consider the partition $\Ii=\bigcup_{k=1}^r\Ii_k$, and define $W_k$ and $Y_k$
  as in the statement of Proposition~\ref{prop:BHJbin}.
  As in the proof of Theorem~\ref{thm:processapprox}, we may assume that
  $d\leq n^{1/2}$ and $n\leq n^{1/10}$. With these restrictions, we have
  \begin{align*}
    \log^+\max\lambda_j &= O\bigl(r\log(2d-1)\bigr),\\
    \lambda_k^{-1} &= O\biggl(\frac{k}{(2d-1)^k}\biggr),\\
    (\lambda_j\lambda_k)^{-1/2} &= O\biggl(\frac{\sqrt{jk}}{(2d-1)^{(j+k)/2}}\biggr).
  \end{align*}
  
  We have already bounded all the terms in \eqref{eq:BHJbin} in the previous proof.
  From \eqref{eq:psquaredcov},
  \begin{align*}
    \sum_{k=1}^r\sum_{\alpha\in\Ii_k}\frac{p_\alpha^2}{\lambda_k} &= O\Bigl(\frac{d}{n}\Bigr).
  \end{align*}
  From \eqref{eq:-cov},
  \begin{align}
    \sum_{\alpha\in\Ii_k}\sum_{\beta\in\Ii_\alpha^-\cap\Ii_j}\abs{\cov(I_\alpha,I_\beta)}
      &= O\biggl(\frac{(2d-1)^{j+k-1}}{n}\biggr).\label{eq:jk-cov}
  \end{align}
  Recalling the partition of $\Ii_\alpha^+$ on p.~\pageref{pluspartition},
  and following \eqref{eq:covstartsum}, for any $\alpha\in\Ii_k$ and $l\geq 1$,
  \begin{align*}
    \sum_{\beta\in\Ii_\alpha^l\cap\Ii_j} \cov(I_\alpha,I_\beta)&= O\biggl(\frac{k(2d-1)^{j-l}}{n^{k+1}}\biggr),
  \end{align*}
  and
  \begin{align*}
    \sum_{\alpha\in\Ii_k}\sum_{l\geq 1}\sum_{\beta\in \Ii_\alpha^l\cap\Ii_j}\cov(I_\alpha, I_\beta)
    &=\sum_{\alpha\in\Ii_k}O\biggl(\frac{k(2d-1)^{j-1}}{n^{k+1}}\biggr)=O\biggl(\frac{(2d-1)^{j+k-1}}{n}
    \biggr).
  \end{align*}
  Together with \eqref{eq:jk0cov}, this shows that
  \begin{align*}
    \sum_{\alpha\in\Ii_k}\sum_{\beta\in\Ii_\alpha^+\cap\Ii_j}\cov(I_\alpha,I_\beta)&= O\biggl(\frac{(2d-1)^{j+k-1}}{n}\biggr).
  \end{align*}
  This and \eqref{eq:jk-cov} prove that
  \begin{align*}
    A(j,k)&=O\biggl(\frac{(2d-1)^{j+k-1}}{n}\biggr).
  \end{align*}
  Now, we apply Proposition~\ref{prop:BHJbin}:
  \begin{align*}
    \dtv\bigl((W_1,&\ldots,W_r),\;(Y_1,\ldots,Y_r)\bigr) = O\biggl(
    \frac{r^2(2d-1)^{r-1}\log(2d-1)}{n}\biggr).
  \end{align*}
  Last, we apply \eqref{eq:YZdist} to bound the distance between
  $(Y_1,\ldots,Y_r)$ and $(Z_1,\ldots,Z_r)$ and complete the proof.
\end{proof}

\section{Poisson approximation in the uniform model}

\subsection{Preliminaries}
For vertices $u$ and
$v$ in a graph, we will use the notation
$u\sim v$ to denote that the edge $uv$ exists.  The distance between
two vertices is the length of the shortest path between them,
and the distance between two edges or sets of vertices is the shortest
distance between a vertex in one set and a vertex in the other.

Here and throughout, we will use $\Cr{50},\Cr{51},\ldots$
to denote absolute constants whose values are unimportant to us.
\begin{prop}\label{prop:McKayestimate}
  Let $G$ be a random $d$-regular graph on $n$ vertices,
  with $d\leq n^{1/3}$.
  \begin{enumerate}[(a)]
    \item\label{item:onecycle}
      Let $\alpha$ be a cycle of length $k\leq n^{1/10}$ in the complete graph
      $K_n$.  Then
      \begin{align*}
        \P[\alpha\subset G]&\leq \frac{\Cl{50}(d-1)^k}{n^k}.
      \end{align*}
    \item \label{item:twocycles}
      Let $\beta$ be another cycle in $K_n$ of length $j\leq n^{1/10}$,
      and suppose that $\alpha$ and $\beta$ share $f$ edges.
      Then
      \begin{align*}
        \P[\alpha\cup\beta\subset G]&\leq \frac{\Cl{51}(d-1)^{j+k-f}}{n^{j+k-f}}.
      \end{align*}
    \item \label{item:attachedcycles}
      Let $H$ be a subgraph of $K_n$ consisting of a $j$-cycle
      and a $k$-cycle joined by path of length $l$, as in
      Figure~\ref{fig:attachedcycles}.  Suppose that
      $j,k,l\leq n^{1/10}$.  Then
      \begin{align*}
        \P[H\subset G]&\leq\frac{\Cl{60}(d-1)^{j+k+l}}{n^{j+k+l}}.
      \end{align*}
  \end{enumerate}
\end{prop}
\begin{figure}
  \begin{center}
  \begin{tikzpicture}
    \foreach \x in {0,90,180,270}
      \node[vert] (a\x) at (\x:1) {};
    \draw[thick] (a0)--(a90)--(a180)--(a270)--(a0);
    \foreach \x in {0, 72, ..., 288}
      \node[vert,shift={(4,0)}] (b\x) at (\x:1) {};
    \draw[thick] (b0)--(b72)--(b144)--(b216)--(b288)--(b0);
    \path (1.72,0) node[vert] (c1){} (2.44,0) node[vert] (c2){};
    \draw[thick] (a0)--(c1)--(c2)--(b144);
  \end{tikzpicture}
  \end{center}
  \caption{A $4$-cycle and a $5$-cycle, connected by a path
  of length $3$.}\label{fig:attachedcycles}
\end{figure}
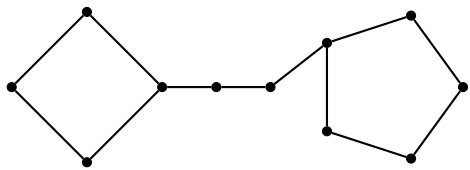

\begin{proof}
  These statements all follow directly from Theorem~3a in \cite{MWW}.
\end{proof}

\subsection{Counting switchings}
We will follow \cite{MWW}, defining and counting switchings.
After this, we will break with that paper by using the switchings to apply Stein's method.
Besides some small notational differences,
the definitions will be the same as those in \cite{MWW}.
To avoid repetition
 of the phrase ``cycles of length $r$ or less,'' we
 will refer to such cycles as \emph{short}.

 \begin{figure}
  \begin{center}
    \begin{tikzpicture}[scale=1.2]
      \begin{scope}[bend angle=20]
      \fill[black!10, rounded corners] (-0.6,-0.5) rectangle (3.6, 1.7);
      \draw
            (0,1.2) node[vert,label=above:$v_0$] (v0) {}
            ++(1,0) node[vert,label=above:$v_1$] (v1) {}
            ++(1,0) node[vert,label=above:$v_2$] (v2) {}
            ++(1,0) node[vert,label=above:$v_3$] (v3) {}
            (-0.15,0) node[vert,label={[xshift=-0.09cm]below:$u_0$}] (u0) {}
            (0.15,0) node[vert,label={[xshift=0.09cm]below:$w_0$}] (w0) {}
            (0.85,0) node[vert,label={[xshift=-0.09cm]below:$u_1$}] (u1) {}
            (1.15,0) node[vert,label={[xshift=0.09cm]below:$w_1$}] (w1) {}
            (1.85,0) node[vert,label={[xshift=-0.09cm]below:$u_2$}] (u2) {}
            (2.15,0) node[vert,label={[xshift=0.09cm]below:$w_2$}] (w2) {}
            (2.85,0) node[vert,label={[xshift=-0.09cm]below:$u_3$}] (u3) {}
            (3.15,0) node[vert,label={[xshift=0.09cm]below:$w_3$}] (w3) {};
        \draw[thick] (v0)--(v1)--(v2)--(v3);
        \draw[thick, bend right] (v0) to (v3);   
        \draw[thick] (w0)--(u1) (w1)--(u2) (w2)--(u3) (u0) to[bend left] (w3);  
      \end{scope}
      \draw[<->, decorate, decoration={snake,amplitude=.4mm,
            segment length=2mm,post length=1.5mm,pre length=1.5mm}] 
            (3.7, 0.6)--(5.5,0.6);
      \begin{scope}[xshift=6.2cm]
        \fill[black!10, rounded corners] (-0.6,-0.5) rectangle (3.6, 1.7);
      \draw (0,1.2) node[vert,label=above:$v_0$] (v0p) {}
            ++(1,0) node[vert,label=above:$v_1$] (v1p) {}
            ++(1,0) node[vert,label=above:$v_2$] (v2p) {}
            ++(1,0) node[vert,label=above:$v_3$] (v3p) {}
            (-0.15,0) node[vert,label={[xshift=-0.09cm]below:$u_0$}] (u0p) {}
            (0.15,0) node[vert,label={[xshift=0.09cm]below:$w_0$}] (w0p) {}
            (0.85,0) node[vert,label={[xshift=-0.09cm]below:$u_1$}] (u1p) {}
            (1.15,0) node[vert,label={[xshift=0.09cm]below:$w_1$}] (w1p) {}
            (1.85,0) node[vert,label={[xshift=-0.09cm]below:$u_2$}] (u2p) {}
            (2.15,0) node[vert,label={[xshift=0.09cm]below:$w_2$}] (w2p) {}
            (2.85,0) node[vert,label={[xshift=-0.09cm]below:$u_3$}] (u3p) {}
            (3.15,0) node[vert,label={[xshift=0.09cm]below:$w_3$}] (w3p) {};
        \draw[thick] (u0p)--(v0p)--(w0p)
                     (u1p)--(v1p)--(w1p)
                     (u2p)--(v2p)--(w2p)
                     (u3p)--(v3p)--(w3p);
      \end{scope}
    \end{tikzpicture}
  \end{center}
  \caption{The change from left to right
  is a \emph{forward switching}, and from right to left is
  a \emph{backward switching}.}
  \label{fig:switchings}
\end{figure}
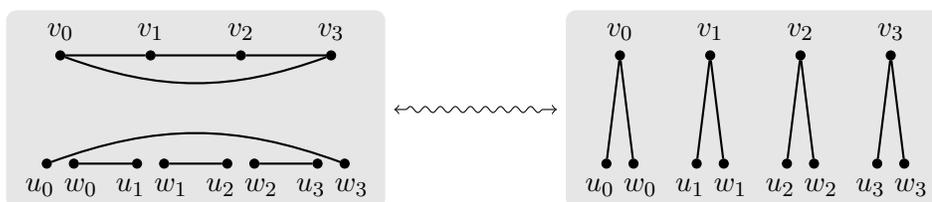

Let $G$ be a $d$-regular graph.
Suppose that $\alpha=v_0\cdots v_{k-1}$ is a cycle in $G$, and
let $e_i=v_iv_{i+1}$, interpreting
all indices modulo $k$ from now on.  Let $e_i'=w_iu_{i+1}$ for $0\leq i\leq k-1$
be oriented edges such that neither $u_i$ nor $w_i$ is adjacent to
$v_i$.  Consider the act of deleting these $2k$ edges and replacing them
with
the edges $v_iu_i$ and $v_iw_i$ for $0\leq i\leq k-1$
to obtain a new $d$-regular graph $G'$ with the cycle $\alpha$
deleted
(see Figure~\ref{fig:switchings}).  We call this action
induced given by the sequences $(v_i)$, $(u_i)$, and $(w_i)$
a \emph{forward $\alpha$-switching}.   We will consider forward 
$\alpha$-switchings
only up to cyclic rotation of indices; that is, we identify the $2k$
different $\alpha$-switchings obtained by cyclically rotating
all sequences $v_i$, $u_i$, and $w_i$.

To go the opposite direction, suppose $G$ contains oriented paths 
$u_iv_iw_i$ for $0\leq i\leq k-1$
such that $v_i\not\sim v_{i+1}$ and $w_i\not\sim u_{i+1}$.
Consider the act of deleting all edges $u_iv_i$ and $v_iw_i$
and replacing them with $v_iv_{i+1}$ and $w_iu_{i+1}$ for all
$0\leq i\leq k-1$ to create a new graph $G'$ that contains
the cycle $\alpha=v_0\cdots v_{k-1}$.  We call this a \emph{backwards
$\alpha$-switching}.
  Again, we consider switchings
only up to cyclic rotation of all indices.

We call an $\alpha$-switching \emph{valid} if  $\alpha$
 is the
only short cycle created or destroyed by the switching.
For each valid forward $\alpha$-switching taking $G$ to
$G'$, there is a corresponding valid backwards $\alpha$-switching
taking $G'$  to $G$.
Let $F_\alpha$ and $B_\alpha$
be the number of valid forward and backwards $\alpha$-switchings,
respectively, on some graph $G$.  
Using arguments drawn from \cite[Lemma~3]{MWW}, we give
some estimates on them.
\begin{lem}\label{lem:forwardswitchings}
  Let $G$ be a deterministic $d$-regular graph on $n$ vertices
  with cycle counts $\{\Cy{k},\,k\geq 3\}$.
  For any short cycle $\alpha\subset G$ of length $k$,  
  \begin{align}
    F_\alpha\leq [n]_kd^k.\label{eq:fub}
  \end{align} 
  If $\alpha$ does not
  share an edge with another short cycle,
  \begin{align}
    F_\alpha&\geq
      [n]_kd^k\left(1- \frac{2k\sum_{j=3}^rj\Cy{j}+\Cl{4}k(d-1)^r}{nd}\right).
      \label{eq:flb}
   \end{align}
\end{lem}
\begin{proof}
  The question is, with $\alpha=v_0\cdots v_{k-1}$
  and $e_i=v_iv_{i+1}$ given, how many ways are there to choose
  $e_0',\ldots,e_{k-1}'$ that give
  a valid switching?
  There are at most $[n]_kd^k$ choices of oriented
  edges $e_0',\ldots,e_{k-1}'$, which proves the upper bound \eqref{eq:fub}.
  For the lower bound, we demonstrate a procedure to choose
  these edges that is guaranteed to give us a valid forward
  $\alpha$-switching.  
  Suppose that
  $e_0',\ldots,e_{k-1}'$ satisfy
  \begin{enumerate}[(a)]
    \item $e_i'$ is not contained in any short cycle;
      \label{li:a}
    \item the distance from $e_i$ to $e_i'$ is at least $r$;
      \label{li:b}
    \item the distance from $e_i'$ to $e_{i'}'$ is at least $r/2$;
      \label{li:c}
    \item the distance from $w_i$ to $u_i$ is at least $r$.
      \label{li:d}
  \end{enumerate}
  Then the switching is valid by an argument identical to the one in \cite{MWW},
  which we will reproduce for convenience.
  By (\ref{li:b}), for all $i$, neither $u_i$ nor $w_i$ is adjacent
  to $v_i$ (or to $v_{i'}$ for any $i'$), as required in the definition
  of a switching.  Let $G'$
  be the graph obtained by applying the switching.
  We need to check now that the switching is valid; that is,
  the only short cycle it creates or destroys is $\alpha$.
  
  Since $\alpha$ shares no edges with other short cycles, its deletion
  does not destroy any other short cycles.  Condition (\ref{li:a}) ensures
  that no short cycles are destroyed by removing $e_0',\ldots,e_{k-1}'$.
  The switching does not create any short cycles either:
  Suppose otherwise, and let $\beta$
  be the new cycle in $G'$.  It consists of paths in $G\cap G'$,
  separated by new edges in $G'$.
  Any such path in $G\cap G'$ must have length at least $r/2$, because
  \begin{itemize}
    \item if it starts and ends in $\alpha$ and has length less than $r/2$,
      then combining this path with a path in $\alpha$
      gives an short cycle in $G$ that overlaps with $\alpha$;
    \item if it starts in $\alpha$ and finishes in  $W=\{u_0,w_0,\ldots,u_{k-1},
      w_{k-1}\}$ and has length less than $r/2$, then combining this path
      with a path in $\alpha$ gives a path violating condition (\ref{li:b});
    \item if it starts at some $e_i'$ and ends at $e_{i'}'$ then
      it must have length $r/2$ by (\ref{li:c}) if $i'\neq i$, and
      by (\ref{li:a}) if $i'=i$.
  \end{itemize}
  Thus $\beta$ contains exactly one path in $G\cap G'$.  The remainder of
  $\beta$ must be an edge $u_iv_i$ or $w_iv_i$, impossible by (\ref{li:b}),
  or a path $u_iv_iw_i$, impossible by (\ref{li:d}).
  
  Now, we find the number of switchings that satisfy conditions
  (\ref{li:a})--(\ref{li:d}) to get a lower bound on $F_\alpha$.
  We will do this by bounding from above the number of switchings
  out of the $[n]_kd^k$ counted in \eqref{eq:fub}
  that fail each condition (\ref{li:a})--(\ref{li:d}).
  \begin{itemize}
    \item
      There are a total of $\sum_{j=3}^rj\Cy{j}$ edges in short
      cycles in $G$.  Choosing one of the edges $e_0',\ldots,e_{k-1}'$
      from these and the rest arbitrarily, there are at most
      $[n-1]_{k-1}d^{k-1}k\sum_{j=3}^r2j\Cy{j}$ switchings 
      that fail condition (\ref{li:a}).
    \item
      The number of edges of distance less than $r$ from some edge is at most
      $2\sum_{j=0}^r(d-1)^j-1=O((d-1)^r)$.  
      At most
      $[n-1]_{k-1}d^{k-1}kO\big((d-1)^r\big)$ switchings then fail
      condition (\ref{li:b}).
    \item By a similar argument,
      at most $[n]_{k-1}d^{k-1}k^2O\big((d-1)^{r/2}\big)$ switchings fail
      condition (\ref{li:c}).
    \item By a similar argument, 
      at most $[n]_{k-1}d^{k-1}kO\big((d-1)^r\big)$ switchings fail
      condition (\ref{li:d}).
  \end{itemize}
  Adding these up and combining $O(\cdot)$ terms, we find that at most
  \begin{align*}
    [n-1]_{k-1}d^{k-1}k\left(\sum_{j=3}^r2j\Cy{j}+O\big((d-1)^r\big)\right)
  \end{align*}
  switchings out of the original $[n]_kd^k$ fail conditions
  by (\ref{li:a})--(\ref{li:d}), establishing \eqref{eq:flb}.
\end{proof}
  
For backwards switchings, we give a similar upper bound, but we 
only give our lower bound in expectation.
\begin{lem}\label{lem:backwardsswitchings}
  Let $G$ be a random $d$-regular graph on $n$ vertices, and let
  $\alpha$ be a cycle of length $k\leq r$
  in the complete graph $K_n$.  Then
  \begin{align}
    B_\alpha &\leq \bigl(d(d-1)\bigr)^k\label{eq:bup}\\
    \intertext{and}
    \E B_\alpha &\geq \bigl(d(d-1)\bigr)^k\left(
      1 - \frac{\Cl{bs}k(d-1)^{r-1}}{n}\right).\label{eq:blp}
  \end{align}
\end{lem}
\begin{proof}
  The question this time is given $\alpha$,
  how many choices of oriented paths yield a valid
  switching?  For any fixed $\alpha$,
  there are at most $(d(d-1))^k$ choices of oriented paths, proving
  \eqref{eq:bup}.
  For the lower bound, let $B=\sum_{\beta}B_\beta$, where $\beta$
  runs over all cycles of length $k$ in the complete graph.  
  We will first show
  that
  \begin{align}
    B\geq \frac{[n]_k\bigl(d(d-1)\bigr)^k}{2k}\left(
      1 - \frac{4k\sum_{j=3}^rj\Cy{j} + O\bigl(
        k(d-1)^{r}\bigr)}{nd}\right).\label{eq:bigb}
  \end{align}
  As in Lemma~\ref{lem:forwardswitchings}, we give
  conditions that ensure
  a valid switching.  Let  $\beta=v_0\cdots v_{k-1}$, and suppose
  that the paths $u_iv_iw_i$ in $G$
  for $0\leq i\leq k-1$
  satisfy
  \begin{enumerate}[(a)]
    \item the edges $v_iu_i$ and $v_iw_i$ are not contained in any short cycles;
      \label{li:2a}
    \item for all $1\leq j\leq r/2$,
      the distance between the paths $u_iv_iw_i$ and $u_{i+j}v_{i+j}w_{i+j}$
      is at least $r-j+1$.\label{li:2b}
  \end{enumerate}
  Any choice of edges satisfying these conditions gives a valid
  backwards switching:
  Condition (\ref{li:2b}) ensures that $v_i\not\sim v_{i+1}$ and
  $w_i\not\sim u_{i+1}$, as required in the definition of a switching.
  Let $G'$ be the graph obtained by applying the switching.
  We need to check that no short cycles besides $\beta$ 
  are created or destroyed 
  by the switching.  By (\ref{li:2a}), none are destroyed.
  Suppose a short cycle $\beta'$ other than $\beta$ is created in $G'$.
  It consists of paths in $G\cap G'$, portions of $\beta$, 
  and edges $w_iu_{i+1}$.
  Any such path in $G\cap G'$ must have length at least $r/2$
  because
  \begin{itemize}
    \item if it starts at $u_i$, $v_i$, or $w_i$ and ends at
      $u_{i+j}$, $v_{i+j}$, or $w_{i+j}$ for $1\leq j\leq r/2$, then
      (\ref{li:2b}) implies this;
    \item if it starts and ends at one of $u_i$, $v_i$, and $w_i$, then
      (\ref{li:2a}) implies this.
  \end{itemize}
  Hence $\beta'$ must contain exactly one such path.
  The remainder of $\beta'$ must either be an edge $w_iu_{i+1}$,
  or a portion of $\beta'$, both of which are impossible by (\ref{li:2b}).
  
  There are $[n]_kd^k/2k$ choices for $\beta$, and at most
  $(d(d-1))^k$ choices for $u_i, w_i$, $0\leq i<k$.
  As before, we count how many of these potential switchings
  satisfy conditions (\ref{li:2a}) and (\ref{li:2b}) to get a lower
  bound on $B$.
  By similar arguments as in the proof of Lemma~\ref{lem:forwardswitchings}, 
  we find
  that at most
  \begin{align*}
    2[n-1]_{k-1}\big(d(d-1)\big)^{k-1}(d-1)\sum_{j=3}^rjC_j
  \end{align*}
  of the switchings violate condition (\ref{li:2a}), and
  at most $[n]_{k-1}\big(d(d-1)\big)^{k-1}
  O\big((d-1)^{r+1}\big)$ violate condition (\ref{li:2b}), which
  proves \eqref{eq:bigb}.
  
  By Proposition~\ref{prop:McKayestimate}\ref{item:onecycle}
  (or by \cite[eq.~2.2]{MWW}),
  \begin{align*}
    \E \Cy{k} &\leq \frac{\Cr{50}(d-1)^k}{2k}.
  \end{align*}
  Applying this to \eqref{eq:bigb} gives
  \begin{align*}
    \E B \geq \frac{[n]_k\bigl(d(d-1)\bigr)^k}{2k}
    \left( 1-O\left(\frac{k(d-1)^{r-1}}{n}\right)\right)
  \end{align*}
  
    By the exchangeability of the vertex labels of $G$, the law of
  $B_\beta$ is the same for all $k$-cycles $\beta$.
  It follows that $\E B = ([n]_k/2k)\E B_\alpha$, 
  proving \eqref{eq:blp}.
\end{proof}

\subsection{Applying Stein's method}
We will prove a generalization of Theorem~\ref{thm:uniformprocesspoiapprox},
allowing the process of cycles to be indexed by any collection of cycles,
rather than just all cycles of length $r$ or less.
\begin{thm}\label{thm:genpoiapprox}
  Let $G$ be a random $d$-regular graph on $n$ vertices.
  For some collection $\Ii$ of cycles in the complete graph $K_n$
  of maximum length $r$,
  we define $\I=(I_\alpha,\,\alpha\in\Ii)$, with $I_\alpha=\1\{\text{$G$
  contains $\alpha$}\}$.  Let $\Z=(Z_\alpha,\,\alpha\in\Ii)$
  be a vector of independent Poisson random variables, with $\E Z_\alpha
  =(d-1)^{\abs{\alpha}}/[n]_{\abs{\alpha}}$, where $\abs{\alpha}$ denotes
  the length of the cycle $\alpha$.
  
  For some absolute constant $\Cr{52}$, for all $n$
  and $d,r\geq 3$ satisfying $r\leq n^{1/10}$ and $d\leq n^{1/3}$,
  \begin{align*}
    \dtv(\I,\,\Z) &\leq\sum_{\alpha\in\Ii}\frac{\Cl{52}\abs{\alpha}
    (d-1)^{\abs{\alpha}+r-1}}{n^{\abs{\alpha}+1}}.
  \end{align*}
\end{thm}
\begin{proof}
  We will construct an exchangeable pair by taking a step in a reversible
  Markov chain.  To make this chain, define a graph $\GGG$
  whose vertices consist of all $d$-regular graphs on $n$
  vertices.
  For every valid forward $\alpha$-switching with $\alpha\in\Ii$ 
  from a graph
  $G_0$ to $G_1$, make an undirected edge in $\GGG$
  between $G_0$ and $G_1$. Place a weight of 
  $1/[n]_{\abs{\alpha}}d^{\abs{\alpha}}$ on each of these edges.
  The essential fact that will make our
  arguments work is that valid forward $\alpha$-switchings from
  $G_0$ to $G_1$
  are in bijective correspondence with valid backwards $\alpha$-switchings from
  $G_1$ to $G_0$.  Thus, we could have equivalently defined $\GGG$ 
  by forming an edge for every valid backwards switching.

  Define the degree of a vertex in a graph with weighted edges to be the
  sum of the adjacent edge weights.
  Let $d_0$ be the maximum degree of $\GGG$ as defined so far.
  To make $\GGG$ regular, add a weighted loop to each vertex that brings
  its degree up to $d_0$.
  Now, consider a random walk on $\GGG$ that moves with probability
  proportional to the edge weights.  This random walk is a Markov chain
  reversible with respect to the uniform distribution on $d$-regular
  graphs on $n$ vertices.  Thus, if $G$ has this distribution, and we
  obtain $G'$ by advancing one step in the random walk, 
  the pair of graphs $(G,G')$ is exchangeable.
  
  Let $I'_\alpha$ be an indicator on $G'$ containing
  the cycle $\alpha$, and define $\I'=(I'_\alpha,\,\alpha\in\Ii)$.
  It follows from the exchangeability of $G$ and $G'$ that $\I$ and
  $\I'$ are exchangeable, and we can apply
  Proposition~\ref{prop:steinsmethod} on this pair.  Define the events 
  $\Delta_\alpha^+$ and $\Delta_\alpha^-$
  as in that proposition.  By our construction,
  \begin{align*}
    \P[\Delta_\alpha^+\mid G]=\frac{B_\alpha}{d_0[n]_{\abs{\alpha}}
    d^{\abs{\alpha}}},\qquad
    \P[\Delta_\alpha^-\mid G]=\frac{F_\alpha}{d_0[n]_{\abs{\alpha}}
    d^{\abs{\alpha}}}.
  \end{align*}
  Thus by Proposition~\ref{prop:steinsmethod} with 
  all constants set to $d_0$,
  \begin{align}
    \dtv(\I,\,\Z) &\leq \sum_{\alpha\in\Ii}
       \E\left\lvert\frac{(d-1)^{\abs{\alpha}}}{[n]_{\abs{\alpha}}}
         -\frac{B_\alpha}{[n]_{\abs{\alpha}}d^{\abs{\alpha}}}\right\rvert
         +\sum_{\alpha\in\Ii}  
        \E\left\lvert I_\alpha
         -\frac{F_\alpha}{[n]_{\abs{\alpha}}d^{\abs{\alpha}}}\right\rvert.
       \label{eq:dtvbound}
  \end{align}
  We will bound these two sums.  Fix some $\alpha\in\Ii$, and let 
  $\abs{\alpha}=k$.  By Lemma~\ref{lem:backwardsswitchings},
  \begin{align*}
    \frac{B_\alpha}{[n]_{k}d^{k}}
    &\leq \frac{(d-1)^{k}}{[n]_{k}}.
  \end{align*}
  Thus
  \begin{align*}
    \E\left\lvert\frac{(d-1)^{k}}{[n]_{k}}
         -\frac{B_\alpha}{[n]_{k}d^{k}}\right\rvert
       &=\E\left[\frac{(d-1)^{k}}{[n]_{k}}
         -\frac{B_\alpha}{[n]_{k}d^{k}}\right].
  \end{align*}
  Applying the lower bound on $\E B_\alpha$
  from Lemma~\ref{lem:backwardsswitchings} then
  gives
  \begin{align}
    \E\left\lvert\frac{(d-1)^{k}}{[n]_{k}}
         -\frac{B_\alpha}{[n]_{k}d^{k}}\right\rvert
       &\leq \frac{\Cr{bs}k(d-1)^{k+r-1}}{n[n]_{k}}.\label{eq:E1}
  \end{align}
          
  In bounding the other sum, we partition our state space of random regular
  graphs into three events:
  \begin{align*}
    A_1&=\{\text{$G$ does not contain $\alpha$}\},\\
    A_2&=\{\text{$G$ contains $\alpha$, which does not share
      an edge with another short cycle in $G$}\},\\
    A_3&=\{\text{$G$ contains $\alpha$, which shares an edge
    with another short cycle in $G$}\}.
  \end{align*}
  On $A_1$, we have $I_\alpha=F_\alpha=0$.  On $A_2$, both bounds
  from Lemma~\ref{lem:forwardswitchings} apply,
  giving us
  \begin{align*}
    \left\lvert I_\alpha
         -\frac{F_\alpha}{[n]_{k}d^{k}}\right\rvert
      &\leq \frac{2k\sum_{j=3}^rjC_j+\Cr{4}k(d-1)^r}{nd}.
  \end{align*}
  On $A_3$, we have $I_\alpha=1$ and $F_\alpha=0$.
  In all,
  \begin{align*}
    \E\left\lvert I_\alpha
         -\frac{F_\alpha}{[n]_{k}d^{k}}\right\rvert
       &\leq \E\left[\1_{A_2}\frac{2k\sum_{j=3}^rjC_j+\Cr{4}k(d-1)^r}{nd}
       +\1_{A_3}\right]\\
       &=\frac{2k}{nd}\E\biggl[\1_{A_2}\sum_{j=3}^rjC_j\biggr] 
          + \frac{\Cr{4}k(d-1)^r}{nd}
       \P[A_2] + \P[A_3].
  \end{align*}
    Let $\Jj$ be the set of all cycles of length $r$ or less in $K_n$
  that share no edges with $\alpha$.  On the set $A_2$, the graph $G$
  contains no cycles outside of this set (except for $\alpha$), and
  $\sum_{j=3}^r jC_j=k+\sum_{\beta\in\Jj}\abs{\beta}I_\beta$.
  Thus
  \begin{align}
    \E\left\lvert I_\alpha
         -\frac{F_\alpha}{[n]_{k}d^{k}}\right\rvert&\leq
         \frac{2k^2}{nd}\E\1_{A_2}+\frac{2k}{nd}\sum_{\beta\in\Jj}
         \abs{\beta}\E\1_{A_2}I_\beta+\frac{\Cr{4}k(d-1)^r}{nd}
       \P[A_2] + \P[A_3]\nonumber\\
       &\leq \frac{2k^2}{nd}\E{I_{\alpha}}+\frac{2k}{nd}\sum_{\beta\in\Jj}
         \abs{\beta}\E I_{\alpha}I_\beta+\frac{\Cr{4}k(d-1)^r}{nd}
       \E I_\alpha + \P[A_3].\label{eq:sums}
  \end{align}
  By Proposition~\ref{prop:McKayestimate}\ref{item:onecycle},
  \begin{align}
    \frac{2k^2}{nd}\E I_\alpha&=O\Bigl(\frac{k^2(d-1)^k}{n^{k+1}}\Bigr)
      \label{eq:part1}\\
    \intertext{and}
    \frac{\Cr{4}k(d-1)^r}{nd}
       \E I_\alpha&=O\Bigl(\frac{k(d-1)^{k+r-1}}{n^{k+1}}\Bigr)\label{eq:part3}.
  \end{align}
  By Proposition~\ref{prop:McKayestimate}\ref{item:twocycles},
  for any $\beta\in\Jj$, we have $\E I_\alpha I_\beta\leq \Cr{51}(d-1)^{j+k}
    /n^{j+k}$.  For each $3\leq j\leq r$, there are at most $[n]_j/2j$ cycles
  in $\Jj$ of length $j$.  Therefore
  \begin{align}
    \frac{2k}{nd}\sum_{\beta\in\Jj}
         \abs{\beta}\E I_{\alpha}I_\beta
      &\leq \frac{2k}{nd}\sum_{j=3}^r\frac{[n]_j}{2j} \Bigl(\frac{j
      \Cr{51}(d-1)^{j+k}}{n^{j+k}}\Bigr)\nonumber\\
      &\leq \frac{k}{nd}\sum_{j=3}^r \frac{
      \Cr{51}(d-1)^{j+k}}{n^{k}}
      = O\Bigl(\frac{k(d-1)^{k+r-1}}{n^{k+1}}\Bigr).\label{eq:part2}
  \end{align}

  The last term of \eqref{eq:sums} is the most difficult to bound.
  Let $\Kk$ be the set of short cycles in $K_n$  that
  share an edge with $\alpha$, not including $\alpha$ itself.  
  By a union bound,
  \begin{align}
    \P[A_3]&\leq \sum_{\beta\in\Kk}\E I_\alpha I_\beta.\label{eq:A3ub}
  \end{align}
  Now, we classify and count the cycles $\beta\in\Kk$ according to 
  the structure of $\alpha\cup\beta$.
  Suppose that $\beta$ has length $j$, and consider
  the intersection of $\alpha$ and $\beta$ (the graph consisting
  of all vertices and edges contained in both $\alpha$ and $\beta$).
  Suppose this intersection graph has $p$ components and $f$ edges.
  As computed on \cite[p.~5]{MWW}, the number of possible isomorphism
  types of $\alpha\cup\beta$ given $p$ and $f$ is at most 
  $(2r^3)^{p-1}/(p-1)!^2$.  
  For each possible isomorphism type of $\alpha\cup\beta$, 
  there are no more than $2kn^{j-p-f}$ possible choices of $\beta$
  such that $\alpha\cup\beta$ falls into this isomorphism class.
  This is because $\alpha\cup\beta$ has $j+k-p-f$
  vertices, $k$ of which are determined by $\alpha$. In defining $\beta$,
  the remaining $j-p-f$ vertices can be chosen to be anything, and the
  intersection of $\alpha$ and $\beta$ can be rotated around $\alpha$ in
  $2k$ ways, all without changing the isomorphism class of $\alpha\cup\beta$.
  In all, we have shown that the number of $j$-cycles whose overlap with
  $\alpha$ has $p$ components and $f$ edges is at most
  \begin{align*}
    \frac{(2r^3)^{p-1}}{(p-1)!^2}
      2kn^{j-p-f}.
  \end{align*}
  
  For any such choice of $\beta$, we have
  $\E I_\alpha I_\beta
  \leq \Cr{51}(d-1)^{j+k-f}/n^{j+k-f}$ 
  by 
  Proposition~\ref{prop:McKayestimate}\ref{item:twocycles}.
  Applying this to \eqref{eq:A3ub},
  \begin{align}
    \P[A_3]&\leq\sum_{j=3}^r\sum_{p,f\geq 1} \frac{(2r^3)^{p-1}}{(p-1)!^2}
      2kn^{j-p-f}\frac{\Cr{51}(d-1)^{j+k-f}}{n^{j+k-f}}\nonumber\\
    &=\sum_{j=3}^r\sum_{p,f\geq 1} \frac{(2r^3)^{p-1}}{(p-1)!^2}
      \frac{2k\Cr{51}(d-1)^{j+k-f}}{n^{k+p}}\nonumber\\
    &=\sum_{j=3}^rO\Bigl(\frac{k(d-1)^{j+k-1}}{n^{k+1}}\Bigr)
    =O\Bigl(\frac{k(d-1)^{k+r-1}}{n^{k+1}}\Bigr).\label{eq:part4}
  \end{align}
  Combining \eqref{eq:part1}, \eqref{eq:part3}, \eqref{eq:part2}, and
  \eqref{eq:part4}, we have
  \begin{align*}
    \E\left\lvert I_\alpha
         -\frac{F_\alpha}{[n]_{k}d^{k}}\right\rvert&=
       O\Bigl(\frac{k(d-1)^{k+r-1}}{n^{k+1}}\Bigr).
  \end{align*}
  Applying this and \eqref{eq:E1} to \eqref{eq:dtvbound} establishes
  the theorem.
\end{proof}

\begin{proof}[Proof of Theorem~\ref{thm:uniformprocesspoiapprox}]
  If $r>n^{1/10}$ or $d>n^{1/3}$, then 
  $c(d-1)^{2r-1}/n>1$ for a sufficiently large choice of
  $c$, and the total variation bound is trivial.
  Thus we can assume that this is not the case and apply
  the previous theorem:
  \begin{align*}
    \dtv(\I,\,\Z) &\leq \sum_{\alpha\in\Ii}\frac{\Cr{52}\abs{\alpha}
    (d-1)^{\abs{\alpha}+r-1}}{n^{\abs{\alpha}+1}}\\
      &= \sum_{k=3}^r\frac{[n]_k}{2k}\Bigl(\frac{\Cr{52}k
    (d-1)^{k+r-1}}{n^{k+1}}\Bigr)\\
    &=O\Bigl(\frac{(d-1)^{2r-1}}{n}\Bigr).\qedhere
  \end{align*}
\end{proof}
Since the cycle counts $(\Cy{3},\ldots,\Cy{r})$ are a functional
of $\I$, this corollary implies that
\begin{align*}
  \dtv\big((\Cy{3},\ldots,\Cy{r}),\,(Z_3,\ldots,Z_r)\big)
    &\leq \frac{c(d-1)^{2r-1}}{n},
\end{align*}
where $(Z_3,\ldots,Z_r)$ is a vector of independent Poisson random
variables with $\E Z_k = (d-1)^k/2k$.  
This bound is often less than optimal, since this theorem
fails to exploit the $\lambda_k^{-1/2}$ factors in 
Proposition~\ref{prop:steinsmethod}.  We will take advantage
of these factors in the following proposition, and then apply this to
prove Theorem~\ref{thm:bestpoiapprox}.
\begin{prop}\label{prop:genpoiapprox}
  With the set-up of Theorem~\ref{thm:genpoiapprox}, divide up the
  collection of cycles $\Ii$ into bins $\Bb_1,\ldots,\Bb_s$. Let
  \begin{align*}
    I_k=\sum_{\alpha\in\Bb_k}I_\alpha, \qquad Z_k=\sum_{\alpha\in\Bb_k}Z_\alpha,
  \end{align*}
  and let $\lambda_k=\E Z_k$.
  Then
  \begin{align*}
    \dtv\big((I_1,\ldots,I_s),\,(Z_1,\ldots,Z_s)\big))
      &\leq \Cr{52}\sum_{k=1}^s\xi_k\sum_{\alpha\in\Bb_k}
      \frac{\abs{\alpha}
    (d-1)^{\abs{\alpha}+r-1}}{n^{\abs{\alpha}+1}},
  \end{align*}
  where $\xi_k=\min\big(1, 1.4\lambda_k^{-1/2}\big)$.
\end{prop}
\begin{proof}
  Define the exchangeable pair $(G,G')$ as in
  Theorem~\ref{thm:genpoiapprox}, and define
  $I'_1,\ldots,I'_s$ as the analogous quantities in
  $G'$.
  Define $\Delta_k^+$ and $\Delta_k^-$
  as in Proposition~\ref{prop:steinsmethod}, noting that
  \begin{align*}
    \P[\Delta_k^+\mid G] = \sum_{\alpha\in\Bb_k}\frac{B_\alpha}{d_0[n]_{\abs{\alpha}}
    d^{\abs{\alpha}}},\qquad\qquad
    \P[\Delta_k^-\mid G] = \sum_{\alpha\in\Bb_k}\frac{F_\alpha}{d_0[n]_{\abs{\alpha}}
    d^{\abs{\alpha}}}.
  \end{align*}
  By Proposition~\ref{prop:steinsmethod},
  \begin{align*}
    \dtv\big((I_1,\ldots&, I_s),\,(Z_1,\ldots,Z_s)\big)\\
    &\leq \sum_{k=1}^s
       \xi_k\left(\E\left\lvert\lambda_k-d_0\P[\Delta_k^+\mid G]\right\rvert+
         \E\left\lvert I_k-d_0\P[\Delta_k^-\mid G]\right\rvert\right)\\
    &=\sum_{k=1}^s
       \xi_k\E\Biggl\lvert\sum_{\alpha\in\Bb_k}\left(
       \frac{(d-1)^{\abs{\alpha}}}{[n]_{\abs{\alpha}}}-\frac{B_\alpha}
       {[n]_{\abs{\alpha}}d^{\abs{\alpha}}}\right)
       \Biggr\rvert\\
       &\qquad\quad\phantom{}+\sum_{k=1}^s\xi_k
         \E\left\lvert\sum_{\alpha\in\Bb_k}\left(I_\alpha-\frac{F_\alpha}
       {[n]_{\abs{\alpha}}d^{\abs{\alpha}}}\right)
       \right\rvert.
  \end{align*}
  These summands were already bounded in expectation in 
  Theorem~\ref{thm:genpoiapprox}, and applying these bounds proves
  the proposition.
\end{proof}
\begin{proof}[Proof of \thref{thm:bestpoiapprox}]
  If $d> n^{1/3}$ or $r>n^{1/10}$, then
  $c\sqrt{r}(d-1)^{3r/2-1}/n>1$ for a 
  sufficiently large choice of $c$,
  and the theorem holds trivially.  Thus we can assume that
  $d\leq n^{1/3}$ and $r\leq n^{1/10}$.
  
  Let $\lambda_k=(d-1)^k/2k$.
  With $\Ii_k$ defined as the set of all cycles in $K_n$ of length $k$,
  we apply the previous proposition with bins $\Ii_3,\ldots,\Ii_r$ to get
  \begin{align*}
    \dtv\big((\Cy{3},\ldots,\Cy{r}),\,(Z_3,\ldots,Z_r)\big)
      &\leq \Cr{52}\sum_{k=3}^r1.4\lambda_k^{-1/2}\sum_{\alpha\in\Ii_k}
      \frac{k
    (d-1)^{k+r-1}}{n^{k+1}}\\
    &=\sum_{k=3}^rO\Bigl(
      \frac{\sqrt{k}
    (d-1)^{k/2+r-1}}{n}\Bigr)\\
    &=O\Bigl(\frac{\sqrt{r}(d-1)^{3r/2-1}}{n}\Bigr).\qedhere
  \end{align*}  
\end{proof}

\chapter{Fluctuations of linear eigenvalue statistics}
\label{chap:fluctuations}
\section{Fluctuations for random regular graphs: main results}
\label{sec:flucresults}
Let $\lambda_1\geq\cdots\geq \lambda_n$ be the eigenvalues of $(d-1)^{-1/2}A_n$,
where $A_n$ is the adjacency matrix of a random $d$-regular graph.
The main result is that the fluctuations of $\sum f(\lambda_i)$ for a sufficiently
smooth function~$f$ converge
either in law either to compound Poisson or to Gaussian, depending on whether
$d$ is held fixed or grows. The exact limiting distribution depends on $f$;
it can be written in terms of the decomposition of $f$ as a sum of modified Chebyshev
polynomials, which we define now:
\begin{align*}
  \Gamma_0(x) & = 1, \\
  \Gamma_{2k}(x) &= 2 T_{2k}\left (\frac{x}{2} \right ) + \frac{d-2}{(d-1)^k}
  &\text{for $k \geq 1$,}\\
  \Gamma_{2k+1}(x) &= 2 T_{2k+1}\left(\frac{x}{2}\right)&\text{for $k \geq 0$,}
\end{align*}
with
$\{T_n(x)\}_{n \in \mathbb{N}}$ the Chebyshev polynomials of the first kind 
on the interval $[-1,1]$.

Let $\rho>1$, and consider the image of the circle 
of radius $\rho$,
centered at the origin, under the map $f(z) = \frac{z+ z^{-1}}{2}$. We
call this the Bernstein ellipse of radius $\rho$. 
The ellipse has foci at $\pm 1$, and the sum
of the major semiaxis and the minor semiaxis is exactly~$\rho$. 
Analyticity on a Bernstein ellipse implies a decomposition as a sum of Chebyshev
polynomials. We can now give our main result on eigenvalue fluctuations:

\begin{thm}\thlabel{thm:dfixedlimit}
  Fix $d\geq 3$, and let $G_n$ be a random $d$-regular graph on
  $n$ vertices from the permutation or uniform model, with adjacency matrix $A_n$.  
  Let $\lambda_1\geq\cdots\geq\lambda_n$ be the eigenvalues
  of $(d-1)^{-1/2}A_n$.
  
  Let $\alpha_0=1$ in the case of the permutation model and $\alpha_0=3/2$ for the permutation
  model.
  Suppose that $f$ is a function defined on $\CC$, analytic
  inside a Bernstein ellipse of radius $2\rho$, where $\rho=(d-1)^{\alpha}$
  for some $\alpha>\alpha_0$, and such that $|f(z)|$ is bounded inside this ellipse.
  Then $f(x)$ can be expanded on $[-2,2]$ as
  \begin{align*}
    f(x) = \sum_{k=0}^{\infty}a_k\Gamma_k(x),
  \end{align*}
  and $Y_f^{(n)}\defeq \sum_{i=1}^nf(\lambda_i)-na_0$ converges
  in law as $n\to\infty$ to the infinitely divisible random variable
  \begin{align*}
    Y_f\defeq \sum_{k=1}^{\infty}\frac{a_k}{(d-1)^{k/2}}\CNBW{k},
  \end{align*}
  with $\CNBW{k}$ as defined on p.~\pageref{def:CNBW} for the permutation
  or uniform model of random graph.
\end{thm}

We can also prove that the limiting distribution of linear eigenvalue
functionals is normal when the degree of $G_n$ grows with $n$.
The conditions of the theorem are messy, and more needs to be defined
before we can even state it. The result is found in
\thref{thm:dgrowslimit}.

\section{Proof of eigenvalue fluctuation results}
\label{sec:eigenfluc}
We will use \thref{thm:bestpermutationpoiapprox,thm:bestpoiapprox} to estimate the distribution
of \emph{cyclically non-backtracking walks} in a random regular graph. 
As we will see in \thref{prop:cnbweigenvalues}, counts
of these walks can be written in terms of the graph's eigenvalues, which allows us to compute the limiting
fluctuations of linear eigenvalue statistics.

If a walk on a graph begins and ends at the same vertex, we
call it \emph{closed}.
We call a walk on a graph \emph{non-backtracking} if it never follows
an edge and immediately follows that same edge backwards.
Non-backtracking walks are also known as irreducible.

     \begin{figure}
      \begin{center}
        \begin{tikzpicture}[scale=1.8,vert/.style={circle,fill,inner sep=0,
              minimum size=0.15cm,draw},>=stealth]
            \draw[thick] (0,0) node[vert,label=right:1](s0) {}
                  -- ++(0,1cm) node[vert,label=below right:2](s1) {}
                  -- ++(0.707cm,0.707cm) node[vert,label=right:3](s2) {}
                  -- ++(-0.707cm,0.707cm) node[vert,label=right:4](s3) {}
                  -- ++(-0.707cm,-0.707cm) node[vert,label=left:5](s4) {}
                  --(s1);
        \end{tikzpicture}
      \end{center}
      \caption{The walk $1\to 2\to 3\to 4\to 5\to 2\to 1$ is non-backtracking,
      but not cyclically non-backtracking. Such walks have a
    ``lollipop'' shape.}
      \label{fig:cnbw}
    \end{figure}
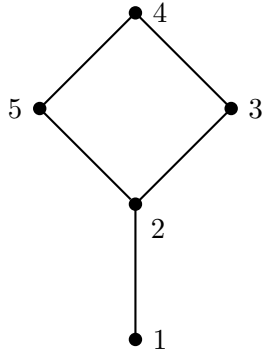  
 Consider a closed non-backtracking walk, and suppose that its last
 step is anything other than the reverse of its first step (i.e., the
 walk does not look like the one given  in
 Figure~\ref{fig:cnbw}).  Then we call it a \emph{cyclically non-backtracking
 walk}.  These walks occasionally go by the name
 strongly irreducible.

\renewcommand{\Cy}[2][\infty]{C_{#2}^{(#1)}}

Let $G_n$ be a random $d$-regular graph on $n$ vertices, with the exact model
to be specified later. To allow for more consistent statements between the permutation
and uniform models, we talk about $d$-regular graphs rather than $2d$-regular
graphs from the permutation model, with the
understanding that $d$ is even.
Let $\Cy[n]{k}$ be the number of cycles of length $k$ in $G_n$.
  We define the random variable $\CNBW[n]{k}$ to be the number
 of cyclically non-backtracking walks of length $k$ in $G_n$.
Define $(\Cy{k},\, k\geq 1)$ to be
independent
Poisson random variables.
When we discuss the permutation model, take $\E \Cy{k}=a(d,k)/2k$.
When we work with the uniform model, take $\E \Cy{k}=(d-1)^k/2k$ for
$k\geq 3$, and define $\Cy{1}$, $\Cy{2}$, $\Cy[n]{1}$,
and $\Cy[n]{2}$ as zero.

  Define \label{def:CNBW}
  \begin{align*}
    \CNBW{k} = \sum_{j\mid k} 2j\Cy{j}.
  \end{align*}
  For any cycle in $G_n$ of length $j$, where $j$ divides $k$, we obtain $2j$ 
  cyclically non-backtracking walks
  of length~$k$ by choosing a starting point and direction and then
  walking around the cycle repeatedly.
  In fact, if $d$ and $k$ are small compared to $n$, 
  then these are likely
  to be the only cyclically non-backtracking walks of length~$k$ in $G_n$, as we will prove
  in the course of the following theorems.
  \begin{thm}\thlabel{thm:CNBWbound}
    For some absolute constant~$c$, it holds in the permutation model
    of random $d$-regular graph that
    \begin{align*}
      \dtv\left( \big( \CNBW[n]{k},\, 1\leq k\leq r  \big), 
        \big( \CNBW[\infty]{k},\, 1\leq k\leq r   \big)\right)
        &\leq \frac{cr^4(d-1)^{r}}{n},\\
        \intertext{and in the uniform model of random $d$-regular graph that}
      \dtv\left( \big( \CNBW[n]{k},\, 1\leq k\leq r  \big), 
        \big( \CNBW[\infty]{k},\, 1\leq k\leq r   \big)\right)
        &\leq \frac{c\sqrt{r}(d-1)^{3r/2}}{n}.
    \end{align*}
  \end{thm}
\begin{proof}
  For any measurable function $f$ and random variables $X$ and $Y$,
  we have $\dtv(f(X),\,f(Y))\leq \dtv(X,Y)$.  
  It follows by \thref{thm:bestpermutationpoiapprox} that in the permutation model,
  \begin{align}
    \dtv \left( \bigg( \sum_{j|k}2j\Cy[n]{j},\, 1\leq k\leq r  \bigg), 
        \big( \CNBW[\infty]{k},\, 1\leq k\leq r   \big)\right)
        &\leq O\biggl(\frac{r^2(d-1)^{r}\log(d-1)}{n}\biggr),\label{eq:almosttvperm}\\
        \intertext{and it follows by Theorem~\ref{thm:bestpoiapprox} that in the uniform model,}
    \dtv \left( \bigg( \sum_{j|k}2j\Cy[n]{j},\, 1\leq k\leq r  \bigg), 
        \big( \CNBW[\infty]{k},\, 1\leq k\leq r   \big)\right)
        &\leq O\biggl(\frac{\sqrt{r}(d-1)^{3r/2-1}}{n}\biggr).\label{eq:almosttv}
  \end{align}

  To finish the proof, we will show that
  \begin{align}
    \bigg( \sum_{j|k}2j\Cy[n]{j},\,1\leq k\leq r  \bigg)=
    \big( \CNBW[n]{k},\,1\leq k\leq r  \big)\label{eq:vecequality}
  \end{align}
  with high probability, in both models. We go out of order and consider the uniform
  model first.
  These two vectors differ exactly when either of the following
  occur:
  \begin{enumerate}
    \item[Event $E_1$:] \label{li:overlap}$G_n$ contains a $j$-cycle and a
      $k$-cycle with a vertex in common, with $j+k\leq r$.
    \item[Event $E_2$:]\label{li:close}$G_n$ contains 
      a $j$-cycle and a $k$-cycle whose distance
      is $l$, with $l\geq 1$ and $j+k+2l\leq r$ (see 
      Figure~\ref{fig:attachedcycles}).
  \end{enumerate}
  We have already done most of the work in bounding the probability
  of event $E_1$.  Let $\alpha$ be some arbitrary $k$-cycle.
  In \eqref{eq:part4}, we bounded the probability that $G_n$ contained
  $\alpha$ and another cycle sharing an \emph{edge} with $\alpha$.
  With the same notation and nearly the same analysis (the only
  real change
  is allowing $f$ to be zero),
  \begin{align*}
    \P[E_1] &\leq \sum_{k=3}^{r-3} \frac{[n]_k}{2k}
      \sum_{j=3}^{r-3-k}\sum_{\substack{p\geq 1,\\f\geq 0}}
        \frac{(2r^3)^{p-1}}{(p-1)!^2}2kn^{j-p-f}
      \frac{\Cr{51}(d-1)^{j+k-f}}{n^{j+k-f}}\\
      &=O\left(\frac{(d-1)^r}{n}\right).
  \end{align*}

  To bound the probability of $E_2$, first observe that the number
  of subgraphs of $K_n$ consisting of a $j$-cycle and a $k$-cycle (which
  do not overlap) connected by a path of length $l$ is $[n]_{j+k+l-1}/4$.
  By Proposition~\ref{prop:McKayestimate}\ref{item:attachedcycles},
  each of these is contained in $G_n$ with probability
  at $O\bigl((d-1)^{j+k+l}/n^{j+k+l}\bigr)$.  By a union bound,
  \begin{align*}
    \P[E_2] &\leq \sum_{j+k+2l\leq r}\frac{[n]_{j+k+l-1}}{4}
      O\left(\frac{(d-1)^{j+k+l}}{n^{j+k+l}}\right)\\
      &= O\left(\frac{(d-1)^{r-1}}{n}\right).
  \end{align*}
  Thus \eqref{eq:vecequality} holds with probability
  $1-O\bigl((d-1)^r/n\bigr)$.  If two random variables
  are equal with probability $1-\eps$, then the total variation distance between their laws is at most $\eps$.  Thus
  the two random vectors in \eqref{eq:vecequality} have total variation
  distance $O\bigl((d-1)^r/n\bigr)$.  This fact and
  \eqref{eq:almosttv} prove the theorem for the uniform model.
  
  A similar argument in the permutation model would work.
  Instead, we will just cite \cite[Corollary~16]{DJPP}, which says that
  \eqref{eq:vecequality} holds with probability $O(r^4(d-1)^r/n)$ in the permutation model,
  using an argument based on \cite{LP}. This together with \eqref{eq:almosttvperm} completes
  the proof.
\end{proof}

Next, we will relate
Theorem~\ref{thm:CNBWbound}
to the eigenvalues of the adjacency matrix of $G_n$.
Recall the modified Chebyshev polynomials $\Gamma_k(x)$ defined
in Section~\ref{sec:flucresults}.
The following proposition is folkloric, following a long tradition
of linking up counts of walks on graphs with polynomial traces of their adjacency matrices.
\begin{prop}[{\cite[Proposition~32]{DJPP}}]\thlabel{prop:cnbweigenvalues}
  Let $A_n$ be the adjacency matrix of $G_n$, and
  let $\lambda_1\ge\cdots\ge\lambda_n$ be the eigenvalues
  of $(d-1)^{-1/2}A_n$.  Then
  \begin{align*}
    \sum_{i=1}^n\Gamma_k(\lambda_i)&=(d-1)^{-k/2}\CNBW[n]{k}.
  \end{align*}
\end{prop}

By Theorem~\ref{thm:CNBWbound}, we know the limiting distribution
of $\sum_{i=1}^nf(\lambda_i)$ when $f(x)=\Gamma_k(x)$.
The plan now is to extend this to a more general class of functions
by approximating by this
polynomial basis. We will need the following bounds on the eigenvalues
of random regular graphs.
\begin{prop}\label{prop:ebound}
  Let $G_n$ be a random $d$-regular graph on $n$ vertices, in either the permutation
  or uniform models.
  \begin{enumerate}[(a)]
    \item Suppose that $d\geq 3$ is fixed.  For any $\eps>0$, asymptotically
      almost surely, all but the highest eigenvalue of $G_n$
      is bounded by $2\sqrt{d-1}+\eps$.\label{item:dfixed}
    \item Suppose that $d=d(n)$ satisfies $d=o(n^{1/2})$.
      Then for some absolute constant $\Cr{ks}$, asymptotically
      almost surely, all but the highest eigenvalue
      of $G_n$ is bounded by $\Cl{ks}\sqrt{d}$.
      \label{item:dgrows}
  \end{enumerate}
\end{prop}
\begin{proof}
  In the permutation model,
  \cite[Theorem~1.1]{Fri} proves (\ref{item:dfixed}) and \cite[Theorem~24]{DJPP}
  proves (\ref{item:dgrows}).
  
  In the uniform model, it is well known that (\ref{item:dfixed}) follows 
  from the results in \cite{Fri} by various contiguity
  results, but we cannot find an
  argument written down anywhere and will give one here.
  When $d$ is even, it follows from \cite[Theorem~1.1]{Fri} and
  the fact that for fixed $d$, permutation random graphs have no loops
  or multiple edges with probability bounded away from zero.
  This implies that the eigenvalue bound holds for permutation random
  graphs conditioned to be simple, and
  \cite[Corollary~1.1]{GJKW} transfers the result to the uniform model.
  When $d$ is odd (and $n$ even, as it has to be), 
  we apply \cite[Theorem~1.3]{Fri}, which gives
  the eigenvalue bound for graphs formed by superimposing $d$ random
  perfect matchings of the $n$ vertices. These are simple with probability
  bounded away from zero, and \cite[Corollary~4.17]{Wor99} transfers
  the result to the uniform model.
  
  Fact~(\ref{item:dgrows}) in the uniform model is proven in a more general context in
  \cite[Lemma~18]{BFSU}.
\end{proof}

\begin{proof}[Proof of \thref{thm:dfixedlimit}]
  The following facts about the Chebyshev approximation follow
  exactly as in Lemma~34 of \cite{DJPP}:
  \begin{enumerate}[(i)]
    \item \label{li:pointwise}The Chebyshev series approximation for
      $f(x)$ converges pointwise on the interval $\bigl[2, 
      d/\sqrt{d-1}\bigr]$.
    \item \label{li:uniform} The series converges uniformly
      on $[-2-\eps,2+\eps]$, for some $\eps>0$.  
      In fact, defining the partial sum
      $f_m(x)=\sum_{k=0}^{m}a_k\Gamma_k(x)$,
      \begin{align*}
        \sup_{|x|\leq 2+\eps}\abs{f(x)-f_m(x)}&\leq M(d-1)^{-\alpha' m},
      \end{align*}
      where $M$ is a constant depending on $f$ and $d$,
      and $\alpha_0<\alpha'<\alpha$.
    \item \label{li:coefficients}The coefficients obey the bound
      \begin{align*}
        \abs{a_k}&\leq M(d-1)^{-\alpha k}.
      \end{align*}
  \end{enumerate}

  The sum defining $Y_f$ converges almost
  surely, since it can be rewritten as
  \begin{align*}
    Y_f=\sum_{j=1}^{\infty}\sum_{i=1}^{\infty}\frac{a_{ij}}{(d-1)^{ij/2}}
    2j\Cy{j},
  \end{align*}
  and this is a sum of independent random variables, bounded in
  $L^2$ by fact~(\ref{li:coefficients}).
  For some $\beta<1/\alpha$, define
  \begin{align*}
    r_n&=\biggl\lfloor\frac{\beta\log n}{\log(d-1)}\biggr\rfloor,\\
    X_f^{(n)} &= \sum_{k=1}^{r_n}\frac{a_k}{(d-1)^{k/2}}\CNBW[n]{k}.
  \end{align*}
  We will use $X_f^{(n)}$ to approximate $Y_f^{(n)}$, noting that
  $X_f^{(n)}=\sum_{i=1}^nf_{r_n}(\lambda_i)-na_0$.
  By Theorem~\ref{thm:CNBWbound} and our choice of $r_n$,
  \begin{align*}
    \limn\dtv\biggl(X_f^{(n)},\; \sum_{k=1}^{r_n}\frac{a_k}{(d-1)^{k/2}}\CNBW{k}\biggr)=0.
  \end{align*}
  This sum converges
  almost surely to $Y_f$ as $n$ tends to infinity, 
  so $X_f^{(n)}$ converges in law
  to $Y_f$.  By Slutsky's Theorem, we need only show that
  $Y_f^{(n)}-X_f^{(n)}$ converges to zero in probability.
  
  Fix $\delta>0$.  We need to show that
  \begin{align*}
    \limn \P\left[\bigl\lvert Y_f^{(n)}-X_f^{(n)}\bigr\rvert > \delta\right]=0.
  \end{align*}
  We have
  \begin{align*}
    \bigl\lvert Y_f^{(n)}-X_f^{(n)}\bigr\rvert&\leq \sum_{i=1}^n \abs{f(\lambda_i)
      -f_{r_n}(\lambda_i)}.
  \end{align*}
  The top eigenvalue $\lambda_1$ is always equal to $d/2\sqrt{d-1}$,
  and by fact~(\ref{li:pointwise}), we have the deterministic
  limit $f_{r_n}(\lambda_1)\to f(\lambda_1)$.  Thus
  $f(\lambda_i)
      -f_{r_n}(\lambda_i) < \delta/2$ for all sufficiently
  large $n$.
  
  Suppose that the remaining eigenvalues are contained in
  $[-2-\eps,2+\eps]$.  By fact~(\ref{li:uniform}),
  \begin{align}
    \sum_{i=2}^n \abs{f(\lambda_i)-f_{r_n}(\lambda_i)}
      &\leq  M(n-1)(d-1)^{-\alpha' r_n}
      \leq M n^{-\alpha'\beta+1},\label{eq:eigsum}
  \end{align}
  and this tends to zero since $\alpha'\beta<1$.
  For sufficiently large $n$, this sum is thus
  bounded by $\delta/2$.  We can conclude that
  for all large enough $n$,
  \begin{align*}
    \P\Bigl[\bigl\lvert Y_f^{(n)}-X_f^{(n)}\bigr\rvert > \delta\Bigr]
      &\leq \P\biggl[\sup_{2\leq i\leq n}\abs{\lambda_i}\leq 2+\eps\biggr],
  \end{align*}
  and this tends to zero by Proposition~\ref{prop:ebound}\ref{item:dfixed}.
\end{proof}

  The following theorem
can be applied only when the degree of the graph
grows more slowly than any positive power of $n$. 
This does not appear explicitly in the statement
of the theorem, but
its conditions cannot be satisfied otherwise.

To remove dependence on $d$ from our polynomial basis, define
\begin{align*}
  \Phi_0(x)&=1\\
  \Phi_k(x)&=2T_k\left(\frac{x}{2}\right)\quad\text{for $k\geq 1$.}
\end{align*}
\begin{thm}\thlabel{thm:dgrowslimit}
  Let $G_n$ be a  random $d_n$-regular graph 
  on $n$ vertices from the permutation
  or uniform models, with $d_n\to\infty$ as $n\to\infty$.
  Let $\lambda_1\geq\cdots\geq\lambda_n$ be the eigenvalues
  of $(d_n-1)^{-1/2}A_n$.
  Suppose $f$ is an entire function on $\CC$,
  and recall $\Cr{ks}$ from 
  Proposition~\ref{prop:ebound}.  The function $f$ admits the
  absolutely convergent expansion $f(x)=\sum_{i=0}^{\infty}a_i\Phi_i(x)$
  on $[-\Cr{ks},\Cr{ks}]$.  
  Denote the $k$th truncation of
  this series by $f_k\defeq \sum_{i=0}^ka_i\Phi_i$.
  Let
  \begin{align*}
    r_n=\Bigl\lfloor{\frac{\beta\log n}{\log(d_n-1)}}\Bigr\rfloor,
  \end{align*}
  with $\beta$ to be specified later.
  Suppose that the following conditions on $f$ hold:
  \begin{enumerate}[(i)]
    \item Let $\alpha_0=1$ in the case of the permutation model and $\alpha_0=3/2$
      for the uniform model. For some $\alpha>\alpha_0$ and $M> 0$,\label{item:asump1}
      \begin{align*}
        \sup_{|x|\leq\Cr{ks}}|f(x)-f_k(x)|\leq M\exp(-\alpha kh(k)),
      \end{align*}
      where $h$ is some function such that
      $h(r_n)\geq \log(d_n-1)$ for some choice
      of $\beta<1/\alpha$, for sufficiently large $n$.
    \item \label{item:asump2}
      \begin{align*}
        \limn\left|f_{r_n}\left(d_n(d_n-1)^{-1/2}\right) - 
        f\left(d_n(d_n-1)^{-1/2}\right)\right|=0.
      \end{align*}
  \end{enumerate}
  Let $\mu_k(d)=\E \CNBW{k}$, noting that $\CNBW{k}$
  depends on $d$.
  We define the following array of constants, which we will use to recenter
  the random variable $\sum_{i=1}^n f(\lambda_i)$:
  \begin{align*}
    m_f(n)\defeq na_0+\sum_{k=1}^{r_n}\frac{a_k}{(d_n-1)^{k/2}}
      \big(\mu_k(d_n)-(d_n-2)n\mathbf{1}\{\text{$k$ is even}\}
      \big)
  \end{align*}
   Then, as $n\to\infty$, the random variable
   \begin{align*}
    Y_f^{(n)}\defeq  \sum_{i=1}^nf(\lambda_i) - m_f(n)
   \end{align*}
   converges in law to a normal random variable with mean zero
   and variance $\sigma_f=\sum_{k=1}^{\infty}2ka_k^2$ for the permutation model
   case and  $\sigma_f=\sum_{k=3}^{\infty}2ka_k^2$ for the uniform model case.
\end{thm}
\begin{proof}
  As $n$ tends tends to infinity, so does $r_n$, 
  since assumption~(\ref{item:asump1}) could not
  be satisfied otherwise.
  Let $k_0=1$ in the permutation model case and $k_0=3$ in the uniform model
  case.
  Define
  \begin{align*}
    X_f^{(n)}=\sum_{k=k_0}^{r_n}\frac{a_k}{(d_n-1)^{k/2}}\CNBW[n]{k}
      -\E \sum_{k=k_0}^{r_n}\frac{a_k}{(d_n-1)^{k/2}}\CNBW{k}  \\\intertext{and}
    \widetilde{X}_f^{(n)}=\sum_{k=k_0}^{r_n}\frac{a_k}{(d_n-1)^{k/2}}\CNBW{k}
      -\E \sum_{k=k_0}^{r_n}\frac{a_k}{(d_n-1)^{k/2}}\CNBW{k},
  \end{align*}
  noting that $X_f^{(n)}=\sum_{i=1}^n f_{r_n}(\lambda_i)-m_f(n)$.
  Also, note that $\CNBW{k}$ depends on $d_n$.
  
  Let 
  \begin{align*}
    N_k^{(n)}=\begin{cases}
      (d_n-1)^{-k/2}\bigl(\CNBW{k}-\E\CNBW{k}\bigr)&\text{if $k\leq r_n$,}\\
      0&\text{otherwise,}
    \end{cases}
  \end{align*}
  and let
  $Z_1,Z_2,\ldots$ be independent normals with $\E Z_k=0$ and
  $\E Z_k^2=2k$.  We will show that $(N_k^{(n)},\,k\geq k_0)$
  converges in law to $(Z_k,\,k\geq k_0)$ as $n\to\infty$.
  Rewrite $N_k^{(n)}$ as
  \begin{align*}
    N_k^{(n)} = \frac{1}{(d_n-1)^{k/2}}\big(2k\Cy{k}-(d_n-1)^k\big)
      +\frac{1}{(d_n-1)^{k/2}}\sum_{\substack{j\mid k\\j<k}}\big(2j\Cy{j}
       - (d_n-1)^j\big).
  \end{align*}
  The first term converges to a centered normal with variance
  $2k$ as $n\to\infty$, by the normal approximation of the Poisson
  distribution.  The random variables $\Cy{1},\Cy{2},\ldots$
  are independent, so the convergence of $(N_k^{(n)},\,k\geq k_0)$
  follows if we show that the remaining terms converge to zero
  in probability.  This holds by Chebyshev's inequality, since
  \begin{align*}
    \var \bigg[\frac{1}{(d_n-1)^{k/2}}\sum_{\substack{j\mid k\\j<k}}\big(2j\Cy{j}
       - (d_n-1)^j\big)\bigg]
       = \sum_{\substack{j\mid k\\j<k}}O\bigl(2j(d_n-1)^{j-k}\bigr),
  \end{align*}
  and this vanishes as $d_n$ grows.
  
  It follows by the continuous mapping theorem
  that $\widetilde{X}_f^{(n)}$ converges to normal with
  variance $\sigma_f$.  By Theorem~\ref{thm:CNBWbound},
  the total variation distance between $X_f^{(n)}$ and
  $\widetilde{X}_f^{(n)}$ approaches zero as $n\to\infty$,
  so $X_f^{(n)}$ converges in law to the same limit.
  
  All that remains is to show that $Y_f^{(n)}-X_f^{(n)}$
  converges to zero in probability.  This follows exactly
  as in Theorem~\ref{thm:dfixedlimit}, using
  assumptions (i) and (ii) and
  Proposition~\ref{prop:ebound}\ref{item:dgrows}.
\end{proof}

\chapter{Minor processes and the Gaussian free field}
\label{chap:GFF}
The typical approach to random matrices is to consider a sequence of random matrices $X_n$ of increasing
size. Each matrix $X_n$ is considered in isolation; the different matrices are not considered on a common
probability space, so they have no joint distribution. Some recent work has instead looked at the
matrices together on a single probability space. 
For example, suppose that $X$ is an infinite random Hermitian matrix with
independent real standard Gaussians along the diagonal and independent complex
standard Gaussians above the diagonal. Let $X_n$ be the first $n$ rows and columns of $X$.
Then $X_n$ is drawn from the \emph{Gaussian Unitary Ensemble}, and the joint distribution
of the eigenvalues of these matrices is called the \emph{GUE-corners process} or \emph{GUE-minors process}.
This process was studied in \cite{Bar} and \cite{JN}.
One can also form general $\beta$-Hermite corners processes \cite[Definition~1.1]{GS}
and $\beta$-Jacobi corners processes \cite{BG}.
These processes are closely related to interacting particle systems; see \cite{FerSurvey}
for a survey. There are also many connections with the KPZ universality class of random surfaces
\cite{BorFer}. Minors of Dyson's Brownian motion have also been studied \cite{ANvM} and can
be put into a common framework with corners processes \cite{Warren,GS}.

The connection to the Gaussian free field (to be called the GFF from now on) comes from
\cite{Bor1}.
We describe a particular but important case of that paper's main result, given by considering
only the single sequence $\{1,2,\ldots\}$.
Let $W$ be an infinite symmetric matrix whose entries have all moments finite.
Suppose the the entries above the diagonal are i.i.d.\ and match
the standard Gaussian to four moments, and the diagonal entries have variance~$2$.
Let $W_n$ be the matrix consisting of the first $n$ rows and columns of $W$.
Borodin then considered the joint eigenvalue fluctuations of these random matrices.

Let $z$ be a complex number in the upper half plane $\mathbb{H}$. Define $y=\abs{z}^2$ and $x=2 \Re(z)$.  
Consider the minor $W(\lfloor ny \rfloor)$, and
let $N(z)$ be the number of its eigenvalues that are
greater than or equal to $\sqrt{n}x$. 
Define the \textit{height function} 
\begin{align}
  \label{eq:heightfn}
H_n(z)\defeq  \sqrt{\frac{\pi}{2}} N(z). 
\end{align}
Then Borodin shows that $\{ H_n(z) - E H_n(z),\; z\in \mathbb{H} \}$, 
converges in a certain sense to the GFF on $\HH$,
a random generalized function that we describe in more detail in Sections~\ref{sec:GFFbackground1}
and~\ref{sec:GFFbackground2}.

We will prove a similar result for the eigenvalue fluctuations of the growing
random regular graphs described in Section~\ref{sec:growingrrg}.
Our first result is about the process of \textit{short cycles} in the graph process $G(t)$. 
By a cycle of length $k$ in a graph, we mean what is sometimes called
a simple cycle: a walk in the graph that begins
and ends at the same vertex, and that otherwise repeats no vertices.
We will give a more formal definition in Section~\ref{sec:wordcombinatorics}.
Let $(C_k^{(s)}(t), \; k \in \NN)$ denote the number of cycles of various lengths $k$ that are present in $G(s+t)$. This process is not Markov, but nonetheless  it converges to a Markov process (indexed by $t$) as $s$ tends to infinity. 

To describe the limit, recall the value of $a(d,k)$, given in \eqref{eq:adk}.
Consider the set of natural numbers $\NN=\{1,2,\ldots\}$ with the measure
\[
\mu(k)= \frac{1}{2}\left[ a(d,k) - a(d,k-1)\right], \quad k\in \NN, \quad a(d,0)\defeq 0. 
\]
Consider a Poisson point process $\chi$ on $\NN \times [0, \infty)$ with an intensity measure given on $\NN \times (0, \infty)$
by the product measure $\mu \otimes \leb$, where $\leb$ is the Lebesgue measure,  and with additional masses of $a(d,k)/2k$ on $ (k,0)$ for $k\in \NN$.

Let $\yulelaw{x}$ denote the law of an one-dimensional pure-birth process on $\NN$ given by the generator:
\[
L f(k) = k\left( f(k+1) - f(k)\right), \quad k \in \NN,
\] 
starting from $x\in \NN$. This is also known as the \textit{Yule process}. 

\newcommand{\limitN}{{N}}
Suppose we are given a realization of $\chi$. For any atom $(k,y)$ of the countably many atoms of $\chi$, we start an independent process $(X_{k,y}(t),\; t\ge 0)$ with law $\yulelaw{k}$. Define the random sequence
\[
\limitN_k(t)\defeq  \sum_{(j,y)\in \chi \cap\{ [k] \times [0, t] \} } 1\left\{  X_{j,y}(t-y)=k  \right\}.
\]
In other words, at time $t$, for every site $k$, 
we count how many of the processes that started at time $y\le t$ at site $j \le k$ are currently at $k$. 
Note that both $(\limitN_k(\cdot), \; k\in \NN)$ and $(\limitN_k(\cdot), \; k\in [K])$, for some $K\in \NN$, are Markov processes, while $N_k(\cdot)$ for
fixed $k$ is not.

\begin{thm}\thlabel{mainthm:cycles}
As $s\to\infty$, the process $(C_k^{(s)}(t),\; k\in \NN,\; 0\le t < \infty)$ converges in law
 in the Skorokhod space $D_{\RR^{\infty}}[0, \infty)$ to the Markov process $(\limitN_k(t),\; k\in \NN,\; 0\le t < \infty)$. The limiting process is stationary. 
\end{thm}

\begin{rmk}
In fact, the same argument used to prove Theorem~\ref{mainthm:cycles} shows that the process $(C_k^{(s)}(t),\; -\infty < t < \infty)$ converges in law to the Markov process $(\limitN_k(t),\; -\infty < t < \infty)$ running in stationarity. The same conclusion holds for all the following theorems in this section.
\end{rmk}

We now focus on eigenvalues of $G(t)$. Note that there is no easy exact relationship between the eigenvalues of $G_n$ for various $n$ since the eigenvectors play a role in determining any such identity. In fact, the eigenvalues of $G_n$ and $G_{n+1}$ need not be interlaced.
We will follow the approach of the previous sections and consider
linear eigenvalue statistics.
For any $2d$-regular graph on $n$ vertices $G$
and function $f\colon\RR \rightarrow \RR$, we will define the random variable 
\[
\tr f(G)\defeq \sum_{i=1}^n \hat{f}(\lambda_i)
\]
where $\lambda_1\ge \ldots \ge \lambda_n$ are the eigenvalues of adjacency matrix of $G$
divided by $2\sqrt{2d-1}$, and $\hat{f}$
is $f$ with its constant term adjusted (see \thref{def:tr,rmk:tr} for an explanation). 
The scaling is necessary to take a limit with respect to $d$.    
Let $[n]=\{1,\ldots,n\}$, and let $[\infty]=\NN$.
\begin{thm}\label{mainthm:eigenvalues}
For each $d$, there exists a set of polynomials $f_1,f_2,\ldots$ with $f_i$ of degree $i$
such that for any $K\in \NN \cup \{ \infty\}$, the process 
$( \tr f_k(G(s+t)), \; k \in [K], \; t \ge 0 )$ converges in law, as $s$ tends to infinity, to the Markov process $(N_k(t),\; k\in [K],\; t \ge 0 )$ of Theorem~\ref{mainthm:cycles}.
(The polynomials are given explicitly in \eqref{eq:fbasis}.)
For any polynomial $f$, the process $\big(\tr f(G(s+t))\big)$
converges to a linear combination of the coordinate processes
of $(N_k(t),\,k\in \NN)$.
\end{thm}

Next, we take $d$ to infinity. We will make the following notational convention: for any polynomial $f$, we will denote the limiting
process of $( \tr f(G(s+t)), \; t \ge 0 )$ 
by $(\tr f\left(G(\infty + t) \right),\; t\ge 0 )$.
Recall that this process is a linear combination of $(N_k(t),\, k \in \NN,\, t\ge 0)$. 

\begin{thm}\thlabel{mainthm:chebycov}
Let $\{T_k, \; k \in \NN\}$ denote the Chebyshev orthogonal polynomials of the first kind on $[-1,1]$. 
As $d$ tends to infinity, the collection of processes
\[
\left( \tr T_{k} \left(G(\infty + t) \right) - \E \tr T_{k} \left(G(\infty + t) \right),\; t\ge 0, \; k\in \NN \right)
\]
converges weakly in $D_{\RR^{\infty}}[0,\infty)$ to a collection of independent Ornstein-Uhlenbeck processes $\left( U_k(t),\;t\ge 0,\; k\in \NN \right)$, running in equilibrium. Here the equilibrium distribution of $U_k$ is $N(0, k/2)$ and $U_k$ satisfies the stochastic differential equation
\[
d U_k(t)= - k U_k(t)\,d t + k\, d W_k(t), \quad t\ge 0,
\]
and $(W_k,\; k \in \NN )$ are i.i.d.\ standard one-dimensional Brownian motions. 

Thus, the collection of random variables
$\big(\tr T_{k} \left(G(\infty + t) \right) - \E \tr T_{k} \left(G(\infty + t) \right)\big)$,
indexed by $k$ and $t$, converges as $d$ tends to infinity
to a centered Gaussian process with covariance kernel given by
\begin{align}\thlabel{eq:covcheby}
\lim_{d\rightarrow \infty} \cov\left( \tr T_i \left(G(\infty + t) \right), \tr T_k \left(G(\infty + s) \right)  \right)= \delta_{ik} \frac{k}{2}e^{k(s-t)}.
\end{align}
for $s\leq t$.
\end{thm}

This covariance structure is intimately linked to the GFF; we will make this
more apparent in
\thref{thm:GFFconvergence}.
For the moment, this is best illustrated by a comparison to Borodin's result.
We specialize \cite[Proposition 3]{Bor1} for the case of GOE ($\beta=1$). 
 Fix $m$ positive real numbers $t_1 < t_2 < \ldots < t_m$. In the notation of \cite{Bor1}, we take $L=n$ and $B_i(n)=[\lfloor t_i n \rfloor]$. The matrix $W(n)$ is defined as the first $n$ rows and columns 
 of an infinite Wigner matrix.
 Then, for any positive integers $j_1, j_2, \ldots, j_m$, the random vector 
 \begin{align*}
\left( \tr T_{j_i}\left( W(\lfloor t_i n \rfloor)/2\sqrt{t_i n} \right) - \E  \tr T_{j_i}\left( W(\lfloor t_i n \rfloor) /2\sqrt{t_i n} \right),\; i\in [m] \right)
 \end{align*}
 converges in law as $n$ tends to infinity to a centered Gaussian vector.
 For $s\leq t$, 
\begin{align*}
\lim_{n\rightarrow \infty} \cov\left(  \tr T_i\left( W( \lfloor t n \rfloor)/ 2\sqrt{tn} \right), \; \tr T_k \left( W (\lfloor s n \rfloor ) /2\sqrt{sn}\right)  \right)=\delta_{ik} \frac{k}{2}\left(\frac{s}{t} \right)^{k/2}, 
\end{align*}
nearly the same as \eqref{eq:covcheby}. The appearance of the exponential in \eqref{eq:covcheby}
comes from the time-change we introduced when we made our graph process $G(t)$ run in continuous time.

Here, we have taken a limit in $t$ followed by a limit in $d$. When we take the limit in $t$, we get
an abstract limiting object. In order to give a direct connection between the eigenvalue fluctuations
and the GFF, we need to take the two limits simultaneously. As we now vary both $t$ and $d$,
recall the notation $G(t, d)$ from Section~\ref{sec:growingrrg}. Let $N(t)$
be the number of vertices in $G(t,d)$, which does not depend on $d$.

\begin{prop}\thlabel{thm:diagonallimit}
  There exists an increasing, right-continuous $d(t)$ taking integer values
  and growing to infinity such that as $s\to\infty$, the process
  \begin{align*}
    \Bigl(\tr T_k\bigl(G(s+t,2d(s+t))\bigr) - 
    \E\bigl[\tr T_k\bigl(G(s+t,2d(s+t))\bigr)\,\big\vert\,N(t)\bigr],\;k\in\NN,\,t\geq 0\Bigr)
  \end{align*}
  converges weakly in $D_{\RR^{\infty}}[0,\infty)$ to the same limit
  of Ornstein-Uhlenbeck processes
  $(U_k(\cdot),\,k\geq 1)$ as in \thref{mainthm:chebycov}.
\end{prop}

Now, we define a height function $F_t(x)$ as the number of eigenvalues of the adjacency matrix
of $G(t, 2d(t))$ that
are less than or equal to $2\sqrt{2d(t)-1}x$, taking $d(t)$ from the previous proposition.
\newcommand{\F}{\overline{F}}
Let $\F_t(x)$ denote the centered height function
\begin{align*}
  \F_t(x) = F_t(x) - \E[F_t(x)\mid N(t)].
\end{align*}
(We need to subtract off this conditional expectation, not just the expectation, because
otherwise the fluctuations of $N(t)$ swamp the eigenvalue fluctuations that we are interested in.)
Define
\begin{align*}
  H_s(x,t) = \sqrt{\frac{\pi}{2}} \F_{s+t}(x).
\end{align*}
As $s\to\infty$, these functions converge to the GFF in the following sense:
\begin{thm}\thlabel{thm:GFFconvergence}
  Let $\Omega(x,t) = e^t\bigl(x + i\sqrt{1-x^2}\bigr)$ for $-1\leq x\leq 1$.
  Let $h$ denote the GFF on $\HH$ with vanishing Dirichlet boundary conditions.
  For any polynomials $p_1(x),\ldots,p_n(x)\in\CC[x]$ and times $t_1,\ldots,t_n$,
  \begin{align*}
    \biggl( \int_{-\infty}^{\infty} p_i(x)H_s(x,t_i))\,dx,\;1\leq i\leq n\biggr)
    \toL
    \biggl( \int_{-1}^{1} p_i(x)h(\Omega(x,t_i))\,dx,\;1\leq i\leq n\biggr)
  \end{align*}
  as $s\to\infty$.
\end{thm}

\begin{rmk}
  A common model for random regular graphs is the \emph{configuration model} or
  \emph{pairing model} (see \cite{W} for more information). The model is
  defined as follows: Start with $n$ buckets, each containing $d$ prevertices.
  Then, separate these $dn$ prevertices into pairs, choosing uniformly
  from every possible pairing. Finally, collapse each bucket into a single vertex,
  making an edge between one vertex and another if a prevertex in one bucket is paired
  with a prevertex in the other bucket. This model has the advantage that choosing
  a graph from it conditional on it containing no loops or parallel edges is the same
  as choosing a graph uniformly from the set of graphs without loops and parallel edges.
  The model also allows for graphs of odd degrees, unlike the permutation model.
  
  It is possible to construct a process of growing random regular graphs similar
  to the one in this paper using
  a dynamic version of this model. Given some initial pairing of prevertices
  labeled $\{1,\ldots,dn\}$, 
  extend it to a random pairing of $\{1,\ldots,dn+2\}$  by the following procedure:
  Choose $X$ uniformly from $\{1,\ldots,dn+1\}$. Pair $dn+2$ with $X$.
  If $X=dn+1$, leave the other pairs unchanged; if not, pair the previous partner of
  $X$ with $dn+1$. This is an analogue of the Chinese Restaurant Process
   in the setting of random pairings,
  in that if the initial pairing is uniformly chosen, then so is the extended one.
  
  If $d$ is odd, we repeat this procedure a total of $d$ times to extend a random
  $d$-regular graph on $n$ vertices to have
  $n+2$ vertices (when $d$ is odd, the number of vertices in the graph must be even).
  When $d$ is even, repeat $d/2$ times to add one new vertex to a random graph.
  In this way, we can construct a sequence of growing random regular graphs.
  We believe that all the results of this paper hold in this model with minor changes,
  with similar proofs.
\end{rmk}

  \newcommand{\orb}{\mathop{\mathrm{Orb}}\nolimits}
\newcommand{\wchain}[1]{\widetilde{P}_{#1}}

\section{Preliminaries}

\subsection{A primer on the Chinese Restaurant Process}
\label{sec:CRPprimer}
The Chinese Restaurant Process, 
introduced by Dubins and Pitman, is a particular example of a two parameter family of stochastic processes that constructs sequentially random exchangeable partitions of the positive integers via the cyclic decomposition of a random permutation. Our short description is taken from \cite[Section 3.1]{Pit}. 

An initially empty restaurant has an unlimited number of circular tables numbered $1,2,\ldots$, each capable of seating an unlimited number of customers. Customers numbered $1,2,\ldots$ arrive one by one and are seated at the tables according to the following plan. Person $1$ sits at table $1$. For $n \ge 1$ suppose that $n$ customers have already entered the restaurant, and are seated in some arrangement, with at least one customer at each of the tables $j$ for $1\le j \le k$ (say), where $k$ is the number of tables occupied by the first $n$ customers to arrive. Let customer $n + 1$ choose with equal probability to sit at any of the following
$n + 1$ places: to the left of customer $j$ for some $1\le j \le n$, or alone at table $k+1$. 
Define $\pi^{(n)}\colon[n]\to[n]$ as the permutation 
whose cyclic decomposition is given by the tables; that is, if
after $n$ customers have entered the restaurant, customers $i$ and $j$ are 
seated at the same table, with $i$ to the left of
$j$, then $\pi^{(n)}(i)=j$, and if customer $i$ is seated alone at some table 
then $\pi^{(n)}(i)=i$. The sequence $(\pi^{(n)})$ is then a tower of random permutations
as defined in Section~\ref{sec:growingrrg}.

\subsection{Combinatorics on words}\label{sec:wordcombinatorics}
  Recall the discussion on p.~\pageref{permmodeldiscussion} on viewing
  the graph formed from independent permutations
  $\pi^{(n)}_1,\ldots,\pi^{(n)}_d$ as a directed, 
  edge-labeled
  graph. As we did there, we drop the subscripts
  and let $\pi_l=\pi^{(n)}_l$.
  We previously discussed the word formed as we walked around a cycle
  by writing
  down the label of each edge as it is traversed, putting
  $\pi_i$ or $\pi_i^{-1}$ according to the direction
  we walk over the edge.
  Now, we will treat this more rigorously.
  
    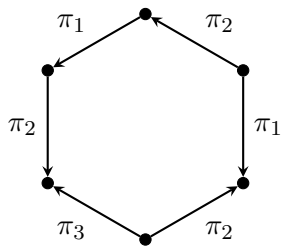
\begin{figure}
      \begin{center}
        \begin{tikzpicture}[scale=1.5,vert/.style={circle,fill,inner sep=0,
              minimum size=0.15cm,draw},>=stealth]
          \node[vert] (s0) at (270:1) {};
            \node[vert] (s1) at (330:1) {};
            \node[vert] (s2) at (30:1) {};
            \node[vert] (s3) at (90:1) {};
            \node[vert] (s4) at (150:1) {};
            \node[vert] (s5) at (210:1) {};
            \draw[thick,->] (s0) to node[auto,swap] {$\pi_2$} (s1);
            \draw[thick,<-] (s1) to node[auto,swap] {$\pi_1$} (s2);
            \draw[thick,->] (s2) to node[auto,swap] {$\pi_2$} (s3);
            \draw[thick,->] (s3) to node[auto,swap] {$\pi_1$} (s4);
            \draw[thick,->] (s4) to node[auto,swap] {$\pi_2$} (s5);
            \draw[thick,<-] (s5) to node[auto,swap] {$\pi_3$} (s0);
        \end{tikzpicture}
      \end{center}
    \caption{A cycle whose word
    is the equivalence class of
    $\pi_2\pi_1^{-1}\pi_2\pi_1\pi_2\pi_3^{-1}$ in $\Ww_6/D_{12}$.}
    \label{fig:cycleword}
  \end{figure}
    Let $\Ww_k$ denote the set
  of cyclically reduced words of length $k$.
  We would like to associate each $k$-cycle in $G_n$ 
  with the word in $\Ww_k$ formed
  by the above procedure, but since we can start the walk
  at any point in the cycle and walk in either of two directions, there
  are actually up to $2k$ different words that could be
  formed by it.  Thus we identify
  elements of $\Ww_k$ that differ only by rotation and inversion
  (for example, $\pi_1\pi_2^{-1}\pi_1\pi_2$ and 
  $\pi_1^{-1}\pi_2\pi_1^{-1}\pi_2^{-1}$) and denote the resulting
  set by $\Ww_k/D_{2k}$, where $D_{2k}$ is the dihedral group acting
  on the set $\Ww_k$ in the natural way.
  \begin{definition}[Properties of words]
     For any $k$-cycle
     in $G_n$, the element of $\Ww_k/D_{2k}$ given by walking
     around the cycle is called the \emph{word} of the cycle  (see 
     Figure~\ref{fig:cycleword}).
  For any word $w$, let $\abs{w}$ denote the length of $w$.
  Let $h(w)$ be the largest number $m$ such that
  $w=u^m$ for some word $u$.  If $h(w)=1$, we call $w$ \emph{primitive}.
  For any $w\in\Ww_k$, the orbit of $w$ under the action 
  of $D_{2k}$ contains $2k/h(w)$ elements, a fact which we will frequently use.
  Let $c(w)$ denote the number of pairs of double letters in $w$, i.e.,
  the number of integers $i$ modulo $|w|$ such that $w_i=w_{i+1}$.
  For example, $c(\pi_1\pi_1\pi_2^{-1}\pi_2^{-1}\pi_1)=3$.
  If $\abs{w}=1$, we take $c(w)=0$.
  We will also consider $|\cdot|$, $h(\cdot)$, and $c(\cdot)$ as
  functions on $\Ww_k/D_{2k}$, since they are invariant
  under cyclic rotation and inversion.
  \end{definition}
  
  To more easily refer to words in $\Ww_k/D_{2k}$, choose some
  canonical representative $w_1\cdots w_k\in\Ww_k$ for
  every $w\in\Ww_k/D_{2k}$.  Based on this, we will often
  think of elements of $\Ww_k/D_{2k}$ as words instead of equivalence
  classes, and we will make statements about the $i$th letter
  of a word in $\Ww_k/D_{2k}$.  For $w=w_1\cdots w_k\in\Ww_k/D_{2k}$,
  let $w^{(i)}$ refer
  to the word in $\Ww_{k+1}/D_{2k+2}$ given by
  $w_1\cdots w_i w_i w_{i+1}\cdots w_k$.  We refer to this
  operation as \emph{doubling} the $i$th letter of $w$.
  A related operation is to 
  \emph{halve} a pair of double letters, for example producing
  $\pi_1\pi_2\pi_3\pi_4$ from $\pi_1\pi_2\pi_3\pi_4\pi_1$. 
  (Since we apply these operations
  to words identified with their rotations, we do not need
  to be specific about which letter of the pair is deleted.)
  The following technical lemma underpins most
  of our combinatorial calculations.
  \begin{lem}\label{lem:doublehalve}
    Let $u\in\Ww_k/D_{2k}$ and $w\in\Ww_{k+1}/D_{2k+2}$.
    Suppose that $a$ letters in $u$ can be doubled to form $w$,
    and $b$ pairs of double letters in $w$ can be halved to form $u$.
    Then
    \begin{align*}
      \frac{a}{h(u)}=\frac{b}{h(w)}.
    \end{align*}
  \end{lem}
  \begin{rmk}
    At first glance, one might expect that $a=b$.  The example
    $u=\pi_1\pi_2\pi_1\pi_1\pi_2$ and $w=\pi_1\pi_1\pi_2\pi_1\pi_1\pi_2$
    shows that this is wrong, since only one letter in $u$ can
    be doubled to give $w$, but two different pairs in $w$ can
    be halved to give $u$.
  \end{rmk}
  \begin{proof}
    Let $\orb(u)$ and $\orb(w)$ denote the orbits of $u$ and $w$
    under the action of the dihedral group
    in $\Ww_k$ and $\Ww_{k+1}$, 
    respectively.  
    When we speak of halving
    a pair of letters in a word in $\orb(w)$, always delete the second of
    the two letters (for example, $\pi_1\pi_2\pi_1$ becomes $\pi_1\pi_2$, not
    $\pi_2\pi_1$).  When we double a letter in a word in $\orb(u)$,
    put the new letter after the doubled letter (for example, doubling
    the second letter of $\pi_1\pi_2^{-1}$ gives $\pi_1\pi_2^{-1}\pi_2^{-1}$,
    not $\pi_2^{-1}\pi_1\pi_2^{-1}$.)
    
    For each of the $2k/h(u)$ words in $\orb(u)$, there are
    $a$ doubling operations yielding a word in $\orb(w)$.
    For each of the $(2k+2)/h(w)$ words in $\orb(w)$, there
    are $b$ halving operations yielding a word in $\orb(u)$.
    For every halving operation on a word in $\orb(w)$,
    there is a corresponding doubling operation on a word in $\orb(u)$
    and vice versa, except for halving operations that
    straddle the ends of the word, as in $\pi_1\pi_2\pi_1$.
    There are $2b/h(w)$ of these, giving us
    \begin{align*}
      \frac{2ka}{h(u)} &= \frac{(2k+2)b}{h(w)}-\frac{2b}{h(w)}\\
        &=\frac{2kb}{h(w)},
    \end{align*}
    and the lemma follows from this.
  \end{proof}
  
  Let $\Ww'=\bigcup_{k=1}^{\infty}\Ww_k/D_{2k}$, and let
  $\Ww'_K=\bigcup_{k=1}^K\Ww_k/D_{2k}$.
  We will use the previous lemma to prove the following
  technical property of the $c(\cdot)$ statistic.
\begin{lem}\label{lem:wordcounts}
  In the vector space with basis $\{q_w\}_{w\in\Ww_K'}$,
  \begin{align*}
    \sum_{w\in\Ww_{K-1}'}\sum_{i=1}^{|w|}\frac{1}{h(w)}q_{w^{(i)}}
      &= \sum_{w\in\Ww_K'}\frac{c(w)}{h(w)}q_w.
  \end{align*}
\end{lem}
\begin{proof}
  Fix some $w\in\Ww_k/D_{2k}$, and 
  let $a(u)$ denote the number of letters of $u$ that can be doubled to give
  $w$, for any $u\in\Ww_{k-1}/D_{2k-2}$.  We need to prove that
  \begin{align*}
    \sum_{u\in\Ww_{k-1}/D_{2k-2}}\frac{a(u)}{h(u)}=\frac{c(w)}{h(w)}.
  \end{align*}
  Let $b(u)$ be the number of pairs in $w$ that can be halved to give $u$.
  By Lemma~\ref{lem:doublehalve},
  \begin{align*}
    \sum_{u\in\Ww_{k-1}/D_{2k-2}}\frac{a(u)}{h(u)} &= \sum_{u\in\Ww_{k-1}/
    D_{2k-2}}
    \frac{b(u)}{h(w)},
  \end{align*}
  and $\sum_{u\in\Ww_{k-1}/D_{2k-2}}b(u)=c(w)$.
\end{proof}

 \section{The process limit of the cycle structure}\label{sec:lim}
As the graph $G(t)$ grows, new cycles form, which we can classify into two
types.  Suppose a new vertex numbered $n$ is inserted at time
$t$, and this insertion creates a new cycle.
If the edges entering and leaving vertex
$n$ in the new cycle have the same edge label, then the new cycle has ``grown'' from a cycle with one fewer vertex, as in
Figure~\ref{fig:growth}.  If the edges entering and leaving $n$ in the cycle
have different labels, then the cycle has formed ``spontaneously'' as in
Figure~\ref{fig:spontaneous}, rather
than growing from a smaller cycle.
This classification will prove essential in understanding
the evolution of cycles in $G(t)$.
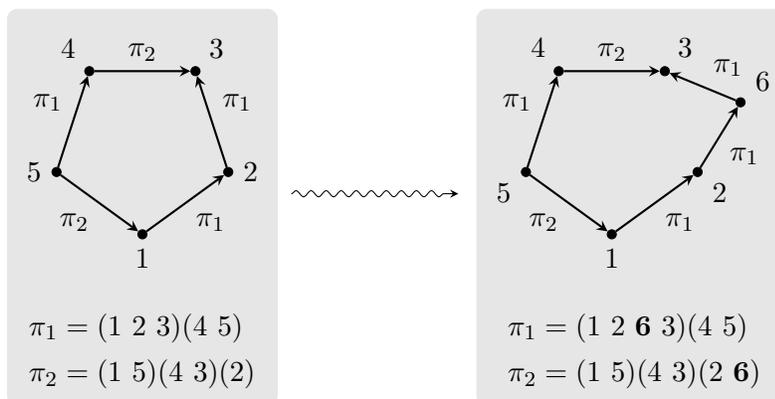
\begin{figure}
  \begin{center}
    \begin{tikzpicture}[scale=1.2,>=stealth,auto=right]
      \begin{scope}
      \fill[black!10,rounded corners] (-1.5,-2.9) rectangle (1.5,1.5);
      \path  (270:1) node[vert,label=270:$1$] (s0){}
             (342:1) node[vert,label=right:$2$] (s1){}
             (54:1) node[vert,label=above right:$3$] (s2){}
             (126:1) node[vert,label=above left:$4$] (s3){}
             (198:1) node[vert,label=left:$5$] (s4){};
      \draw[thick,->] (s0) -- node{$\pi_1$} (s1);
      \draw[thick,->] (s1) -- node{$\pi_1$} (s2);
      \draw[thick,<-] (s2) -- node{$\pi_2$} (s3);
      \draw[thick,<-] (s3) -- node{$\pi_1$} (s4);
      \draw[thick,->] (s4) -- node{$\pi_2$} (s0);
      \path (0,-2) node[text width=4cm] {
        \begin{align*}
          \pi_1&=(1\ 2\ 3)(4\ 5)\\[-6pt]
          \pi_2&=(1\ 5)(4\ 3)(2)
        \end{align*}
      };
      \end{scope}
      \draw[->, decorate, decoration={snake,amplitude=.4mm,
          segment length=2mm,post length=1.5mm}] 
          (1.65, -.55)--+(1.85,0);

      \begin{scope}[xshift=5.2cm]
      \fill[black!10,rounded corners] (-1.5,-2.9) rectangle (1.9,1.5);
      \path  (270:1) node[vert,label=270:$1$] (s0){}
             (342:1) node[vert,label=342:$2$] (s1){}
             (54:1) node[vert,label=54:$3$] (s2){}
             (126:1) node[vert,label=126:$4$] (s3){}
             (198:1) node[vert,label=198:$5$] (s4){}
             (18:1.5) node[vert,label=18:$6$] (new) {};
      \draw[thick,->] (s0) -- node{$\pi_1$} (s1);
      \draw[thick,->] (s1) -- node{$\pi_1$} (new);
      \draw[thick,->] (new) -- node{$\pi_1$} (s2);
      \draw[thick,<-] (s2) -- node{$\pi_2$} (s3);
      \draw[thick,<-] (s3) -- node{$\pi_1$} (s4);
      \draw[thick,->] (s4) -- node{$\pi_2$} (s0);
      \path (0.25,-2) node[text width=4cm] {
        \begin{align*}
          \pi_1&=(1\ 2\ \mathbf{6}\ 3)(4\ 5)\\[-6pt]
          \pi_2&=(1\ 5)(4\ 3)(2\ \mathbf{6})
        \end{align*}
      };
      \end{scope}
    \end{tikzpicture}
  \end{center}
  \caption{The vertex $6$ is inserted between
  vertices $2$ and $3$ in $\pi_1$, causing the above 
  cycle to grow.}\label{fig:growth}
\end{figure}
\begin{figure}
  \begin{center}
    \begin{tikzpicture}[scale=1.15,>=stealth,auto=right]
      \begin{scope}
        \fill[black!10,rounded corners] (-0.35,-1.9) rectangle (4.35,1.6);
        \path (0,0) node[vert,label=below:$1$] (s1) {}
              (1,0) node[vert,label=below:$2$] (s2) {}
              (2,0) node[vert,label=below:$3$] (s3) {}
              (3,0) node[vert,label=below:$4$] (s4) {}
              (4,0) node[vert,label=below:$5$] (s5) {};
        \draw[<-,thick] (s1)--node {$\pi_2$} (s2);
        \draw[->,thick] (s2)--node {$\pi_1$} (s3);
        \draw[->,thick] (s3)--node {$\pi_2$} (s4);
        \draw[<-,thick] (s4)--node {$\pi_1$} (s5);
        \path (2,-1) node[text width=4cm] {
          \begin{align*}
            \pi_1&=(2\ 3\ 1)(4\ 5)\\[-6pt]
            \pi_2&=(2\ 1\ 3\ 4\ 5)
          \end{align*}
        };
      \end{scope}
      \draw[->, decorate, decoration={snake,amplitude=.4mm,
          segment length=2mm,post length=1.5mm}] 
          (4.4, -.1)--+(1.2,0);
      \begin{scope}[xshift=6cm]
        \fill[black!10,rounded corners] (-0.35,-1.9) rectangle (4.35,1.6);
        \path (0,0) node[vert,label=below:$1$] (s1) {}
              (1,0) node[vert,label=below:$2$] (s2) {}
              (2,0) node[vert,label=below:$3$] (s3) {}
              (3,0) node[vert,label=below:$4$] (s4) {}
              (4,0) node[vert,label=below:$5$] (s5) {}
              (2,1) node[vert,label=above:$6$] (s6) {};
        \draw[<-,thick] (s1)--node {$\pi_2$} (s2);
        \draw[->,thick] (s2)--node {$\pi_3$} (s3);
        \draw[->,thick] (s3)--node {$\pi_2$} (s4);
        \draw[<-,thick] (s4)--node {$\pi_1$} (s5);
        \draw[->,thick,bend left] (s1) to node[swap] {$\pi_1$} (s6);
        \draw[->,thick,bend left] (s6) to node[swap] {$\pi_2$} (s5);
        \path (2,-1) node[text width=4cm] {
          \begin{align*}
            \pi_1&=(2\ 3\ 1\ \mathbf{6})(4\ 5)\\[-6pt]
            \pi_2&=(2\ 1\ 3\ 4\ \mathbf{6}\ 5)
          \end{align*}
        };
      \end{scope}
    \end{tikzpicture}
  \end{center}
  \caption{A cycle forms ``spontaneously'' when the vertex $6$
  is inserted into the graph.}\label{fig:spontaneous}
\end{figure}
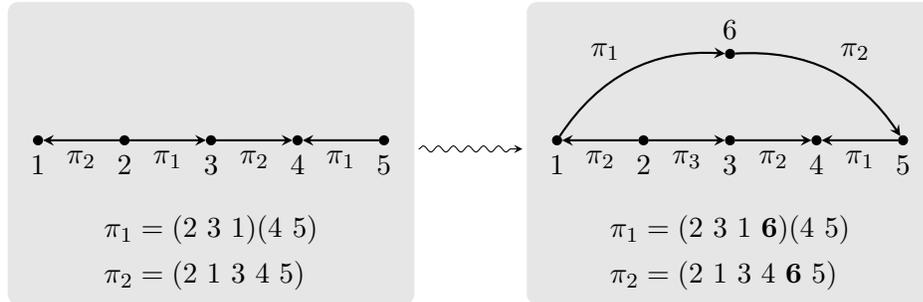

Once a cycle comes into existence in $G(t)$, it remains until a new vertex
is inserted into one of its edges.  Typically, this results in the cycle
growing to a larger cycle, as in Figure~\ref{fig:growth}.
If a new vertex is simultaneously inserted into multiple
edges of the same cycle, the cycle is instead split into 
smaller cycles
as in Figure~\ref{fig:split}.  These new cycles
are spontaneously formed, according to the classification of new cycles
given in the previous paragraph.
Tracking the evolution of these smaller cycles in turn,
we see that as the graph evolves,
a cycle grows into a cluster of overlapping cycles.
However, it will follow from
Proposition~\ref{prop:nooverlaps} that for short cycles, this behavior is not
typical.  Thus in our limiting object,
 cycles will grow only into larger cycles.\label{par:split}
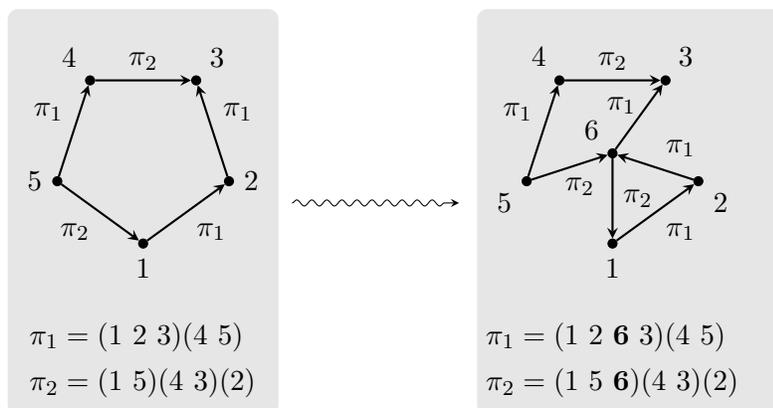
\begin{figure}
  \begin{center}
    \begin{tikzpicture}[scale=1.2,>=stealth,auto=right]
      \begin{scope}
      \fill[black!10,rounded corners] (-1.5,-2.9) rectangle (1.5,1.6);
      \path  (270:1) node[vert,label=270:$1$] (s0){}
             (342:1) node[vert,label=right:$2$] (s1){}
             (54:1) node[vert,label=above right:$3$] (s2){}
             (126:1) node[vert,label=above left:$4$] (s3){}
             (198:1) node[vert,label=left:$5$] (s4){};
      \draw[thick,->] (s0) -- node{$\pi_1$} (s1);
      \draw[thick,->] (s1) -- node{$\pi_1$} (s2);
      \draw[thick,<-] (s2) -- node{$\pi_2$} (s3);
      \draw[thick,<-] (s3) -- node{$\pi_1$} (s4);
      \draw[thick,->] (s4) -- node{$\pi_2$} (s0);
      \path (0,-2) node[text width=4cm] {
        \begin{align*}
          \pi_1&=(1\ 2\ 3)(4\ 5)\\[-6pt]
          \pi_2&=(1\ 5)(4\ 3)(2)
        \end{align*}
      };
      \end{scope}
      \draw[->, decorate, decoration={snake,amplitude=.4mm,
          segment length=2mm,post length=1.5mm}] 
          (1.65, -.55)--+(1.85,0);

      \begin{scope}[xshift=5.2cm]
      \fill[black!10,rounded corners] (-1.5,-2.9) rectangle (1.9,1.6);
      \path  (270:1) node[vert,label=270:$1$] (s0){}
             (342:1) node[vert,label=342:$2$] (s1){}
             (54:1) node[vert,label=54:$3$] (s2){}
             (126:1) node[vert,label=126:$4$] (s3){}
             (198:1) node[vert,label=198:$5$] (s4){}
             (0:0) node[vert,label=135:$6$] (new) {};
      \draw[thick,->] (s0) -- node{$\pi_1$} (s1);
      \draw[thick,->] (s1) -- node{$\pi_1$} (new);
      \draw[thick,->] (new) -- node[swap,yshift=-0.1cm,xshift=0.1cm]{$\pi_1$} (s2);
      \draw[thick,<-] (s2) -- node{$\pi_2$} (s3);
      \draw[thick,<-] (s3) -- node{$\pi_1$} (s4);
      \draw[thick,->] (s4) -- node[xshift=-0.2cm]{$\pi_2$} (new);
      \draw[thick,->] (new) -- node[swap]{$\pi_2$} (s0);
      \path (0,-2) node[text width=4cm] {
        \begin{align*}
          \pi_1&=(1\ 2\ \mathbf{6}\ 3)(4\ 5)\\[-6pt]
          \pi_2&=(1\ 5\ \mathbf{6})(4\ 3)(2)
        \end{align*}
      };
      \end{scope}
    \end{tikzpicture}
  \end{center}
  \caption{The vertex $6$ is inserted into the cycle
  in two different places in the same step, causing
  the cycle to split in two.  Note that
  each new cycle would be classified as spontaneously
  formed.}\label{fig:split}
\end{figure}

\subsection{Heuristics for the limiting process}\label{sec:heuristics}
We give some estimates that will motivate
the definition of the limiting process in Section~\ref{sec:limdef}.
This section is entirely motivational, and
we will not attempt to make anything rigorous.

Suppose that vertex $n$ is inserted into $G(t)$ at some
time $t$.
First, we consider the rate that cycles form spontaneously
with some word $w\in\Ww_k/D_{2k}$.  There are 
$2k/h(w)$ words in the orbit of $w$ under
the action of $D_{2k}$, and out of these, $2(k-c(w))/h(w)$
have nonequal first and last letters.
For each such word $u=u_1\cdots u_k$, we can
give a walk on the graph by starting at vertex $n$
and following the edges indicated by $u$, going from
$n$ to $u_1(n)$ to $u_2(u_1(n))$ and so on.
If this walk happens to be a cycle, the condition
$u_1\neq u_k$ implies that 
it would be spontaneously formed.

In a short interval $\Delta t$ when $G(t)$ has $n-1$ vertices,
the probability that vertex $n$ is inserted is about $n\,\Delta t$.
For any word $u$, 
the walk from vertex $n$ generated by $u$
is a cycle with probability approximately $1/n$, since
after applying the random permutations $u_1,\ldots,u_k$ in turn,
we will be left at an approximately uniform random vertex.
Any new spontaneous cycle formed with word $w$
will be counted by one of these walks,
with $u$ in the orbit of $w$, and it will be counted again
by the walk generated by
$u_k^{-1}\cdots u_1^{-1}$.  
The expected number of spontaneous
cycles formed in a short interval $\Delta t$ is then approximately
\begin{align*}
  \frac{1}{h(w)}(k-c(w))\frac{n\,\Delta t}{n}&=
  \frac{1}{h(w)}(k-c(w))\,\Delta t.
\end{align*}
Thus we will model the spontaneous formation of cycles with word $w$
by a Poisson process with rate $(k-c(w))/h(w)$.

Next, we consider how often a cycle with word $w\in\Ww_k$ grows into a larger
cycle.
Suppose that $G(t)$ has $n-1$ vertices, and that it contains a cycle of
the form
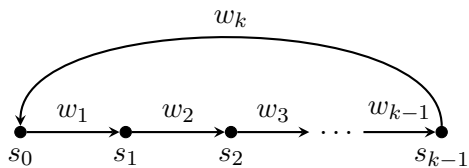
\begin{center}
  \begin{tikzpicture}[xscale=1.4,vert/.style={circle,fill,inner sep=0,
        minimum size=0.15cm,draw},>=stealth]
    \foreach \i in {0,1,2}
      \node[vert, label=below:{$s_{\i}$}] (s\i) at (\i,0) {};
    \node (s3) at (3,0) {$\ldots$};
    \node[vert, label=below:{$s_{k-1}$}] (s4) at (4,0) {};    
    \draw[thick,->] (s0) to node[auto] {$w_1$} (s1);
    \draw[thick,->] (s1) to node[auto] {$w_2$} (s2);
    \draw[thick,->] (s2) to node[auto] {$w_3$} (s3);
    \draw[thick,->] (s3) to node[auto] {$w_{k-1}$} (s4);
    \draw[thick,->,] (s4) to[out=90,in=90] node[auto,swap] {$w_k$} (s0);
  \end{tikzpicture}
\end{center}
When vertex $n$ is inserted into the graph, 
the probability that it is inserted after $s_{i-1}$ in permutation
$w_i$ is $1/n$.  Thus, after a spontaneous cycle with word $w$
has formed, we can model the evolution of its word as a 
continuous-time Markov chain
where each letter is doubled with rate one.

\subsection{Formal definition of the limiting process}\label{sec:limdef}
\newcommand{\N}[2]{N_{#1}
                          \ifthenelse{\equal{#2}{}}{}{({#2})}}
\newcommand{\bN}[2]{\overleftarrow{N}_{#1}
                          \ifthenelse{\equal{#2}{}}{}{({#2})}}
\newcommand{\bM}[2]{\overleftarrow{M}_{#1}
                          \ifthenelse{\equal{#2}{}}{}{({#2})}}
%

Consider the measure $\mu$
on $\Ww'$ given by
\[
\mu(w)= \frac{|w|-c(w)}{h(w)}.
\]
Consider a Poisson point process $\chi$ on 
$\Ww' \times[0,\infty)$ with an intensity measure given by the product measure 
$\mu \otimes \leb$, where $\leb$ refers to the Lebesgue measure. Each atom 
$(w,t)$ of $\chi$ represents a new spontaneous cycle with word $w$ formed at time $t$.

Now, we define
a continuous-time Markov chain on the countable space
$\Ww'$ governed by the following
rates: From state $w\in\Ww_k/D_{2k}$, jump with rate one to each
of the $k$ words 
in $\Ww_{k+1}/D_{2k+2}$ obtained by doubling a letter of $w$.
If a word can be formed in more than one way by doubling a letter in $w$,
then it receives a correspondingly higher rate.  For example, from
$w=\pi_1\pi_1\pi_2$, the chain jumps to $\pi_1\pi_1\pi_1\pi_2$ with rate
two and to $\pi_1\pi_1\pi_2\pi_2$ with rate one.  Let $\wchain{w}$
denote the law of this process started from $w\in\Ww'$.

Suppose we are given a realization of $\chi$. For any 
atom $(w,s)$ of the countably many atoms of $\chi$, 
we start an independent process $(X_{w,s}(t),\, t\ge 0)$ with law 
$\wchain{w}$. Define the stochastic process
\[
\N{w}{t}\defeq  \sum_{\substack{(u,s)\in \chi\\s\leq t}} 1
  \left\{\text{$X_{u,s}(t-s)=w$}  \right\}.
\]
Interpreting these processes as in the previous section, 
$\N{w}{t}$ counts the number of cycles
formed spontaneously at time $s$ that have grown to have word $w$ at time $t$.

The fact that the process exists is obvious since one can define the countably 
many independent Markov chains on a suitable product space. 
The following lemma establishes some of its key properties.

\begin{lem}\label{lem:chainprops}
Recall that $\Ww_L'=\bigcup_{k=1}^L\Ww_k/D_{2k}$.
We have the following conclusions:
\begin{enumerate}[(i)]
\item For any $L \in \NN$, the stochastic process $\{ (\N{w}{t},\,
  w\in \Ww_L'), \; t\ge 0 \}$ is a time-homogeneous Markov 
  process with respect to its natural filtration, with RCLL paths.
  \label{item:markov}

\item Recall that for 
  $w\in\Ww_k/D_{2k}$, the element $w^{(i)}\in\Ww_{k+1}/D_{2k+2}$ is
  the word formed by doubling the $i$th letter of $w$.
  The generator for the Markov process 
  $\{ (\N{w}{t},\,
  w\in \Ww_L'), \; t\ge 0 \}$ acts on $f$ at 
  $x=(x_{w},\,w\in \Ww_L')$ by
\[\begin{split}
\mathcal{L} f(x) =  &\sum_{w\in\Ww_L'} \sum_{i=1}^{|w|}x_{w} 
\left[ f(x-e_{w}+e_{w^{(i)}}) - f(x) \right]\\ &+ \sum_{w\in\Ww'_L}
 \frac{|w|-c(w)}{h(w)} \left[ f(x+e_{w}) - f(x) \right],
\end{split}\]
where $e_{w}$ is the canonical basis vector equal to one at entry $w$ and
equal to zero everywhere else.  For a word $u$ of length greater than
$L$, take $e_u=0$.\label{item:generator}
\item The product measure of $\Poi(1/h(w))$ over all $w\in \Ww_L'$ is 
  the unique invariant measure for this Markov process.\label{item:invariant}
\end{enumerate}
\end{lem}

\begin{proof} Conclusion (i) follows from construction, as does conclusion (ii). To prove conclusion (iii), we start by the fundamental identity of the Poisson distribution: if $X\eqd \Poi(\lambda)$, then for any function $f$, we have
\begin{equation}\label{eq:poiiden}
\E X g(X)= \lambda \E g(X+1).
\end{equation}

We need to show that if the coordinates of $X=(X_{w},\,w\in \Ww'_L)$ 
are independent Poisson random variables with $\E X_{w}=1/h(w)$, then
\begin{align}
  \label{eq:poiinvpf}
\E \mathcal{L}f(X)=0.
\end{align}

Since the process is an irreducible Markov chain on countable state space, the existence of one invariant distribution shows that the chain is positive recurrent and that the invariant distribution is unique.

To argue \eqref{eq:poiinvpf} we will repeatedly apply identity \eqref{eq:poiiden} to functions $g$
constructed from $f$ by keeping all but one coordinate fixed.
Thus, for any $w\in\Ww_L'$ and $1\leq i\leq|w|$,
we condition on all $X_u$ with $u\neq w$ and hold
those coordinates of $f$ fixed to obtain,
\begin{align*}
\E X_w f\left( X - e_{w} + e_{w^{(i)}} \right)&= \frac{1}{h(w)} \E 
f\left( X + e_{w^{(i)}} \right)
\end{align*}
taking $e_{w^{(i)}}=0$ when $|w|=L$.
In the same way,
\begin{align*}
\E X_{w} f\left( X \right)= \frac{1}{h(w)} \E f\left( X + e_{w} \right).
\end{align*}
By these two equalities,
\begin{align*}
  \E\sum_{w\in\Ww_L'}\sum_{i=1}^{|w|}&X_w[f(X-e_w+e_{w^{(i)}})-f(X)]\\
    &= \sum_{w\in\Ww_L'}\sum_{i=1}^{|w|}\frac{1}{h(w)}\E\big[ f(X+e_{w^{(i)}})
    -f(X+e_w)\big]\\
    &=\sum_{w\in\Ww_{L-1}'}\sum_{i=1}^{|w|}\frac{1}{h(w)}\E f(X+e_{w^{(i)}})
      + \sum_{w\in\Ww_L/D_{2L}}\frac{|w|}{h(w)}\E f(X)\\
    &\phantom{=}\quad  - \sum_{w\in\Ww'_L}\frac{|w|}{h(w)}\E f(X+e_w).
\end{align*}
Specializing Lemma~\ref{lem:wordcounts} to $q_w=\E f(X+e_w)$,
the first sum is
\begin{align*}
  \sum_{w\in\Ww_{L-1}'}\sum_{i=1}^{|w|}\frac{1}{h(w)}\E f(X+e_{w^{(i)}})
  &=\sum_{w\in\Ww'_L}\frac{c(w)}{h(w)}\E f(X+e_w),
\end{align*}
which gives us
\begin{align*}
  \E\sum_{w\in\Ww_L'}\sum_{i=1}^{|w|}&X_w[f(X-e_w+e_{w^{(i)}})-f(X)]\\
  &= \sum_{w\in\Ww'_L}\frac{c(w)-|w|}{h(w)}\E f(X+e_w)
    +\sum_{w\in\Ww_L/D_{2L}}\frac{|w|}{h(w)}\E f(X).
\end{align*}
All that remains in proving \eqref{eq:poiinvpf} is to show that
\begin{align*}
  \sum_{w\in\Ww_L'}\frac{|w|-c(w)}{h(w)}=\sum_{w\in\Ww_L/D_{2L}}
    \frac{|w|}{h(w)}.
\end{align*}
Specializing Lemma~\ref{lem:wordcounts} to $q_w=1$ shows that
$\sum_{w\in\Ww_L'}c(w)/h(w)=\sum_{w\in\Ww_{L-1}'}|w|/{h(w)}$.
Thus
\begin{align*}
  \sum_{w\in\Ww_L'}\frac{|w|-c(w)}{h(w)}&=
  \sum_{w\in\Ww_L'}\frac{|w|}{h(w)}-
  \sum_{w\in\Ww_{L-1}'}\frac{|w|}{h(w)}\\
  &=  \sum_{w\in\Ww_L/D_{2L}}\frac{|w|}{h(w)},
\end{align*}
establishing \eqref{eq:poiinvpf} and completing the proof.
\end{proof}
From now on, we will consider the process $(\N{w}{t},\,k\in\NN,\,t\geq 0)$
to be running under stationarity, i.e., with marginal distributions
given by conclusion~(\ref{item:invariant}) of the last lemma.
This process is easily constructed as described above, but with
additional point masses 
of weight $1/h(w)$ for each $w\in\Ww'$ at $(w,0)$ added to the intensity
measure of $\chi$, thus giving us the correct distribution at time zero.

\subsection{Time-reversed processes}\label{sec:reversals}
Fix some time $T>0$.  We define the time-reversal 
$\bN{w}{t}\defeq \N{w}{T-t}$ for
$0\leq t\leq T$.
\begin{lem}
  For any fixed $L\in\NN$, the process
  $\{(\bN{w}{t},\,w\in \Ww'_L),\,0\leq t\leq T\}$
  is a time-homogenous Markov process with respect to the natural
  filtration.
  A trivial modification at jump times renders RCLL paths.
  The transition rates of this chain are given
  as follows.
  Let $u\in\Ww_{k-1}/D_{2k-1}$ and $w\in\Ww_k/D_{2k}$, and
  suppose that $u$ can be obtained from $w$
  by halving $b$ different pairs.  Let $x = (x_w,\,w\in\Ww'_L)$.
  \begin{enumerate}[(i)]
    \item The chain jumps from $x$ to
        $x+e_{u}-e_{w}$ with rate $bx_{w}$.\label{item:shrink}
    \item The chain jumps from $x$ to 
        $x-e_{w}$ with rate $(k-c(w))x_{w}$.\label{item:die}
    \item If $w\in\Ww_L/D_{2L}$, then the
      chain jumps from $x$ to $x+e_{w}$ 
      with rate $L/h(w)$.\label{item:dropin}
  \end{enumerate}
\end{lem}
\begin{proof}
  Any Markov process run backwards under stationarity is Markov.
  If the chain has transition rate $r(x,y)$ from states $x$ to $y$,
  then the transition rate of the backwards chain from $x$ to $y$
  is $r(y,x)\nu(y)/\nu(x)$, where
  $\nu$ is the stationary distribution.
  We will let $\nu$ be the stationary distribution
  from Lemma~\ref{lem:chainprops}\ref{item:invariant} and
  calculate the transition rates of the backwards chain,
  using the rates given in
  Lemma~\ref{lem:chainprops}\ref{item:generator}.
  
  Let $a$ denote the number of letters in $u$ that give $w$ when
  doubled.  The transition rate of the original chain
  from $x+e_u-e_w$ to $x$ is $a(x_u+1)$, so
  the transition rate of the backwards chain
  from $x$ to $x+e_{u}-e_{w}$ is
  \begin{align*}
    a(x_u+1)\frac{\nu(x+e_{k-1,c-1}-e_{k,c})}{\nu(x)}
      &=\frac{ah(w)x_w}{h(u)},
  \end{align*}
  and this is equal to $bx_w$ by Lemma~\ref{lem:doublehalve}.
  A similar calculation shows that the transition rate from $x$ to $x-e_{w}$ is
  \begin{align*}
    \frac{(k-c(w))\nu(x-e_{w})}{h(w)\nu(x)} &=(k-c(w))x_w,
  \end{align*}
  proving (\ref{item:die}).  The transition rate from $x$ to $x+e_{w}$
  for $w\in\Ww_L/D_{2L}$
  is
  \begin{align*}
    \frac{\nu(x+e_w)}{\nu(x)}(x_{w}+1)L = \frac{L}{h(w)},
  \end{align*}
  which completes the proof.
\end{proof}
By definition,
\begin{align*}
  \bN{w}{t} = 
  \sum_{\substack{(u,s)\in \chi\\s\leq T-t}} 1\left\{\text{$X_{u,s}(T-t-s)=w$}
    \right\}.
\end{align*}
We will modify this slightly to define the process
\begin{align*}
  \bM{w}{t} \defeq 
  \sum_{\substack{(u,s)\in \chi\\s\leq T-t}} 1\left\{\text{$X_{u,s}(T-t-s)=w$
      and $|X_{u,s}(T-s)|\leq L$}
    \right\}.
\end{align*}
The idea is that $\bM{w}{t}$ is the same as $\bN{w}{t}$, except
that it does not count cycles at time~$t$ that had more than $L$ vertices
at time zero.
The process $(\bM{w}{t},\,w\in\Ww_L')$
is a Markov chain with the same transition rates as
$(\bN{w}{t},\,w\in\Ww_L')$,
except that it does not jump from $x$ to $x+e_{w}$ for $w\in\Ww_L/D_{2L}$.
These two chains also have the same initial distribution, but
$(\bM{w}{t},\,w\in\Ww_L')$ is not stationary (in fact, it is eventually
absorbed at zero).

\section{Process convergence of the cycle structure}\label{sec:processconvergence}

Recall that $\Cy[s]{k}(t)$ is the number of cycles of length $k$ in the graph
$G(s+t)$, defined on p.~\pageref{page:chaindef}.
For $w\in\Ww'$, let $\Cy[s]{w}(t)$ be the number of cycles in $G(s+t)$
with word $w$. We will prove that 
$\big(\Cy[s]{w}(\cdot),\,w\in\Ww'\big)$ converges to a distributional limit,
from which the convergence of $\big(\Cy[s]{k}(\cdot),\,k\in\NN\big)$ will follow.
The proof depends on knowing the limiting \emph{marginal} distribution of
$\Cy[s]{w}(t)$. The following corollary of \thref{thm:processapprox}
gives the facts we need:
\begin{cor}\label{cor:marginals}
  Let $\{Z_w,\,w\in\Ww'_K\}$ 
  be a family of independent Poisson random variables
  with $\E Z_w=1/h(w)$.  For any fixed integer $K$ and $d\geq 1$,
  \begin{enumerate}[(i)]
    \item as $t\to\infty$,\label{item:margC}
      \begin{align*}
        (C_w(t),\,w\in\Ww'_K)
        \toL (Z_w,\,w\in\Ww'_K);
      \end{align*}
    \item as $t\to\infty$, the probability that there exist
      two
      cycles of length $K$ or less sharing a vertex in $G(t)$
      approaches zero.\label{item:margoverlap}
  \end{enumerate}
\end{cor}
\begin{proof}
  When $d=1$, there is only one word of each length in $\Ww'_K$,
  and statement~(\ref{item:margC}) reduces to the well-known fact that the cycle
  counts of a random permutation converge to independent Poisson random
  variables (see \cite{AT} for much more on this subject).
  In this case, $G(t)$ is made up of disjoint cycles for all times $t$,
  so that statement~(\ref{item:margoverlap}) is trivially satisfied.

  When $d\geq 2$, let $\Cy[n]{w}$ be the number of cycles with word $w$ in 
  $G_n$. Observe that $\Cy[n]{w} = \sum_\alpha I_\alpha$, with $I_\alpha$ as in
  the statement of \thref{thm:processapprox} and the sum over all cycles in $\Ii$
  with word~$w$. The random variable $Z_w$ is the analogous sum over $Z_\alpha$,
  since the number of cycles in $\Ii$ with word $w$ is $[n]_k/h(w)$.
  By \thref{thm:processapprox},
  \begin{align}
    (\Cy[n]{w},\,w\in\Ww'_K)
      \toL (Z_w,\,w\in\Ww'_K).\label{eq:discretetimeconvergence}
  \end{align}
  
  Now, we just extend this to continuous time.
  The random vector
  $(C_w(t),\,w\in\Ww'_K)$ is a mixture of the random
  vectors $(\Cy[n]{w},\,w\in\Ww'_K)$ over different values of $n$.
  That is,
  \begin{align*}
    \P\left[\bigl(C_w(t),\,w\in\Ww'_K\bigr)\in A\right]
      &=\sum_{n=1}^{\infty}\P[M_t=n]\P\left[
      \bigl(\Cy[n]{w},\,w\in\Ww'_K\bigr)\in A\right]
  \end{align*}
  for any set $A$, recalling that $G(t)=G_{M_t}$.
  Equation~\eqref{eq:discretetimeconvergence} together
  with the fact that $\P[M_t>N]\to 1$ as $t\to\infty$ for
  any $N$ imply that $(C_w(t),\,w\in\Ww'_K)$ converges
  in law to $(Z_w,\,w\in\Ww'_K)$, establishing
  statement~(\ref{item:margC}).
  
  The discrete time version of statement~(\ref{item:margoverlap}) is
  given by \cite[Corollary~16]{DJPP}.
  Statement~(\ref{item:margoverlap})
  follows from it in the same way.
\end{proof}

Now, we turn to the convergence of the processes.
We will often need to transfer the convergence of a process to its limit
to the convergence of a functional of the process.
The following criterion, which we present without proof,
lets us apply the continuous mapping theorem to do so.
\begin{lem}[{\cite[Section~3.11, Exercise~14]{EK}}]\thlabel{lem:EKcont}
  Let $E$ and $F$ be metric spaces, and let $f\colon E\to F$
  be continuous. Then the mapping $x\mapsto f\circ x$ from
  $D_E[0,\infty)\to D_F[0,\infty)$ is continuous.
\end{lem}

\begin{thm}\thlabel{thm:convergence}
  The process $\big(\Cy[s]{w}(\cdot),\,w\in\Ww'\big)$ converges in law 
  as $s\to\infty$ to
  $(\N{w}{\cdot},\,w\in\Ww')$ in
  the space $D_{\RR^{\infty}}[0,\infty)$.
\end{thm}
\newcommand{\bG}[2]{\overleftarrow{G}_{#1}(#2)}
\newcommand{\bCy}[2][\infty]{\overleftarrow{C}^{(#1)}_{#2}}
\newcommand{\bX}[4][s]{\overleftarrow{M}_{#2\ifthenelse{\equal{#3}{}}{}{,#3}}^{(#1)}(#4)}
  \newcommand{\bGraph}[2]
  {\ifthenelse{\equal{#1}{}}{\overleftarrow{\Gamma}(#2)}
  {\overleftarrow{\Gamma}_{#1}(#2)}}
\begin{proof}
  The main difficulty in
  turning the intuitive ideas of Section~\ref{sec:heuristics}
  into an actual proof is that $\big(\Cy[s]{w}(t),\,w\in\Ww'\big)$
  is not Markov. 
  We now sketch how we evade this problem.
  We will run our chain backwards, defining  $\bG{s}{t}=G(s+T-t)$
  for some fixed $T>0$.  Then,
  we ignore all of $\bG{s}{0}$ except for the subgraph consisting of cycles
  of size $L$ and smaller, which we will call $\bGraph{s}{0}$.
  The graph $\bGraph{s}{t}$ is the evolution of this subgraph as time 
  runs backward, ignoring the rest of  $\bG{s}{t}$.
  Then, we consider the number of cycles with word $w$ in $\bGraph{s}{t}$,
  which we call $\phi_w(\bGraph{s}{t})$.  
Choose $K \ll L$. Then $\phi_w(\bGraph{s}{t})$ is likely to be the same as $\Cy[s]{w}(T-t)$
  for any word $w$ with $|w|\leq K$.
  The remarkable fact that makes $\phi_w(\bGraph{s}{t})$
  possible to analyze is that
  if $\bGraph{s}{0}$
  consists of disjoint cycles, then $\big(\phi_w(\bGraph{s}{t}),\,
  w\in\Ww'_L\big)$
  is a Markov chain governed by the same transition rates
  as $\big(\bM{w}{t},\,w\in\Ww'_L\big)$.
  
  Another important idea of the proof is to ignore the vertex labels
  in $\bG{s}{t}$, so that we do not know in what order the vertices
  will be removed. Thus we can view $\bG{s}{t}$ as a Markov
  chain with the following description:
  Assign each
  vertex an independent $\Exp(1)$ clock.  When the clock of vertex $v$
  goes off, remove it from the graph, and patch together the $\pi_i$-labeled
  edges entering and leaving $v$ for each $1\leq i\leq d$.  
  \step{1}
  {Definitions of $\bGraph{s}{t}$ and $\phi_w$ and analysis of
  $\big(\phi_{w}(\bGraph{s}{t}),\,w\in\Ww_L'\big)$.}

  Fix $T>0$ and define $\bG{s}{t}=G(s+T-t)$.  As mentioned above,
  we will consider
  $\bG{s}{t}$ only up to relabeling of vertices, which makes it
  a process on the countable state space consisting
  of all edge-labeled graphs on finitely many unlabeled vertices.
  With respect to its natural filtration, it is a Markov chain
  in which each vertex is removed with rate one,
  as described above.
  
  To formally define
  $\bGraph{s}{t}$, fix integers $L>K$ and
  let $\bGraph{s}{0}$ be the subgraph of $\bG{s}{0}$ made up of all
  cycles  of length $L$ or less.
  We then evolve $\bGraph{s}{t}$ in parallel with $\bG{s}{t}$.
  When a vertex $v$ is deleted from $\bG{s}{t}$, the corresponding
  vertex $v$ in $\bGraph{s}{t}$ is deleted if it is present.  If $v$
  has a $\pi_i$-labeled edge entering and leaving it in
  $\bGraph{s}{t}$, then these two edges
  are patched together. Other edges
  in $\bGraph{s}{t}$ adjacent to $v$ are deleted.
  This makes $\bGraph{s}{t}$ a subgraph of $\bG{s}{t}$,
   as well as a
  continuous-time Markov chain on the countable
  state space consisting
  of all edge-labeled graphs on finitely many unlabeled vertices.
  The transition probabilities of $\bGraph{s}{t}$ do not depend on $s$.
  
  From Corollary~\ref{cor:marginals}, we can find the limiting
  distribution of $\bGraph{s}{0}$.
  Suppose that $\gamma$ is a graph in the process's state space
  that is not a disjoint union
  of cycles.  By Corollary~\ref{cor:marginals}\ref{item:margoverlap},
  \begin{align*}
    \lim_{s\to\infty}\P[\bGraph{s}{0}=\gamma]=0.
  \end{align*}
  Suppose instead that
     $\gamma$ is made up of disjoint cycles,
  with $z_w$ cycles of word~$w$ for each $w\in\Ww'_L$.  
  By Corollary~\ref{cor:marginals}\ref{item:margC},
  \begin{align}
    \lim_{s\to\infty}\P[\bGraph{s}{0}=\gamma]=\prod_{w\in\Ww_L'}
    \P[Z_w=z_w],\label{eq:Gammadist}
  \end{align}
  where $(Z_w,\,w\in\Ww'_L)$ are independent Poisson random
  variables with $\E Z_w=1/h(w)$.
  Thus $\bGraph{s}{0}$ converges in law as $s\to\infty$ to
  a limiting distribution supported on the graphs
  made up of disjoint unions of cycles.
  For different values of $s$, the chains $\bGraph{s}{t}$ 
  differ only in their
  initial distributions, and the convergence in law
  of $\bGraph{s}{0}$ as $s\to\infty$ induces the process convergence
  of $\{\bGraph{s}{t},\,0\leq t\leq T\}$
  to a Markov
  chain $\{\bGraph{}{t},\,0\leq t\leq T\}$ with the same transition rates 
  whose initial distribution is the limit
  of $\bGraph{s}{0}$.
  
  For any finite 
  edge-labeled graph $G$, let $\phi_{w}(G)$ be the number of
  cycles in $G$ with word~$w$.
  By \thref{lem:EKcont} and the continuous mapping theorem, the process
  $(\phi_{w}(\bGraph{s}{t}),\,
  w\in\Ww_L')$ converges in law to $(\phi_{w}(\bGraph{}{t}),\,w\in\Ww_L')$
  as $s\to\infty$.
  
  We will now demonstrate that this process has the same law
  as $(\bM{w}{t},\,w\in\Ww_L')$.
  The graph $\bGraph{}{t}$ consists of disjoint cycles at time $t=0$,
  and as it evolves, these cycles shrink or are destroyed.  The process
  $(\phi_{w}(\bGraph{}{t}),\,w\in\Ww_L')$ 
  jumps exactly when a vertex in a cycle in $\bGraph{}{t}$ is deleted.
  If the deleted
  vertex lies in a cycle
  between two edges with the same label, the cycle shrinks.
  If the deleted vertex lies in a cycle between two edges with different
  labels, the cycle  is destroyed.
  The only relevant consideration in where the process
  will jump at time $t$ is the number of vertices of these
  two types in $\bGraph{}{t}$, which can be deduced
  from
  $(\phi_{w}(\bGraph{}{t}),\,w\in\Ww_L')$.  
  Thus this process is a Markov chain.
  
  Consider two words $u,w\in\Ww_K'$ such that $w$ can be
  obtained from $u$ by
  doubling a letter.  Suppose that $u$ can be obtained from
  $w$ by halving any of $b$ pairs of letters.  Suppose that
  the chain is at state
  $x=(x_v,\,v\in\Ww_L')$.  There are $bx_w$ vertices that when
  deleted cause the chain to jump from $x$ to $x-e_w+e_u$, each
  of which is removed with rate one.  Thus the chain jumps
  from $x$ to $x-e_w+e_u$ with rate $bx_w$.  Similarly, it
  jumps to $x-e_w$ with rate $(|w|-c(w))x_w$.  These are
  the same rates as the chain $(\bM{w}{t},\,w\in\Ww_L')$
  from Section~\ref{sec:reversals}.  The initial distribution
  given by \eqref{eq:Gammadist} is also the same as that
  of $(\bM{w}{t},\,w\in\Ww_L')$, demonstrating that the two processes
  $\big(\phi_{w}(\bGraph{}{t}),\,w\in\Ww_L'\big)$ 
  and $(\bM{w}{t},\,w\in\Ww_L')$ have the same law.
  \step{2}{Approximation of $\bCy[s]{w}(t)$ by $\phi_{w}(\bGraph{s}{t})$.}
  
  We will compare the two processes
  $\{(\bCy[s]{w}(t),\,w\in\Ww'_K),\,0\leq t\leq T\big\}$
  and $\{(\phi_w(\bGraph{s}{t}),\,w\in\Ww'_K),\,0\leq t\leq T\}$ 
  and
  show that for sufficiently large $L$, they are identical with 
  probability arbitrarily close to one.
  
  Consider some cycle in $\bG{s}{t}$; we can divide its vertices
  into those that lie between two edges of the cycle with different labels,
  and those that lie between two edges with the same label.
  We call this second class the \emph{shrinking vertices} of the cycle,
  because if one is deleted from $\bG{s}{t}$ as it evolves, the cycle
  shrinks.  We define $E_s(L)$ to be the event that for some cycle
  in $\bG{s}{0}$ of size $l>L$, at least $l-K$ of its shrinking vertices
  are deleted by time $T$.
    
  We claim that outside of the event $E_s(L)$, the two processes 
$\{(\bCy[s]{w}(t),\,w\in\Ww'_K),\,0\leq t\leq T\}$
  and $\{(\phi_w(\bGraph{s}{t}),\,w\in\Ww'_K),\,0\leq t\leq T\}$ 
    are identical.  Suppose that these two processes are not identical.
  Then there is some cycle $\alpha$ of size $K$ or less present in
  $\bG{s}{t}$ but not in $\bGraph{s}{t}$ for $0<t\leq T$.
  As explained in Section~\ref{sec:lim}, as a cycle evolves (in forward
  time), it grows into an overlapping cluster of cycles.  
  Thus $\bG{s}{0}$ contains some cluster of overlapping cycles
  that shrinks to $\alpha$ at time $t$.  One of the cycles in this cluster
  has length greater than $L$, or the cluster would be contained
  in $\bGraph{s}{0}$ and $\alpha$ would have been contained in $\bGraph{s}{t}$.
  
  To see that $l-K$ shrinking vertices must be deleted from this
  cycle, consider the evolution of $\alpha$ into the cluster
  of cycles in both forward and reverse time.  If a vertex is inserted into
  a single edge of a cycle in forward time, we see in reverse time the deletion
  of a shrinking vertex.  If a vertex is simultaneously inserted into
  two edges of a cycle, causing the cycle to split,
  we see in reverse time the deletion of a non-shrinking
  vertex of a cycle.  As $\alpha$ grows, a cycle of size greater than $L$ can form
  only by single-insertion of at least $l-K$ vertices into the eventual
  cycle.  In reverse time, this is seen as deletion of $l-K$ shrinking
  vertices.
  This demonstrates
  that $E_s(L)$ holds.
  
  We will now show that for any $\eps>0$, there is an $L$ sufficiently large 
  that
  $\P[E_s(L)]<\eps$ for any $s$.
  Let $w\in\Ww_l/D_{2l}$ with $l>L$, and let 
  $I\subset[l]$
  such that $\abs{I}=l-K$ and
  $w_i=w_{i-1}$ for all $i\in I$,
  considering indices modulo $l$. 
  For any cycle in $\bG{s}{0}$
  with word $l$, the set $I$ corresponds to a set of $l-K$
  shrinking vertices of the cycle.
  
  We define
  $F(w,I)$ to be the event that $\bG{s}{0}$ contains one or more
  cycles with word $w$,
  and that the vertices corresponding to $I$ in one of these
  cycles are all deleted within
  time $T$.
  By a union bound,
  \begin{align}
    \P[E_s(L)]\leq\sum_{w,I}\P[F(w,I)].\label{eq:EsL}
  \end{align}
  
  We proceed by enumerating all pairs of $w$ and $I$.
  For any pair $w, I$, deleting the letters in $w$ at positions given 
  by $I$ results in a word $u\in \Ww_K/D_{2K}$.
  For any given $u=u_1\cdots u_K\in\Ww_K/D_{2K}$, the word $w\in\Ww_l/D_{2l}$ 
  must have the form
  \begin{align*}
     w=
     \underbrace{u_1\cdots u_1}_{\text{$a_1$ times}}
     \underbrace{u_2\cdots u_2}_{\text{$a_2$ times}}\cdots\cdots 
     \underbrace{u_K\cdots
     u_K}_{\text{$a_K$ times}},
  \end{align*}
  with $a_i\geq 1$ and $a_1+\cdots+a_K=l$.
  The number of choices for $a_1,\ldots,a_K$ is $\binom{l-1}{K-1}$,
  the number of compositions of $l$ into $K$ parts, and each
  of these corresponds to a choice of $w$ and $I$.  There are fewer
  than $a(d,K)$
  choices for $u$, giving us a bound of $a(d,K)\binom{l-1}{K-1}$
  choices of pairs $w$ and $I$ for any fixed $l>L$.
  
  Next, we will show that
  for any pair $w$ and $I$ with $|w|=l$,
  \begin{align}
    \P[F(w,I)]\leq (1-e^{-T})^{l-K}.\label{eq:FuI}
  \end{align}
  Condition on $\bG{s}{0}$ having $n$ vertices.
  Consider any of the $[n]_l$ possible sequences of $l$ vertices.
  Choose some representative $w'\in\Ww_l$ of $w$.
  For each of these sequences,
  the probability that
  it forms a cycle with word $w'$ is at most $1/[n]_l$ (recall the original
  definition of our random graphs in terms of random permutations).
  Given that
  the sequence forms a cycle, the probability
  that the vertices of the cycle at positions 
  $I$ are all deleted within time $T$
  is $(1-e^{-T})^{l-K}$.  Hence
  \begin{align*}
    \P\left[F(w,I)\mid \text{$\bG{s}{0}$ has $n$ vertices}\right]
    &\leq [n]_l\frac{1}{[n]_l}(1-e^{-T})^{l-K},\\
    &\leq (1-e^{-T})^{l-K}.
  \end{align*}
  This holds for any $n$,
  establishing \eqref{eq:FuI}.
  
  Applying all of this to \eqref{eq:EsL},
  \begin{align*}
    \P[E_s(L)] &\leq \sum_{l=L+1}^{\infty}a(d,K)\binom{l-1}{K-1}(1-e^{-T})^{l-k}.
  \end{align*}
  This sum converges, which means that for any $\eps>0$,
  we have $\P[E_s(L)]<\eps$ for large enough $L$, independent of
  $s$.
  \step{3}{Approximation of $\bN{w}{t}$ by $\bM{w}{t}$.}
  
  Recall that we defined the processes
  $\{(\bM{w}{t},\,w\in\Ww'_K),\,0\leq t\leq T\}$
  and $\{(\bN{w}{t},\,w\in\Ww'_K),\,0\leq t\leq T\}$
  on the same probability space.  We will show that for sufficiently
  large $L$, the two processes
  are identical with probability arbitrarily close to one.
  
  By their definitions, these two  processes
  are identical unless one of the processes $X_{u,s}(\cdot)$ started
  at each atom of $\chi$ grows from a word of size $K$ or less to
  a word of size $L+1$ before time $T$; we call this event $E(L)$.  
  Let
  \begin{align*}
    Y=\big|\big\{(u,s)\in\chi\colon |u|\leq K,\,s\leq T\big\}\big|,
  \end{align*}
  the number of processes starting from a word of size $K$ or less
  before time $T$.
  
  Suppose that $X(\cdot)$ has law
  $\wchain{w}$ for some word $w\in\Ww_k/D_{2k}$.  
  We can choose $L$ large enough that
  $\P\big[|X(T)|>L\big]<\eps$ for all $k\leq K$.  
  Then
  $\P[E(L)\mid Y]<\eps Y$ by a union bound, and so
  $\P[E(L)]<\eps\E Y$.  Since $\E Y<\infty$, we can
  make $\P[E(L)]$ arbitrarily small by choosing sufficiently large $L$.
  \step{4}{Weak convergence of $\{(\bCy[s]{w}(t),\,w\in\Ww'_K),\,0\leq t\leq T
    \}$
   to $\{(\bN{w}{t},\,w\in\Ww'_K),\,0\leq t\leq T\}$.}
   
  If two processes are identical with probability
  $1-\eps$, then the total variation distance between their laws is at most
  $\eps$.  Thus, by steps~2 and~3, we can choose $L$ large enough that
  the laws of the processes
  $\{(\bCy[s]{w}(t),\,w\in\Ww'_K),\,0\leq t\leq T\}$ and
  $\{(\phi_w(\bGraph{s}{t},\,w\in\Ww'_K),\,
  0\leq t\leq T)\}$ are arbitrarily close in total variation distance,
  uniformly in $s$, and so that the laws of
  $\{(\bM{w}{t},\,w\in\Ww'_K),\,0\leq t\leq T\}$ and
     $\{(\bN{w}{t},\,w\in\Ww'_K),\,0\leq t\leq T\}\}$
     are arbitrarily close in total variation distance.
  Since total variation distance dominates the Prokhorov metric (or any other
  metric for the topology of weak convergence), we can
  choose $L$ such that these two pairs are each within $\eps/3$ in
  the Prokhorov metric.
  Since $\{(\phi_w(\bGraph{s}{t}),\,w\in\Ww'_K),\,0\leq t\leq T\}$
  converges in law to $\{(\bM{w}{t},\,w\in\Ww'_K),\,0\leq t\leq T\}$
  as $s\to\infty$, there is an $s_0$ such that for all $s\geq s_0$, the laws
  of these processes are within $\eps/3$ in the Prokhorov metric.
  We have thus shown that for every $\eps>0$,
  the laws of $\{(\bCy[s]{w}(t),\,w\in\Ww_K'),\,0\leq t\leq T\}$
  and
     $\{(\bN{w}{t},\,w\in\Ww'_K),\,0\leq t\leq T\}$
     are within $\eps$
  for sufficiently large $s$, which proves that the first process
  converges in law to the second in the space $D_{\RR^{\abs{\Ww'_K}}}[0,T]$
  as $s\to\infty$.
  \step{5}{Weak convergence of $\{(\Cy[s]{w}(t),\,w\in\Ww'),\,t\geq 0\}$
  to $\{(\N{w}{t},\,w\in\Ww'),\,t\geq 0\}$.}
  
    It follows
  immediately from the previous step that 
  the (not time-reversed) process
  $\{(\Cy[s]{w}(t),\,w\in\Ww'_K),\,0\leq t\leq T\}$
  converges in law to 
  $\{(\N{w}{t},\,w\in\Ww'_K),\,0\leq t\leq T\}$ for any $T>0$.
  By Theorem~16.17 in \cite{Bil}, 
  $\{(\Cy[s]{w}(t),\,w\in\Ww'_K),\,t\geq 0\}$
  converges in law to
  $\{(\N{w}{t},\,w\in\Ww'_K),\,t\geq 0\}$.
  By \cite[Section~3.11, Exercise~23]{EK}, this also also proves that
  $\{(\Cy[s]{w}(t),\,w\in\Ww'),\,t\geq 0\}$
  converges in law to
  $\{(\N{w}{t},\,w\in\Ww'),\,t\geq 0\}$.  
\end{proof}

\begin{proof}[Proof of Theorem~\ref{mainthm:cycles}]
We will express the graph cycle counts as functionals
of $\big(\Cy[s]{w}(t),\,w\in\Ww'\big)$.  The number of $k$-cycles in $G(s+t)$ is given by
$\Cy[s]{k}(t)=\sum_{w\in\Ww_k/D_{2k}}\Cy[s]{w}(t)$.
Let
\[
\N{k}{t}=\sum_{w\in\Ww_k/D_{2k}}\N{w}{t}.
 \]
 By \thref{lem:EKcont} and
the continuous mapping theorem, $\{(\Cy[s]{k}(t),\,k\in\NN),\,t\geq 0\}$
converges in law to $\{(\N{k}{t},\,k\in\NN),\,t\geq 0\}$ as $s\to\infty$.

It is not hard to see that this limit is Markov and admits the following 
representation: Cycles of size $k$ appear spontaneously with rate
$\sum_{w\in\Ww_k/D_{2k}}\mu(w)$. The size of each
cycle then grows as a pure birth process with generator
$L f(i) = i\left( f(i+1) - f(i)\right)$.
The only thing we need to verify is that
\begin{align}
  \label{eq:ratecount}
\sum_{w\in\Ww_k/D_{2k}}\mu(w)=\sum_{w\in\Ww_k/D_{2k}}\frac{k-c(w)}{h(w)}=\frac{a(d,k)-a(d,k-1)}{2}.
\end{align}
This follows from Lemma~\ref{lem:wordcounts} in the following way. From that lemma we get
\[
\sum_{w \in \Ww_k/D_{2k}} \frac{c(w)}{h(w)}= (k-1) \sum_{w\in \Ww_{k-1}/D_{2(k-1)}}\frac{1}{h(w)}. 
\]
Thus
\[
\sum_{w\in\Ww_k/D_{2k}}\mu(w)=\sum_{w \in \Ww_k/D_{2k}}\frac{k}{h(w)} -  \sum_{w\in \Ww_{k-1}/D_{2(k-1)}}\frac{k-1}{h(w)}.
\]
The two terms on the right side of the above equation are simply half the total number of cyclically reduced words possible, of size $k$ and $k-1$ respectively. The total number of cyclically reduced words of size $k$ on an alphabet of size $d$ is by definition $a(d,k)$, showing \eqref{eq:ratecount} and completing the proof. 
\end{proof}

So far, we have considered $d$ as a constant. We now view
it as a parameter of the graph and allow it to vary.  
Recall that $(\pi_d^{(n)},\,n\geq 1)$ are towers of random
permutations independent for each $d$, and that
$G(n,2d)$ is defined from $\pi_1^{(n)},\ldots,\pi_d^{(n)}$.
For each $d$, we follow the construction used to define
$G(t)$ and construct $G(t,2d)$, a continuous-time version
of $(G(n,2d),\,n\in\NN)$.  Let $\Ww'(d)$ be the set of equivalence
classes of cyclically reduced words as before, with the parameter $d$
made explicit.  Define $\Cy[s]{d,k}(t)$
as the number of $k$-cycles in $G(s+t, 2d)$ and consider
the convergence of the two-dimensional field $\{(\Cy[s]{d,k}(t),\,
d,k\in\NN),\,t\geq 0\}$ as $s\to\infty$.

Again, we will consider this process as a functional of another one.
Define $\Ww'(\infty)=\bigcup_{d=1}^{\infty}\Ww'(d)$,
noting that $\Ww'(1)\subset\Ww'(2)\subset\cdots$.
For any $w\in\Ww'(d)$, the number of cycles in $G(s+t,2d')$
with word $w$ is the same for all $d'\geq d$.
We define $\Cy[s]{w}(t)$ by this, so that
\begin{align*}
  \Cy[s]{d,k}(t)=\sum_{\substack{w\in\Ww'(d)\\|w|=k}}\Cy[s]{w}(t).
\end{align*}
Then we will
prove convergence of $\{(\Cy[s]{w}(t),\,w\in\Ww'(\infty)),\,t\geq 0\}$
as $s\to\infty$.

To define a limit for this process,
we extend $\mu$ to a measure on all of $\Ww'(\infty)$ and define
the Poisson point process $\chi$ on $\Ww'(\infty)\times[0,\infty)$.
The rest of the construction is identical to the one in
Section~\ref{sec:limdef}, giving
us random variables $\big(\N{w}{t},\,w\in\Ww'(\infty)\big)$.

\begin{thm}\thlabel{thm:jointdconvergence}
  The process $\big(\Cy[s]{w}(\cdot),\,w\in\Ww'(\infty)\big)$ converges in law 
  as $s\to\infty$ to
  $\big(\N{w}{\cdot},\,w\in\Ww'(\infty)\big)$.
\end{thm}
\begin{proof}
  For every $d$, we have shown in \thref{thm:convergence} that
  $\big(\Cy[s]{w}(\cdot),\,w\in\Ww'(d)\big)$ converges in law 
  as $s\to\infty$ to
  $(\N{w}{\cdot},\,w\in\Ww'(d))$.
  The rest of the proof then just amounts to the statement that weak convergence
  in $D_{\RR^k}[0,\infty)$ for each $k$ amounts to convergence in $D_{\RR^{\infty}}[0,\infty)$,
  just as in the very end of the proof of \thref{thm:convergence}.
\end{proof}

\begin{thm}\thlabel{mainthm:intertwining}
There is a joint process convergence of $(C_{i,k}^{(s)}(t), \; k \in \NN,\; i \in [d],\; t \ge 0)$ to a limiting process $(\limitN_{i,k}(t),\; k \in \NN,\; i\in [d], \; t\ge 0)$. This limit is a Markov process whose marginal law for every fixed $d$ is described in Theorem~\ref{mainthm:cycles}.
Moreover, for any $d\in \NN$, the process $(\limitN_{d+1, k}(\cdot) - \limitN_{d,k}(\cdot), \; k \in \NN)$ is independent of the process $(\limitN_{i,k}(\cdot),\; k \in \NN,\; i\in [d])$ and evolves as a Markov process. 
Its generator (defined on functions dependent on finitely many coordinates) is given by
\[
Lf(x)= \sum_{k=1}^\infty k x_k\left[ f\left( x + e_{k+1} - e_k  \right) - f(x)\right] + \sum_{k=1}^\infty \nu(d,k) \left[ f(x + e_k) - f(x) \right], 
\] 
where $x$ is a nonnegative sequence, $(e_k, k \in \NN)$ are the canonical orthonormal basis of $\ell^2$, and 
\[
\nu(d,k)=\frac{1}{2}\left[ a(d+1, k) - a(d+1,k-1) - a(d,k) + a(d, k-1)  \right].
\]
\end{thm}

\begin{proof}
  Let
  \begin{align*}
    \N{d,k}{t}= \sum_{\substack{w\in\Ww'(d)\\|w|=k}}\N{w}{t}.
  \end{align*}
  By \thref{lem:EKcont}, the continuous mapping theorem, and
  \thref{thm:jointdconvergence}, $(\N{d,k}{\cdot},\,d,k\in\NN)$
  is the limit of $(\Cy[s]{d,k}(\cdot),\,d,k\in\NN)$ as $s\to\infty$.
  
 Let us now describe what the limiting process is. It is obvious that $(\N{d,k}{\cdot},\; k\in \NN, \; d \in \NN)$ is jointly Markov. For every fixed $d$, the law of the corresponding marginal is given by Theorem~\ref{mainthm:cycles}. To understand the relationship across $d$, notice that cycles of size $k$ in $G(t,2(d+1))$ consist of cycles of size $k$ in $G(t,2d)$ and the extra cycles that contain an edge labeled by $\pi_{d+1}$ or $\pi_{d+1}^{-1}$. Thus
 \begin{align}
 N_{d+1, k}(t) - N_{d, k}(t)= \sum_{\substack{w\in \Ww'(d+1)\backslash \Ww'(d)\\ \abs{w}=k}} N_w(t)
 \label{eq:d+1-d}
 \end{align}
 This process is independent of 
 $(N_{i, \cdot}$, $i\in [d])$, since the set of words involved are disjoint. Moreover, the rates for this process are clearly the following: cycles of size $k$ grow at rate $k$ and new cycles of size $k$ appear at rate $[a(d+1, k) - a(d+1, k-1) - a(d,k) + a(d, k-1)]/2$. This completes the proof of the result. 
\end{proof}

\section{Process limit for linear eigenvalue statistics}
\label{sec:processlimitforlinear}
\subsection{The limiting cycle structure}
As in Section~\ref{sec:eigenfluc}, we must transfer our results from cycles
to cyclically non-backtracking walks.
  Call a cyclically non-backtracking walk \emph{bad} if it is anything other
  than a repeated walk around a cycle.
\begin{prop}\thlabel{prop:nooverlaps}
  Fix an integer $K$.  There is a random time $T$, almost surely finite,
  such that there are no bad cyclically non-backtracking walks of
  length $K$ or less in $G(t)$ for all $t\geq T$.
\end{prop}
\begin{proof}
  We will work with the discrete-time version of our process $(G_n,\,n\in\NN)$.
  We first define some machinery
  introduced in \cite{LP}.
  Consider some cyclically non-backtracking walk of length $k$ 
  on the edge-labeled complete graph $K_n$ of the form
  \begin{align*}
  \begin{tikzpicture}[baseline, auto]
    \begin{scope}[anchor=base west]
      \path (0,0) node (s0) {$s_0$}
            (s0.base east)+(1.1,0) node (s1) {$s_1$}
            (s1.base east)+(1.1,0) node (s2) {$s_2$}
            (s2.base east)+(1.1,0) node (s3) {$\cdots$}
            (s3.base east)+(1.1,0) node (s4) {$s_k=s_0.$};
    \end{scope}      
    \foreach\i [remember=\i as \lasti (initially 0)] in {1, ..., 3}
      \draw[->] (s\lasti.mid east)-- node {$w_\i$} (s\i.mid west);
    \draw[->] (s3.mid east)--node {$w_k$} (s4.mid west);
  \end{tikzpicture}
  \end{align*}
    Here, $s_i\in[n]$ and $w=w_1\cdots w_k$ is the word
    of the walk (that is, each $w_i$
    is $\pi_j$ or $\pi_j^{-1}$ for some $j$, indicating which
    permutation provided the edge for the walk).  
    We say that $G_n$ contains the walk if the random permutations
    $\pi_1,\ldots,\pi_d$ satisfy $w_i(s_{i-1})=s_i$.  In other words,
    $G_n$ contains a walk if considering both as edge-labeled directed
    graphs, the walk is a subgraph of $G_n$.
    
    If $(s'_i,\,0\leq i\leq k)$ is another walk with the same word, we say
    that the two walks are of the same \emph{category} if $s_i=s_j\iff
    s'_i=s_j'$.  In other words, two walks are of the same category
    if they are identical up to relabeling vertices.
    The probability that $G_n$ contains a walk depends only on its category.
    If a walk contains $e$ distinct edges, then
    $G_n$ contains the walk with probability at most $1/[n]_e$.
    
    Let $X_k^{(n)}$ be the number of bad walks of length $k$
    in $G_n$ that
    start at vertex $n$.
    We will first prove that with probability one, $X_k^{(n)}>0$
    for only finitely many $n$.
    Call a category bad if the walks in the category are bad.  
    Let $\Tt_{k,d}$ be the number
    of bad categories of walks of length $k$.  For any particular
    bad category whose walks contain $v$ distinct vertices, there
    are $[n-1]_{v-1}$ walks of that category whose first vertex is
    $n$.  Any bad walk contains
    more edges than vertices, so 
    \begin{align*}
      \E X_k^{(n)} \leq
        \frac{\Tt_{k,d}[n-1]_{v-1}}{[n]_{v+1}}\leq
     \frac{\Tt_{k,d}}{n(n-k)}.
     \end{align*}
     Since $X_k^{(n)}$ takes values in the nonnegative integers,
     $\P[X_k^{(n)}>0]\leq \E X_k^{(n)}$.
       By the Borel-Cantelli lemma, $X_k^{(n)}>0$ for only
       finitely many values of $n$.
       
     Thus, for any fixed $r+1$, there exists a random time $N$
     such that there are no bad walks on $G_n$ of length $r+1$ or less
     starting with vertex $n$,
     for $n\geq N$.  We claim that for $n\geq N$, there are no bad walks
     at all
     on $G_n$ with length $r$ or less.
     Suppose that $G_m$ contains some bad walk of length $k\leq r$, for some
     $m\geq N$.  As the graph evolves, it is easy to compute that with
     probability one, a new vertex is eventually inserted into an edge
     of this walk.
     But at the time $n>m\geq N$ when this occurs, $G_n$ will contain
     a bad walk of length $r+1$ or less starting with vertex $n$, a 
     contradiction.  Thus we have proven that $G_n$ eventually contains
     no bad walks of length $r$ or less.  The equivalent statement
     for the continuous-time version of the graph process follows easily
     from this.
\end{proof}

\newcommand{\nGamma}{\widehat{\Gamma}}
Define
\begin{align*}
  \nGamma_0(x) & = 1, \\
  \nGamma_{2k}(x) &= 2 T_{2k}(x) + \frac{2d-2}{(2d-1)^k}
  &\text{for $k \geq 1$,}\\
  \nGamma_{2k+1}(x) &= 2 T_{2k+1}(x)&\text{for $k \geq 0$.}
\end{align*}
Note that $\nGamma_i(x)$ is the same as $\Gamma_i(x)$ 
from Section~\ref{sec:flucresults}, except that
$x$ and $d$ are replaced by $2x$ and $2d$.
\begin{definition}\thlabel{def:tr}
  Let $G$ be a $2d$-regular graph on $n$ vertices.
  Let $f(x)$ be a polynomial expressed in the basis $\{\nGamma_i(x),\,i\geq 0\}$
  as
  \begin{align*}
    f(x) = \sum_{j=0}^k a_j\nGamma_j(x).
  \end{align*}
  We define $\tr f(G)$ as
  \begin{align*}
    \sum_{i=1}^n f(\lambda_i) - na_0,
  \end{align*}
  where $\lambda_1\geq\cdots\geq\lambda_n$ are the eigenvalues
  of the adjacency matrix of $G$ divided by $2\sqrt{2d-1}$.
\end{definition}
\begin{rmk}\thlabel{rmk:tr}
  The polynomial $f(x)-a_0$ is orthogonal to $1$ with respect to the Kesten--McKay law \eqref{eq:kmlaw},
  since $\nGamma_1(x),\nGamma_2(x),\ldots$ are orthogonal to $1$ with respect to this measure. 
  (To prove this, observe that each of these polynomials can
  be written in terms of the orthogonal polynomials of \cite[Example~5.3]{Sodin}.
  This is done in the proof of \cite[Proposition~32]{DJPP}.)
  This orthogonalization keeps $\tr f(G_n)$ of constant order when
  $n\to\infty$. One can calulate $a_0$ by integrating $f$ against the Kesten-McKay law:
  \begin{align*}
    a_0 = \int_{-2}^2 f(x)\frac{2d(2d-1)\sqrt{4-x^2}}{2\pi \bigl(4d^2 - (2d-1)x^2\bigr)}\,dx.
  \end{align*}
  The most important set of functions for us will be the Chebyshev
  polynomials. For $T_k(x)$ with $k\geq 1$,
  \begin{align*}
    a_0 &= \begin{cases}
      0 &\text{if $k$ is odd,}\\
      -\frac{d-1}{(2d-1)^{k/2}} & \text{if $k$ is even.}
    \end{cases}
  \end{align*}
\end{rmk}

       \newcommand{\BadW}[2][\infty]{B_{#2}^{(#1)}}
\begin{proof}[Proof of Theorem~\ref{mainthm:eigenvalues}]
  Let $\CNBW[s]{k}(t)$ denote the number of cyclically non-backtracking
  walks of length $k$ in $G(s+t)$.  We decompose these into
  those that are repeated walks around cycles of length $j$ for some
  $j$ dividing $k$, and the remaining bad walks, which
  we denote $\BadW[s]{k}(t)$, giving us
  \begin{align*}
    \CNBW[s]{k}(t)=\sum_{j\mid k}2j\Cy[s]{j}(t)+\BadW[s]{k}(t).
  \end{align*}
  
  Proposition~\ref{prop:nooverlaps} implies that
  \begin{align*}
    \lim_{s\to\infty}\P\big[\text{$\BadW[s]{k}(t)=0$ for all 
      $k\leq K$, $t\geq 0$}\big]=1.
  \end{align*}
  This together with \thref{lem:EKcont} and Theorem~\ref{mainthm:cycles} 
  shows that
  as $s$ tends to infinity,
  \begin{align}\label{eq:wrongbasis}
    \big(\CNBW[s]{k}(\cdot),\,1\leq k\leq K\big)\toL
    \bigg(\sum_{j\divides k}2j\N{j}{\cdot},\,1\leq k\leq K\bigg).
  \end{align}

  Now, we modify the polynomials $\nGamma_k$
  to form a new basis $\{f_k,\,k\in\NN\}$ with the right properties,
  which amounts to expressing each $\N{k}{t}$ as a linear combination
  of terms $\sum_{j\divides l}2j\N{j}{t}$.  We do this
  with the M\"obius inversion formula.  Define the polynomial
  \begin{align}
    f_k(x) = \frac{1}{2k}\sum_{j\divides k}\mu
    \left(\frac{k}{j}\right)
      (2d-1)^{j/2}\nGamma_j(x),\label{eq:fbasis}
  \end{align}
  where $\mu$ is the M\"obius function, given by
  \begin{align*}
    \mu(n) = \begin{cases}
       (-1)^a&\text{if $n$ is the product of $a$ distinct primes,}\\
       0&\text{otherwise.}
    \end{cases}
  \end{align*}
  From \thref{prop:cnbweigenvalues}, 
  \eqref{eq:wrongbasis}, and the continuous mapping
  theorem,
  \begin{align*}
    \bigl( \tr f_k(G(s+\cdot)), \, k \in [K]\bigr) \toL 
    \bigl(N_k(\cdot),\,k\in[K]\bigr)
  \end{align*}
  as desired.
  
  For an arbitrary polynomial~$f$, let $\hat{f}$ denote $f-a_0$, the orthogonalized
  version of $f$ from \thref{def:tr}. The polynomial $\hat{f}$ is a linear combination
  of $f_1,f_2,\ldots$, and so the process $\tr f(G(s+\cdot))$ converges to a linear
  combination of the coordinate processes of $(N_k(\cdot),\,k\in\NN)$.
\end{proof}
\subsection{Some properties of the limiting object}
  To prove the process convergence in \thref{mainthm:chebycov,thm:diagonallimit}, we need to know more
  about the limiting cycle process $(N_k(\cdot),\,k\in\NN)$.
  Though the limiting object is not defined in terms of graphs,
  we will nonetheless refer to $N_k(t)$ as the number of $k$-cycles at time~$t$
  in the limiting object. Similarly, if one of the Yule processes counted to define
  the limiting object increases from $j$ to $k$, we will refer to this as a cycle growing
  from size~$j$ to~$k$.

  We start our study of the limiting object
  by decomposing $N_k(t)$ into independent summands in terms
  of the process at time~$s$.
  We first give a definition related to this decomposition.
  \begin{definition}
    Let the random variable $\alpha_{s,t}(j,k)$ be
    the portion of $j$-cycles at time~$s$ that grow to be $k$-cycles
    at time~$t$ in the limiting object. When $s$ and $t$ are clear from context, we will just
    write this as $\alpha(j,k)$.
  \end{definition}
  \begin{lem}\thlabel{lem:Ealphaval}
    For $j\leq k$ and $s\leq t$,
    \begin{align}
      \E \alpha_{s,t}(j,k) = \binom{k-1}{k-j}e^{j(s-t)}\bigl(1-e^{s-t}\bigr)^{k-j}.\label{eq:Ealphaval}
    \end{align}
  \end{lem}
  \begin{proof}
    The quantity $\E \alpha_{s,t}(j,k)$ is the probability that a Yule process
    started from $j$ is at $k$ at time~$t-s$.
    It is known that this is given
    by \eqref{eq:Ealphaval} (see \cite[Exercise~2.11]{Lig}, for example),
    but we will give a proof of it anyhow.
    
    We start with the case that $j=1$, and we assume $s=0$.
    Let $X_t$ be a Yule process from $1$. We would like to show that
    \begin{align}
      \P[X_t=k] = e^{-t}\bigl(1-e^{-t}\bigr)^{k-1},\label{eq:X_tgeo}
    \end{align}
    or equivalently that $X_t-1\sim\Geo\bigl(e^{-t}\bigr)$.
    Let $S_1,S_2,\ldots$ be the holding times of the Yule process. By definition,
    they are independent, with $S_i\sim\Exp(i)$.
    Then
    \begin{align*}
      \P[X_t> k] = \P[S_1+\cdots+S_{k}\leq t].
    \end{align*}
    Now, let $\tau_1,\ldots,\tau_k$ be i.i.d.\ with distribution $\Exp(1)$,
    and consider a counting process with these $k$ points as its jump times.
    Then the first holding time is $\Exp(k)$, the next $\Exp(k-1)$, and so on.
    Thus
    \begin{align*}
      \P[S_1+\cdots+S_{k}\leq t] = \P[\tau_1,\ldots,\tau_k\leq t] = \bigl(1-e^{-t}\bigr)^k,
    \end{align*}
    which shows that $X_t-1\sim\Geo\bigl(e^{-t}\bigr)$, confirming \eqref{eq:X_tgeo}.
    
    To extend this to $j>1$, let $Y_t$ be the sum of $j$ independent Yule processes starting
    from $1$.
    This makes $Y_t$ a Yule process starting from $j$. The random variable $Y_t-j$ is a sum of independent
    $\Geo(e^{-t})$ random variables and thus is negative binomial, the distribution of the number of
    failures before $j$ successes occur in independent Bernoulli trials with a success rate of $e^{-t}$.
    Consulting \cite[eq.~VI.8.1]{Feller} for a formula for this distribution,
    \begin{align*}
      \P[Y_t-j = k-j] &= \binom{k-1}{j-1}e^{-jt}(1-e^{-t}\bigr)^{k-j},
    \end{align*}
    which matches \eqref{eq:Ealphaval} after the substitution of $t-s$ for $t$.
  \end{proof}
  
  We now give our decomposition of $N_k(t)$:
  \begin{lem}\thlabel{lem:decomp}
    Let $j\leq k$ and $s\leq t$.
    The random variable $N_k(t)$ can be decomposed
    into independent, Poisson-distributed summands as
    \begin{align}
      N_k(t) = \sum_{j=1}^k \alpha_{s,t}(j,k)N_{j}(s) + Z.\label{eq:decompN}
    \end{align}
  \end{lem}
  \begin{proof}
    All $k$-cycles at time~$t$ are either $j$-cycles at time~$s$
    that grow to size~$k$, or they are spontaneously formed.
    The random variable $\alpha_{s,t}(j,k)N_j(s)$ is the number of $j$-cycles that grow to
    size~$k$, and we define $Z$ to be the number of cycles that form spontaneously at times in $(s,t]$
    and have size~$k$ at time~$t$. We then have \eqref{eq:decompN}, and we just need to 
    to confirm that the summands are independent and Poisson.
    Cycles at time~$s$ grow independently of each other and of the spontaneously formed cycles,
    which confirms the independence.
    By Raikov's theorem on decompositions of the Poisson distribution into independent
    sums \cite[19.2A]{Loeve}, each summand is Poisson, completing the proof.
    
    This last step is needlessly slick: The random variable $\alpha_{s,t}(j,k)N_j(s)$ is a thinned
    version of a Poisson random variable and hence Poisson itself. A similar argument applies to $Z$.
  \end{proof}

Next, we compute the covariance structure of our limiting object.
\begin{prop}\thlabel{cor:cyclecov}
For any $s \leq t$ and $j,k\in \NN$,
\[
\cov(N_k(t), N_j(s)) = 
\begin{cases}
\frac{a(d,j)}{2j}\binom{k-1}{k-j}e^{j(s-t)}\bigl(1-e^{s-t}\bigr)^{k-j} & \text{if $j\leq k$,}\\
0 &\text{otherwise.}
\end{cases}
\] 
\end{prop}

\begin{proof}
   Suppose $j>k$. As $\{N_i(s),\, i\geq 1\}$ are independent,
   the decomposition~\eqref{eq:decompN} shows that $N_k(t)$ is independent of $N_j(s)$.
   Intuitively, cycles of size greater than $k$ at time~$s$ do not affect
   the cycles of size~$k$ at time~$t$.

When $j\leq k$, the result follows immediately from \thref{lem:Ealphaval,lem:decomp} by decomposing
$N_k(t)$ and taking expectations.
\end{proof}

  \subsection{The process convergence}

  \begin{lem}\thlabel{lem:prm}
    Let $\xi$ be a Poisson random measure on $[0,\infty)$
    with arbitrary $\sigma$-finite intensity measure.
    Let $T_1,T_2,\ldots$ denote the atoms of $\xi$.
    Let $(\tau_i,\,i\geq 1)$ be arbitrary nonnegative
    i.i.d.\ random variables.
    Form a new point process $\zeta$ with atoms
    $T_i+\tau_i$.
    Then $\zeta$ is also a Poisson random measure on $[0,\infty)$.
  \end{lem}
  \begin{proof}
    Let $\mu$ be the intensity measure of $\xi$, and let $P$ be the distribution
    of $\tau_i$.
    We have made $\xi$ into a marked point process, giving
    each atom $T_i$ an independent mark $\tau_i$.
    This is equivalent
    to defining $\{(T_i,\tau_i),\,i\geq 1\}$ to be the atoms
    of a Poisson random measure on $[0,\infty)^2$ with intensity
    measure $\mu\otimes P$ \cite[Corollary~VI.3.5]{Cinlar}. The point process~$\zeta$
    is a deterministic transformation of this one by the map
    $(x,y)\mapsto x+y$, and is hence also a Poisson random measure \cite[Remark~VI.2.4b]{Cinlar}.
  \end{proof}

  The following technical lemma will be used in both \thref{mainthm:chebycov,thm:diagonallimit}.
\begin{lem}\thlabel{lem:tightness}
  Fix $k$ and $T$, and consider $\bigl\{(2d-1)^{-k/2}\bigl(2kN_k(\cdot)-a(d,k)\bigr),\,d\geq 1\bigr\}$,
  a collection of processes in $D[0,T]$ indexed by $d$. This collection is tight.
\end{lem}
\begin{proof}
  Fix $d$, and  define $Y_k(t)$ as the process that starts at $0$ and increases at each point of
  increase of $N_k(t)$; define $Z_k(t)$ as the process that starts at $0$ and increases
  at each point of decrease of $N_k(t)$. Thus, we have $N_k(t)-N_k(0)=Y_k(t)-Z_k(t)$.
  As $N_k(t)$ almost surely jumps
  only by $1$ and $-1$, both $Y_k(t)$ and $Z_k(t)$ are counting processes.
  Observe that $Y_k(t)$
  counts $k$-cycles formed spontaneously or by growth in the time interval $(0,t]$,
  and $Z_k(t)$ counts $k$-cycles that jump to size~$k+1$ in the time interval $(0,t]$.
  \begin{claim} 
    The processes $Y_k(t)$ and $Z_k(t)$ are (non-independent) Poisson processes
    with rate $a(d,k)/2$.
  \end{claim}
  \begin{proof}
  We argue by induction on $k$. As our base case,
  the process $Y_1(t)$ jumps
  when $1$-cycles form spontaneously, which happen according to a Poisson process
  of rate~$a(d,1)/2$. Now, assume that $Y_k(t)$ is a Poisson process of rate~$a(d,k)/2$.
  First, we argue that $Z_k(t)$ is as well.
  Let $\xi$ be the Poisson point process whose atoms are the points of increase of $Y_k(t)$,
  with an extra $N_k(0)$ atoms at $0$. Each atom $T_i$ of $\xi$ is the time that a $k$-cycle
  forms (or $0$ if it was present from the start). Let $\tau_i$ be the amount of time after $T_i$
  that the corresponding $k$-cycle jumps to $k+1$. Then $(\tau_i,\,i\geq 1)$ are i.i.d., 
  and $T_i+\tau_i$ are the jump times of $Z_k(t)$. By \thref{lem:prm},
  $Z_k(t)$ is a (possibly inhomogeneous) Poisson process.
  By the stationarity of the limiting object, $\E N_k(t)=\E N_k(0)$, and hence
  $\E Z_k(t)=\E Y_k(t)=a(d,k)/2$, showing that $Z_k(t)$ is a homogeneous Poisson process
  with rate $a(d,k)/2$.
  
  To complete the induction, we must show that $Y_{k+1}(t)$ is a Poisson process
  of rate $a(d,k+1)/2$. To see this, observe that $Z_k(t)$ counts all $(k+1)$-cycles
  that form by growth in the time interval $(0,t]$. As $Y_{k+1}(t)$ counts all $(k+1)$-cycles
  that form by growth or spontaneously in that interval, it is the sum of $Z_k(t)$
  and an independent Poisson process of rate $\bigl(a(d,k+1)-a(d,k)\bigr)/2$. Thus
  it is a Poisson process of rate $a(d,k+1)/2$.
  \end{proof}
  Now, fix $k$ and let
   $ X_d(t) = (2d-1)^{-k/2}\bigl(2kN_k(t) - a(d,k)\bigr)$.
   We need to show that $\{X_d,\,d\geq 1\}$ is tight.
  As $N_k(t)-N_k(0) = Y_k(t)-Z_k(t)$, we have
  \begin{align*}
    \frac{1}{2k}X_d(t)  &= A_d+B_d(t) - C_d(t)
  \end{align*}
  where
  \begin{align*}
    A_d &= (2d-1)^{-k/2}\biggl(N_k(0) - \frac{a(d,k)}{2k}\biggr),\\
    B_d(t) &= (2d-1)^{-k/2}\biggl(Y_k(t)-\frac{a(d,k)t}{2}\biggr),\\
    C_d(t) &= (2d-1)^{-k/2}\biggl(Z_k(t)-\frac{a(d,k)t}{2}\biggr).
  \end{align*}
  with $Y_k(t)$ and $Z_k(t)$ implicitly depending on $d$.
  
  As $d\to\infty$, the random variable $A_d$ converges in law to Gaussian, and $B_d(t)$ and $C_d(t)$
  converge in law to Brownian motion.
  Viewing $A_d$, $B_d(t)$, and $C_d(t)$ as elements of $D[0,T]$, each thus converges
  weakly to a limit in $C[0,T]$. As tightness in a product space is equivalent to tightness
  of the marginals, the sequence $(A_d,B_d,C_d)$ is tight in $D^3[0,t]$, with all weak limit
  points lying in $C^3[0,t]$.
  
  Given a subsequence of $\{X_d(\cdot)$\}, choose a further subsequence $\{X_{d_i}(\cdot)\}$
  such that $(A_{d_i},B_{d_i},C_{d_i})$ converges. The map 
  \begin{align*}
    (x(t),y(t),z(t))\mapsto x(t) + y(t) - z(t)
  \end{align*}
  is not in general continuous from $D^3[0,T]\to D[0,T]$, but it is continuous
  at $C^3[0,T]$. (This holds because Skorokhod convergence to a continuous function implies
  uniform convergence.) By the continuous mapping theorem, $A_{d_i}+B_{d_i}-C_{d_i}$
  has a weak limit. Thus we have shown that every subsequence of $\{X_d(\cdot)\}$ has a subsequence
  with a weak limit.
\end{proof}

\begin{proof}[Proof of Theorem~\ref{mainthm:chebycov}]
  By \thref{prop:cnbweigenvalues} and \eqref{eq:wrongbasis},
  \begin{align}
2 \tr  T_k\left( G(\infty + t) \right)  = (2d-1)^{-k/2} \sum_{j\mid k} 2j \limitN_j(t).\label{eq:cheblimit}
  \end{align}
  
Now, we will prove finite-dimensional convergence to the stated Ornstein-Uhlenbeck process.
Fix $K\in\NN$ and a sequence of times $t_1<\cdots<t_n$.
We first show that the random vector
\begin{align}
  \Bigl( (2d-1)^{-k/2}\bigl(N_k(t_i)-\E N_k(t_i)\bigr),\;k\in[K],\,i\in[n]\Bigr) \label{eq:mvg}
\end{align}
converges to a multivariate Gaussian, using a slight extension of the decomposition
from \thref{lem:decomp}.
Let $\Ss$ be the set of sequences $s_1,\ldots,s_n$ with $s_i\in\{\delta\}\cup\NN$
that satisfy a certain set of conditions.
Each sequence will represent the history of a growing cycle, with $s_i$ the size of
the cycle at time $t_i$. The symbol $\delta$ will mean ``not yet born.''
Thus, a sequence is in $\Ss$ if it consists of zero or more $\delta$s followed
by a nondecreasing sequence of positive integers. We do not include the sequence of all $\delta$s
in $\Ss$.

Let $S=(s_1,\ldots,s_n)\in\Ss$ and suppose that $s_i$ is the first non-$\delta$ in the sequence.
When $i=1$, define $X_S$ as the number of cycles that have size~$s_j$ at time $t_j$
for all $1\leq j\leq n$. If $i>1$, 
define $X_S$ as the number of cycles that form spontaneously between times $t_{i-1}$
and $t_i$ and have size $s_{j}$ at time~$t_{j}$ for $j\geq i$.

We claim that $\{X_S,\,S\in\Ss\}$ is a collection of independent Poisson random variables.
The number of cycles of each size at time~$t_1$
and the number of cycles of each size at time~$t_i$ that formed after time~$t_{i-1}$
for all $2\leq i\leq n$
are independent Poissons. Each of these random variables is then thinned to form
$\{X_S,\,S\in\Ss\}$, which thus consists of independent Poissons as well.

Now, we will write \eqref{eq:mvg} in terms of this Poisson field.
First, let $\varphi(S)$ denote the first non-$\delta$ character in $S$,
and consider the normalized field
\begin{align}
  \Bigl\{(2d-1)^{-\varphi(S)/2}(X_S-\E X_S),\,S\in\Ss\Bigr\}.\label{eq:scaledpfield}
\end{align}
Fix some $S=(s_1,\ldots,s_n)\in\Ss$ with $s_i=\varphi(S)$ the first non-$\delta$ character. 
The expected number of cycles that form spontaneously between
times $t_{i-1}$ and $t_{i}$ with size~$\varphi(S)$ at time $t_i$ is $O\bigl((2d-1)^{\varphi(S)}\bigr)$
(here, we are interpreting all elements of the big-O expression as constants except for $d$).
The portion of these that grow according to $S$ is in expectation a fixed fraction of these,
with no dependence on $d$.
Thus $\E X_S = O\bigl((2d-1)^{\varphi(S)}\bigr)$.
By the Gaussian approximation to Poisson, the field~\eqref{eq:scaledpfield} converges
as $d\to\infty$ to independent Gaussians.

For each $k\in[K]$ and $i\in[n]$, we have
\begin{align*}
  (2d-1)^{-k/2}\bigl(N_k(t_i) -\E N_k(t_i)\bigr) = \sum_S (2d-1)^{-k/2}(X_S-\E X_S),
\end{align*}
where the sum ranges over all $S=(s_1,\ldots,s_n)\in\Ss$ with $s_i=k$.
Every term of the sum
with $\varphi(S)<k$ vanishes in probability, and the terms with $\varphi(S)=k$ are elements
of the field \eqref{eq:scaledpfield}. By the Gaussian convergence of \eqref{eq:scaledpfield}, the random vector \eqref{eq:mvg}
converges to Gaussian as $d\to\infty$. 

Now, consider a finite-dimensional slice of the process
\begin{align}
  \left( \tr T_{k} \left(G(\infty + t) \right) - \E \tr T_{k} \left(G(\infty + t) \right),\; t\ge 0, \; k\in \NN \right),\label{eq:process}
\end{align}
choosing finitely many choices of $k$ and $t$ and forming a random vector.
Each component has the form given by \eqref{eq:cheblimit} for some $k$ and $t$.
The scaling causes all the terms of the sum there with $j<k$ to vanish in probability.
Subtracting off these terms, we have a random vector whose components are a subset of those
of \eqref{eq:mvg}. Thus the finite-dimensional distributions of \eqref{eq:process}
converge to Gaussian as $d\to\infty$.

Next, we compute the covariances.
For a fixed $d$,  from \eqref{eq:cheblimit} we have
\begin{align}\label{eq:covt}
\cov\left( \tr T_i\left( G(\infty + t) \right), \tr T_j\left( G(\infty + s) \right)  \right)= \frac{1}{4}\left( 2d-1 \right)^{-(i+j)/2}\sum_{k\mid i, \; l \mid j} 4 lk \cov\left( \limitN_k(t), \limitN_l(s) \right)
\end{align}
for $s\leq t$.
We now fix any $i, j, t, s$ and take $d$ to infinity, using the following expression
from \thref{cor:cyclecov}:
\[
\cov(N_k(t), N_l(s)) = 
\begin{cases}
\frac{a(d,l)}{2l}\binom{k-1}{k-l} p^{l} (1-p)^{k-l}, \qquad p=e^{s-t}, \quad \text{if $k \ge l$}.\\
0, \quad \text{otherwise}.  
\end{cases}
\] 
 Any term $a(d,r)$ is asymptotically the same as $(2d-1)^r$. Thus the highest order term in $d$ on the right side of \eqref{eq:covt} is $(2d-1)^{\min(i, j)}$. Unless $i=j$, this term is negligible compared to $(2d-1)^{(i+j)/2}$. This shows that the limiting covariance is zero unless $i=j$. 
On the other hand, when $i=j$, every term on the right side of \eqref{eq:covt} vanishes, except when $k=i= l=j$. Hence,
\[
\lim_{d\rightarrow \infty}\cov\left( \tr T_i\left( G(\infty + t) \right), \tr T_i\left( G(\infty + s) \right)  \right)=\frac{1}{4}2i p^{i}= \frac{i}{2}e^{i(s-t)}.
\]
Thus we have shown convergence of the finite-dimensional distributions
to the limiting process.

To show the process convergence, we appeal to \thref{lem:tightness}.
This lemma shows that all but the highest term of the sum in \eqref{eq:cheblimit}
vanishes in probability, and the remainder is a tight sequence in $d$.
This immediately gives the convergence of
  \[
\left( \tr T_{k} \left(G(\infty + t) \right) - \E \tr T_{k} \left(G(\infty + t) \right),\; t\ge 0, \; k\in \NN \right)
\]
to the limiting process not in $D_{\RR^{\infty}}[0,\infty)$, but in $D^{\infty}[0,\infty)$.
As the limit lies in $C^{\infty}[0,\infty)$, an argument as in the end of
\thref{lem:tightness} shows that the convergence holds in $D_{\RR^{\infty}}[0,\infty)$ as well.
\end{proof}

\subsection{Diagonal convergence}

We now consider eigenvalue statistics where $d$ increases with the size of the graph.
One approach would be to give a quantitative version of
\thref{thm:convergence} that would hold even as $d$ grew, possibly with some
conditions on its growth. We have opted for something much simpler, choosing $d$ to grow
however slowly is necessary to make the convergence still hold.
The point here is more to explain what \thref{mainthm:chebycov} has to do with the GFF
than to study the graph process with $d$ growing.


\begin{proof}[Proof of \thref{thm:diagonallimit}]
  \newcommand{\bp}[2][\infty]{\Theta_{#2}^{(#1)}}
  Fix $K\in\NN$ and $T>0$, and
  let
  \begin{align*}
    \bp[s]{d}(t)=
      \Bigl(\tr T_k(G(s+t,2d)) - \E\tr T_k(G(\infty+t,2d)),\,1\leq k\leq K\bigr).
  \end{align*}
  Considering this as a random element of $D_{\RR^K}[0,T]$,
  \thref{mainthm:cycles} shows that with $d$ held fixed, $\bp[s]{d}(\cdot)$ converges weakly to a limit
  $\bp{d}(\cdot)$ described by \eqref{eq:cheblimit}.
  \thref{mainthm:chebycov} then shows that $\bp{d}(\cdot)$ converges weakly
  to a collection of independent Ornstein-Uhlenbeck processes as $d\to\infty$.
  To take a diagonal limit, we simply take $d$ to grow slowly enough
  that we can almost consider it as fixed.
  The argument will be highly technical but with little more than
  formal content.
  
  Let $\rho$ be a metric for the topology of weak convergence for probability
  measures on $D_{\RR^K}[0,T]$, and use $\rho(X,Y)$ as a shorthand for
  the distance in this metric between the laws of $X$ and $Y$.
  Recall the processes $\Cy[s]{d,k}(t)$ and $N_{d,k}(t)$ from \thref{mainthm:intertwining}.
  Also recall that $\BadW[s]{k}(t)$ is the number of bad cyclically
  non-backtracking walks of length~$k$ in $G(s+t,2d)$, and introduce
  the notation $\BadW[s]{d,k}(t)$ to indicate the dependence on $d$.
  For each $d$, choose $s_d$ large enough that for all $s\geq s_d$,
  \begin{align}
    \rho\Bigl(\bp[s]{d+1},\,\bp{d+1}\Bigr) &< \frac1d,\label{eq:sd1}\\
    \rho\Bigl( \bigl(\Cy[s]{d,k},\,\Cy[s]{d+1,k}\bigr)_{k=1}^K,\,
       \bigl(N_{d,k},\,N_{d+1,k}\bigr)_{k=1}^K\Bigr) &<\frac1d,\label{eq:sd2}\\
       \P\Bigl[\text{$\BadW[s]{d+1,k}(t)>0$ for any
           $k\leq K$, $0\leq t\leq T$}\Bigr] &<\frac1d,\label{eq:sd3}\\
    \intertext{and for all $1\leq k\leq K$,}
    \Bigl\lvert \E\tr T_k\bigl(G(\floor{e^{s/2}},2d(s+t))\bigr)
      -\E\tr T_k\bigl(G(\infty+t,2d(s+t))\bigr)\Bigr\rvert &<\frac1d.
    \label{eq:sd4}
  \end{align}
  It is possible to find $s_d$ satisfying \eqref{eq:sd1}--\eqref{eq:sd3}
  by \thref{mainthm:cycles,mainthm:intertwining}
  and \thref{prop:nooverlaps}, respectively.
  For \eqref{eq:sd4}, we clarify that $G(\floor{e^{s/2}},2d(s+t))$
  refers to the discrete-time graph defined in Section~\ref{sec:growingrrg}.
  For any fixed~$d$, one can check by a combinatorial calculation that
  $\E\tr T_k\bigl(G(n,2d)\bigr)$ converges as $n\to\infty$
  to $\E\tr T_k\bigl(G(\infty+t,2d)\bigr)$, which establishes
  that one can choose $s_d$ to satisfy \eqref{eq:sd4}.
  We can take $s_d$ and $s_{d+1}-s_d$ to be increasing sequences in $d$ by choosing larger
  values for $s_d$ if necessary. Define $d(s)$ to be the right-continuous function with $d(s)=1$
  that jumps from $i-1$ to $i$ at $s_{i}$.
  
  Our first goal is to show that $\bp[s]{d(s+t)}(t)$ converges to the limiting
  Ornstein-Uhlenbeck processes as $d\to\infty$. From \eqref{eq:sd1}
  and \thref{mainthm:chebycov}, we know that
  $\bp[s]{d(s)}(t)$ converges to this limit.
  Thus it suffices to show that the distance between
  $\bigl(\bp[s]{d(s+t)}(t),\,0\leq t\leq T\bigr)$ and 
  $\bigl(\bp[s]{d(s)}(t),\,0\leq t\leq T)$ in $D_{\RR^K}[0,T]$ vanishes
  in probability as $s\to\infty$.

  Consider the $k$th component of
  \begin{align}
    \bigl(\bp[s]{d(s+t)}(t),\,0\leq t\leq T\bigr) - 
    \bigl(\bp[s]{d(s)}(t),\,0\leq t\leq T)\label{eq:vanishingprocess}
  \end{align}
  at time~$t$, which by \thref{prop:cnbweigenvalues} is equal to
  \begin{align}
    \begin{split}
    \frac12 \bigl(2d(s+t)&-1\bigr)^{-k/2}\bigl(\CNBW[s]{d(s+t),k}(t)-
      \E \CNBW[s]{d(s+t),k}(t)\bigr)\\
      - &\frac12 \bigl(2d(s)-1\bigr)^{-k/2}\bigl(\CNBW[s]{d(s),k}(t)-
      \E \CNBW[s]{d(s),k}(t)\bigr),
    \end{split}
    \label{eq:whatvanishes}
  \end{align}
  with $\CNBW[s]{d,k}(t)$ denoting the number of cyclically non-backtracking walks
  in $G(s+t,2d)$. We will show that this vanishes in probability as $s\to\infty$.
  For sufficiently large $s$ and $0\leq t\leq T$, we have either $d(s+t)=d(s)$ or
  $d(s+t)=d(s)+1$. In the first case, \eqref{eq:whatvanishes} is $0$, so it suffices to show
  that
  \begin{align*}
    \frac12 \bigl(2d(s)&+1\bigr)^{-k/2}\bigl(\CNBW[s]{d(s)+1,k}(t)-
      \E \CNBW[s]{d(s)+1,k}(t)\bigr)\\
      &- \frac12 \bigl(2d(s)-1\bigr)^{-k/2}\bigl(\CNBW[s]{d(s),k}(t)-
      \E \CNBW[s]{d(s),k}(t)\bigr)
  \end{align*}
  vanishes in probability.
  By \eqref{eq:sd3}, the difference between this expression and
  \begin{align*}
    \frac12 \bigl(2d(s)&+1\bigr)^{-k/2}
    \sum_{j\divides k}\Bigl(2j\Cy[s]{d(s)+1,j}(t) - 2j\E \Cy[s]{d(s)+1,j}(t)\Bigr)
    \\
       &-\frac12 \bigl(2d(s)-1\bigr)^{-k/2}
      \sum_{j\divides k}\Bigl(2j\Cy[s]{d(s),j}(t) - 2j\E \Cy[s]{d(s),j}(t)\Bigr)
  \end{align*}
  converges to $0$ in probability as $s\to\infty$.
  The scaling makes all terms of the sums besides $j=k$ vanish in probability.
  Thus it sufficies to show that
  \begin{align*}
    k\bigl(2d(s)&+1\bigr)^{-k/2}
    \Bigl(\Cy[s]{d(s)+1,k}(t) - \E \Cy[s]{d(s)+1,k}(t)\Bigr)
    \\
       & -k\bigl(2d(s)-1\bigr)^{-k/2}
      \Bigl(\Cy[s]{d(s),k}(t) - \E \Cy[s]{d(s),k}(t)\Bigr)
  \end{align*}
  vanishes in probability. By \eqref{eq:sd2}, it suffices to show this for
  \begin{align*}
    \bigl(2d(s)+1\bigr)^{-k/2}
    \Bigl(N_{d(s)+1,k}(t) - \E N_{d(s)+1,k}(t)\Bigr)
        -\bigl(2d(s)-1\bigr)^{-k/2}
      \Bigl(N_{d(s),k}(t) - \E N_{d(s),k}(t)\Bigr).
  \end{align*}
  By observing that the second moment of $\bigl((2d(s)+1)^{-k/2}-(2d(s)-1)^{-k/2}\bigr)
    \bigl(N_{d(s)+1,k}(t) - \E N_{d(s)+1,k}(t)\bigr)$ vanishes, it sufficies to show this for
  \begin{align}
    \bigl(2d(s)-1\bigr)^{-k/2}
    \Bigl(N_{d(s)+1,k}(t) - N_{d(s),k}(t) -\E\bigl[ N_{d(s)+1,k}(t)
     -N_{d(s),k}(t)\bigr]\Bigr).\label{eq:laststep}
  \end{align}
  By \eqref{eq:d+1-d}, the random variable $N_{d(s)+1,k}(t) - N_{d(s),k}(t)$
  is distributed as $\Poi\bigl((a(d+1)-a(d))/2k\bigr)$, and the second moment
  of \eqref{eq:laststep} vanishes. Thus we have shown that for any $k$ and $t$,
  the expression \eqref{eq:whatvanishes} converges to $0$ in probability.
  From \eqref{eq:laststep}, we also see that each component of \eqref{eq:vanishingprocess}
  is tight. It follows from this that supremum norm of each component of
  \eqref{eq:vanishingprocess} on $[0,T]$ converges to $0$ in probability.
  This then shows that $\bp[s]{d(s+t)}(t)$ converges to the same weak limit as
  $\bp[s]{d(s)}(t)$.
  
  The next step is showing that 
  \begin{align*}
    \Bigl(\tr T_k\bigl(G(s+t,2d(s+t))\bigr) - 
    \E\bigl[\tr T_k\bigl(G(s+t,2d(s+t))\bigr)\,\big\vert\,N(t)\bigr],\;1\leq k\leq K\Bigr)
  \end{align*}
  converges to the same weak limit in $D_{\RR^K}[0,T]$ as $\bp[s]{d(s+t)}(t)$.
  The difference between the $k$th component of these two processes is 
  \begin{align*}
    \E\bigl[\tr T_k\bigl(G(s+t,2d(s+t))\bigr)\,\big\vert\,N(t)\bigr]
      - \E\tr T_k(G(\infty+t,2d)),
  \end{align*}
  and we would like to show that this vanishes in probability in the supremum norm
  as $s\to\infty$. By \eqref{eq:sd4}, it suffices to show that as $t\to\infty$,
  \begin{align}
    \P[N(t)<e^{t/2}]\to 0.\label{eq:pvanishing}
  \end{align}
  By definition of our continuous-time process, $N(t)+1$ is a Yule process starting
  from $2$. It is well known that $(N(t)+1)e^{-t}\to Z$ a.s., where $Z\sim\Exp(1)$,
  which establishes \eqref{eq:pvanishing}. (To prove this, show
  that $(N(t)+1)e^{-t}\toL Z$
   by a direct calculation, and then observe that if $Y_t$ is a Yule process,
  then $Y_te^{-t}$ is a positive martingale and hence converges a.s.)
  
  The weak convergence of the process
  \begin{align*}
    \Bigl(\tr T_k\bigl(G(s+t,2d(s+t))\bigr) - 
    \E\bigl[\tr T_k\bigl(G(s+t,2d(s+t))\bigr)\,\big\vert\,N(t)\bigr],\;k\in\NN,\,t\geq 0\Bigr)
  \end{align*}
  in $D_{\RR^K}[0,T]$ for arbitrary $K$ and $T$ gives the desired convergence
  in $D_{\RR^{\infty}}[0,\infty)$ by the same argument as at the end
  of the proof of \thref{mainthm:cycles}.
  \end{proof}

\section{Convergence to the Gaussian free field}
\label{sec:GFFconvergence}

The Gaussian free field is a generalization of Brownian motion
where the indexing set has dimension greater than one.
Physicists have long been interested in the GFF because of its importance
in quantum field theory. Mathematicians have come to the GFF more recently,
as it it became clear that it was the limit of a variety
of discrete random surfaces and height functions 
\cite{NS,GOS,Kenyon1,RV,Kenyon2,BorFer,JLS,Bor1,Kuan,Duits,Petrov}
and was closely related to Schramm-Loewner evolution 
\cite{Dubedat,SS1,SS2,MS1,MS2,MS3,MS4}.

At its most basic level, the GFF on the upper half-plane 
with zero Dirichlet boundary conditions
can be thought of as a centered Gaussian field $(h(z),\,z\in\HH)$ with covariances given by
\begin{align*}
  \E \bigl[ h(z) h(w) \bigr] = -\frac{1}{2\pi}\log\biggl\lvert\frac{z-w}{z-\overline{w}}\biggr\rvert.
\end{align*}
The problem with this definition is that no such random function $h$ exists.
If it did exist, then the collection of random variables 
$\int_{\HH} f(z) h(z)\,dz$ indexed by smooth
compactly supported functions $f$ would also be a Gaussian field. 
This field does truly exist, and we will use it to define
the GFF.

We start by giving a bare-bones treatment of the GFF that gives only the very few
properties we need. After this, we give a more languorous account based on
\cite{Sheffield}, \cite{HMP}, and \cite{Dubedat}.

\subsection{Bare-bones background on the Gaussian free field}
\label{sec:GFFbackground1}
Let $h$ denote the GFF on $\HH$ (with zero Dirichlet boundary conditions, the only kind
we will consider).
The only property we use in this thesis is that
if $f(z)$ is a smooth function defined on a smooth path~$\gamma$ satisfying
\eqref{eq:finitepath}, one can define a collection of random variables denoted
$\int_\gamma f(z)h(z)\,dz$ that form a centered Gaussian field.
(Again, $h$ is not really a function, and we are not really integrating against it.
The notation is from \cite{Bor1}, \cite{BG}, and other papers. In Section~\ref{sec:GFFbackground2}, we explain
the real definitions.)
The covariances are given by the following proposition:
\begin{prop}[{\cite[Lemma~4.6]{BG}}]\thlabel{prop:BG2}
  Let $f_1,f_2$ be smooth functions defined on the image of a smooth curve~$\gamma$ such that
  \begin{align}
    \int_\gamma\int_\gamma f_i(z)\biggl(-\frac{1}{2\pi}\log\Bigl\lvert\frac{z-w}{z-\overline{w}}
      \Bigr\rvert\biggr)f_i(w)\,dz\,dw < \infty\label{eq:finitepath}
  \end{align}
  for $i=1,2$.
  Then
  \begin{align*}
    \E\biggl[\biggl(\int_{\gamma}f_1(z)h(z)\,dz\biggr)\biggl(\int_{\gamma}f_2(z)h(z)\,dz\biggr)\biggr]
      &=\int_{\gamma}\int_{\gamma}f_1(z)\biggl(-\frac{1}{2\pi}\log\Bigl\lvert\frac{z-w}{z-\overline{w}}
      \Bigr\rvert\biggr)f_2(w)\,dz\,dw.
  \end{align*}
\end{prop}

\subsection{More background on the Gaussian free field}
\label{sec:GFFbackground2}
We will build up the GFF from scratch, mostly following \cite{Sheffield} with a sprinkling
of \cite{HMP} and \cite{Dubedat}.
Our goal will be to present it in as simply as possible and explain how it meshes
with the more concrete information from the previous section.
To make this account friendlier without bogging it down too much, we present background material
on partial differential equations and Sobolev spaces in italics.
For a proper introduction, see \cite{Evans}, \cite{Hunter}, and \cite{Brezis}.

\newenvironment{pde}{\par\noindent\begin{adjustwidth}{2.5em}{}\textit\bgroup\small\ignorespaces}{\egroup\end{adjustwidth}}
\newcommand{\grad}{\nabla}
\newcommand{\Hloc}{H_{\text{loc}}}
\newcommand{\laplace}{\Delta}
\renewcommand{\ip}[1]{\langle #1 \rangle}
\newcommand{\ipd}[1]{\langle #1 \rangle_\grad}
\newcommand{\norm}[1]{\lVert #1 \rVert}
\newcommand{\normd}[1]{\lVert #1 \rVert_\grad}
{\renewcommand{\H}{H}
\subsubsection{Definition and construction of the Gaussian free field}
Let $D\subset\RR^d$ be a domain (that is, a connected open set).
We define $H_s(D)$ as the space of all smooth, compactly supported, real-valued functions
on $D$, and we endow this space with the \emph{Dirichlet inner product}, given by
$\ipd{f,g}=\int_D\grad f(x)\cdot\grad g(x)\,dx$.
When $d=2$, this inner product is conformally invariant, meaning
that $\ipd{f\circ\varphi,g\circ\varphi}=\ipd{f,g}$ for any conformal map $\varphi$.
We denote the Hilbert space closure of $H_s(D)$ by $\H(D)$.
When $D$ is bounded, $H(D)$ is the subspace $H^1_0(D)$
of the Sobolev space $H^1(D)=W^{1,2}(D)$.
\begin{pde}
The Sobolev space $H^1(D)$ is a Hilbert space consisting of all functions
in $L^2(D)$ whose (weak or distributional) first-order derivatives are also in $L^2$.
When $D$ is bounded, 
the Dirichlet inner product on $H_s(D)$ gives a norm equivalent to the standard one in $H^1(D)$
by the Poincar\'e inequality \cite[Section~5.6.1, Theorem~3]{Evans}.
The Hilbert space completion of $H_s(D)$ is then the closure of $C_c^{\infty}(D)$
in $H^1(D)$, with an inner product equivalent to the usual Sobolev one.
This closure is denoted as $H^{1}_0(D)$, and it consists of the elements
of $H^1(D)$ that are zero on the boundary in the sense of traces \cite[Section~5.5]{Evans}.

\noindent%
When $D$ is unbounded, the situation is slightly messier, but we need to address it so that
we can talk about the GFF on regions like the upper half-plane.
To take advantage of the conformal invariance of the Dirichlet inner product,
we will assume that $D$ is an unbounded domain in $\RR^2$ that admits a conformal map
$\varphi$ onto a bounded domain $D'$.
The space $\Hloc^1(D)$ consists of all functions on $D$ whose restrictions belong to $H^1(U)$
for all open sets $U$
with compact closure in $D$. A sequence converges in $\Hloc^1(D)$ if its restrictions
converge in $H^1(U)$ for all such $U$, which makes this a Fr\'echet space.
We will show that $H(D)\subset\Hloc^1(D)$.

\noindent%
Suppose that $f_n$ forms a Cauchy sequence in $H_s(D)$.
Then $f_n\circ\varphi^{-1}$ is a Cauchy sequence in $H_s(D')$, and it converges
to a limit $g\in H^1_0(D')$. Let $f=g\circ\varphi$.
By the local invariance of Sobolev spaces under smooth
coordinate changes, $f\in\Hloc^1(D)$ and $f_n\to f$ in that space
\cite[Theorem~6.24, Corollary~6.25]{FolPDE}. By conformal invariance,
$f_n\to f$ in the Dirichlet inner product. Thus $H(D)\subset\Hloc^1(D)$.
In particular, elements of $H(D)$ are locally $L^2$-integrable.
\end{pde}
Note that by integration by parts,
the Dirichlet inner product on $H_s(D)$ can be expressed in terms of the usual
inner product in $L^2$ by
\begin{align}
  \ipd{f,g} = \ip{f,-\laplace g}. \label{eq:ips}
\end{align}

Suppose we have a probability space $(\Omega, \Ff, P)$.
A closed subspace of $L^2(\Omega, \Ff, P)$ consisting
of centered Gaussian random variables is called
a \emph{Gaussian Hilbert space}.
We will assume throughout that $\Ff$ is the $\sigma$-algebra
generated by these random variables. 
A trivial example of a Gaussian Hilbert space is the one-dimensional space
$\{t\xi,\, t\in\RR\}$, where $\xi$ is a centered Gaussian.
A non-trivial one is the closed linear span of the collection of random
variables $\{B_t,\,t\geq 0\}$, where $B_t$ is a standard Brownian motion.
The definition and both examples can be found in much more detail in
\cite{JanGH}.

We are now ready to define the GFF, though it will take some work afterwards
to make sense of it. In the following definition, $h$ has no meaning on its own.
For each $f\in \H(D)$, the notation $\ipd{h,f}$ indicates a random variable,
with no assumptions at all on the map $f\mapsto\ip{h,f}$.
\begin{definition}\thlabel{def:GFF}
  The Gaussian free field on a domain $D$ (with zero Dirichlet boundary conditions)
  is the Gaussian Hilbert space of random variables 
  $\bigl\{\ipd{h, f},\,f\in \H(D)\bigr\}$ with covariances given by
  \begin{align}
    \E\bigl[ \ipd{h,f}\ipd{h,g}\bigr] = \ipd{f,g}.\label{eq:ippreserving}
  \end{align}
\end{definition}
The notation $\ipd{h,f}$ suggests that the map $f\mapsto\ipd{h,f}$ should be linear,
and this definition implies that it is: By applying \eqref{eq:ippreserving}, we can show
 that the variance of $\ipd{h,af+bg}-\bigl(a\ipd{h,f}+b\ipd{h,g}\bigr)$ is zero.

By the monotone class lemma, the law of $\bigl\{\ipd{h, f},\,f\in \H(D)\bigr\}$ is determined by
the finite-dimensional distributions; see \cite[Example~A.3]{JanGH}. 
This is where we use the assumption that the $\sigma$-algebra associated with a Gaussian Hilbert
space is the smallest one that makes $\ipd{h,f}$ measurable for all $f\in \H(D)$.
Thus the definition
determines at most one family $\{\ipd{h,f},\,f\in \H(D)\}$ in law.
It is not clear, however, that there even exists such a Gaussian Hilbert space at all.
We resolve this by constructing one:
\begin{prop}\thlabel{prop:GFFexists}
  There exists a Gaussian Hilbert space satisfying \thref{def:GFF}.
\end{prop}
\begin{proof}
  Let $\{f_i,\,i\in\NN\}$ be an ordered orthonormal basis for $\H(D)$ (this space
   is separable and hence has a countable orthonormal
  basis). Let $\{\alpha_i,\,i\in\NN\}$ be independent standard Gaussians.
  For any $f\in \H(D)$ with expansion $f=\sum\beta_i f_i$, we define
  \begin{align}
    \ipd{h,f} = \lim_{k\to\infty}\sum_{i=1}^{k} \beta_i\alpha_i.\label{eq:hfdef}
  \end{align}
  The sum is a martingale bounded in $L^2$ by Parseval's equality and hence converges~a.s.\ 
  and in $L^2$.
  Note that it was necessary to fix an order for the sum, as the sequence need not be absolutely
  summable.
  Thus we have constructed a Gaussian field $\bigl\{\ipd{h,f},\,f\in \H(D)\bigr\}$.
  If $f=\sum\beta_if_i$ and $g=\sum\gamma_if_i$, then it follows from the $L^2$ convergence
  of \eqref{eq:hfdef} that
  \begin{align*}
    \lim_{n\to\infty}\E\biggl[\biggl(\sum_{i=1}^n \beta_i \alpha_i\biggr)
  \biggl(\sum_{i=1}^n \gamma_i \alpha_i\biggr)\biggr]=
    \E\bigl[ \ipd{h,f}\ipd{h,g}\bigr]
  \end{align*}
  Thus
  \begin{align*}
    \E\bigl[ \ipd{h,f}\ipd{h,g}\bigr]
     &= \sum_{i=1}^{\infty}\beta_i\gamma_i = \ipd{f,g}
  \end{align*}
  as desired.
\end{proof}

\subsubsection{An example}
We have defined and constructed the GFF without developing much of an intuition for it.
We show now that the GFF on $D=(0,\infty)$ is Brownian motion.
More precisely, let $B_t$ be a standard Brownian motion and define
$\ip{h, f}=\int_0^{\infty} f(t)B_t\,dt$
for $f\in H_s(D)$.
Then define $\ipd{h,f}=-\ip{h, f''}$ in analogy with \eqref{eq:ips}.
We confirm that this (or rather, its extension to all $f\in\H(D)$)
is the GFF according to \thref{def:GFF}.
For $f,g\in H_s(D)$,
\begin{align*}
  \E\bigl[\ipd{h,f}\ipd{h,g}\bigr] &=
    \E \int_0^{\infty}\int_0^{\infty} B_t f''(t)B_u g''(u)\,du\,dt\\
    &= \int_0^{\infty}\int_0^{\infty}f''(t)g''(u)\min(u,t)\,du\,dt\\
    &= \int_0^{\infty}\biggl( f''(t)\int_0^tug''(u)\,du
      + tf''(t)\int_t^{\infty}g''(u)\,du\biggr)dt\\
    &= \int_0^{\infty} \Bigl(f''(t)\bigl(tg'(t)-g(t)\bigr) - f''(t)tg'(t)\Bigr)dt\\
    &= -\int_0^{\infty}f''(t)g(t) = \ipd{f,g}.
\end{align*}

\subsubsection{Green's functions and an alternate form of the GFF}

The GFF can be written in an alternate form inspired by \eqref{eq:ips}.
Let $H(D)^*$ denote the dual space of $H(D)$, considered as a space of distributions,
and denote the action of $f\in H(D)^*$ on $g\in H(D)$ by $\ip{f,g}$.
\begin{pde}
  When $D$ is bounded and hence $H(D)=H^1_0(D)$, the space $H(D)^*$
  has a well-known characterization.
  Though Hilbert spaces are self-dual, we can instead view the dual space of
  $H^1_0(D)$ as a space of distributions. Viewed in this way, the dual space
  is denoted $H^{-1}(D)$. It consists of all sums of $L^2$-functions (viewed as distributions)
  and first-order distributional derivatives of $L^2$-functions \cite[Proposition~9.20]{Brezis}.
  When $f\in H^{-1}(D)\cap L^2(D)$, the distributional action of $f$ coincides with the $L^2$
  inner product; that is, for $\phi\in H^1_0(D)$,
  we have
    $\ip{f,\phi}=\int_D f\phi$.
\end{pde}
\begin{definition}[The GFF indexed by $H(D)^*$]
  Let $f\in H(D)^*$. By the self-duality of Hilbert spaces,
  there exists $u\in H(D)$ such that $\ip{f,\phi}=\ipd{u,\phi}$
  for all $\phi\in H(D)$.
  We define $\ip{h, f}=\ipd{h,u}$.
\end{definition}
The significance of this definition is as follows.
Suppose $f\in C^{\infty}_c(D)$, and we view it as an element of $H(D)^*$.
Then the function $u\in H(D)$ associated with it solves the partial differential
equation $-\laplace u = f$,
and we have
\begin{align*}
  \ip{h,-\laplace u} = \ipd{h,u},
\end{align*}
as in \eqref{eq:ips}.

This version of the GFF also lends some insight on why the GFF in dimensions two and higher
cannot be represented as a random function. Dirac~$\delta$-measures are elements
of $H^{-1}(D)$ when $d=1$ but not when $d\geq 2$. Thus it makes sense to evaluate
$h$ at a single point~$x$ by $\ip{h,\delta_x}$ only in the one-dimensional case.

\begin{rmk}
  The GFF can also be constructed as a random element of $H^{-\eps}(D)$ for any
  $\eps>0$;
  see \cite[p.~7]{HMP} and \cite[Proposition~2.7, Remark~2.8]{Sheffield} for more details.
  The basic idea is to take
  $\{f_i\}$ and $\{\alpha_i\}$ as in \thref{prop:GFFexists}
  and define
  \begin{align*}
    h = \sum_{i=1}^{\infty} \alpha_i f_i,
  \end{align*}
  which converges a.s.\ in $H^{-\eps}(D)$.
  This defines $\ip{h,f}$ for $f\in C_c^{\infty}(D)$ and
   coincides with our definition of $\ip{h,f}$.
\end{rmk}

The covariances of the Gaussian field $\{\ip{h,f},\,f\in H(D)^*\}$ have a nice
expression in terms of the \emph{Green's function} for the Laplacian operator
on $D$.

\begin{pde}
  The Green's function $G(x,y)$ for the operator $-\laplace$ 
  on a region~$D$ with Dirichlet boundary conditions
  is a solution to
  $-\laplace G(x,\cdot)=\delta_x$ (in the distributional sense)
  that satisfies $G(x,y)=0$ if $x\in\partial D$
  or $y\in\partial D$. The Green's function in general exists and is unique when $D$
  is bounded with $C^1$ boundary. The Green's function for the upper half-plane
  also exists and can be given explicitly:
  \begin{align*}
    G(x,y) &= -\frac{1}{2\pi}\log\biggl\lvert\frac{x - y}{x-\overline{y}}\biggr\rvert,
  \end{align*}
  thinking of $x$ and $y$ as complex.
  If $f\in H_s(D)$, then $u(x) = \int_D G(x,y)f(y)\,dy$
  is in $H(D)$ and
  satisfies $-\laplace u = f$. The equivalent statement holds
  for $u(x) = \int_D G(x,y)\mu(dy)$   if $\mu\in H(D)^*$ 
  is a locally finite measure with compact support in $D$.
  See \cite[Chapter~2]{FolPDE} for a reference on Green's functions and related ideas.
\end{pde}
Let $G$ be the Green's function for $-\laplace$ on $D$ with Dirichlet boundary conditions,
and 
let $\laplace^{-1}f(x)\defeq -\int_d G(x,y)f(y)\,dy$.
For $f,g\in H_s(D)$,
\begin{align}
  \E\bigl[\ip{h,f}\ip{h,g}\bigr]
    &= \E\bigl[\ipd{h,-\laplace^{-1}f}\ipd{h,-\laplace^{-1}g}\bigr]\nonumber\\
    &= \ipd{-\laplace^{-1}f,-\laplace^{-1}g}\nonumber\\
    &=\ip{f, -\laplace^{-1}g} = \int_D\int_D f(x)G(x,y)g(y)\,dy\,dx.
\end{align}
Similarly, if $\mu,\nu\in H(D)^*$ are locally finite, compactly supported measures, then
\begin{align}
  \E\bigl[\ip{h,\mu}\ip{h,\nu}\bigr] &= \int_D\int_D G(x,y)\mu(dx)\nu(dy).\label{eq:measuregreen}
\end{align}


\subsubsection{Traces}
In this section, we explain how to define $\ip{h,\mu}$ when $\mu$ is a measure
supported on a curve~$\gamma$ in $\overline{D}$, which along with \eqref{eq:measuregreen}
explains \thref{prop:BG2}.
Suppose that $\gamma$ is a simple closed curve in $\overline{D}$, and suppose it forms the boundary
of an open set $E$ and is locally a graph of a Lipschitz function.
Suppose that $\mu$ is supported on
$\gamma$ and bounded with respect to the natural measure there.
Precisely, let $\Hh$ denote $1$-dimensional Hausdorff measure
and suppose that $\mu=\rho\,d\Hh$ for a bounded function $\rho$.
Our goal is to define $\ip{h,\mu}$ by showing that $\mu\in H(D)^*$.
\begin{lem}
  If $D\subset\RR^2$ is bounded, or it is unbounded and its complement contains an open set,
  then the functional  $f\mapsto \int f\,d\mu$ for $f\in H_s(D)$ extends to an element of $H(D)^*$.
\end{lem}
\begin{proof}
First, suppose that
$D$ is bounded.
It suffices to show that $f\mapsto\int f\,d\mu$ is a bounded linear functional with respect to
the Sobolev norm, since this is equivalent to the one given by the
Dirichlet inner product.
The restriction map $H^1_0(D)\to H^1(E)$ is obviously linear and bounded.
By the Sobolev trace theorem \cite[Theorem~4.3.1]{EG},
there is a bounded trace operator $T\colon H^1(E)\to L^2(d\Hh)$
such that $Tf=f\vert_{\gamma}$ when $f$ is continuous.
Thus for $f\in H_s(D)$, we have
\begin{align*}
  \biggl\lvert\int f\,d\mu\biggr\rvert &\leq \int\abs{Tf}\norm{\rho}_\infty\,d\Hh
   \leq \norm{\rho}_\infty\norm{Tf}_{L^2(d\Hh)}\Hh(\gamma)^{1/2}\leq C\norm{f}_{H^1_0(D)}.
\end{align*}
Thus $f\mapsto \int f\,d\mu$ is bounded and admits a unique extension to all
$f\in H(D)$.

Now, suppose that $D\subset\RR^2$ is unbounded. We will identify $\RR^2$ with $\CC$. 
Suppose that there is a neighborhood
of $z_0\in\CC$ disjoint from $D$. Consider the conformal map $\varphi(z)=1/(z-z_0)$, and let 
$D'=\varphi(D)$, a bounded set. The pushforward measure $\mu'=\mu\circ\varphi$
is supported on $\varphi(\gamma)$, and it has a bounded density with respect to
$\Hh$. By the previous paragraph, for some $C$ and
any $f\in H_s(D)$ we have
\begin{align*}
  \biggl\lvert\int_D f\,d\mu\biggr\rvert = \biggl\lvert\int_{D'}f\circ\varphi\,d\mu'\biggr\rvert
  \leq C\bigl\lVert f\circ\varphi\bigr\rVert_{\grad} = C\norm{f}_\grad
\end{align*}
by the conformal invariance of $\norm{\cdot}_\grad$.
Thus $f\mapsto \int f\,d\mu$ extends to a bounded linear functional on $H(D)$.
\end{proof}

Identifying $\mu$ with its associated element of $H(D)^*$, we have justified
the existence of $\ip{h,\mu}$.
This is the random variable denoted by
$\int_\gamma \rho(z)h(z)\,dz$ in \thref{prop:BG2}.
Together with \eqref{eq:measuregreen}, this explains \thref{prop:BG2}.
}
\subsection{Convergence of fluctuation process to the Gaussian free field}

Recall that $F_t(x)$ counts the eigenvalues of $G(t, 2d(t))$ that
are less than or equal to $2\sqrt{2d(t)-1}x$ and that
\begin{align*}
  \F_t(x) = F_t(x) - \E[F_t(x)\mid N(t)].
\end{align*}
Our goal is to show that $\F_{s+t}(x)$, considered as a function is $x$ and $t$, converges in some sense
to the Gaussian free field. First, we show that integrals against $\F_t(x)$ can be expressed in terms
of traces.
As usual, $T_k(x)$ and $U_k(x)$ denote the Chebyshev polynomials of order $k$ on $[-1,1]$
of the first and second kind, respectively.
\begin{lem}\thlabel{lem:heighttrace}
  \begin{align*}
    \int_{-\infty}^{\infty} U_{k-1}(x)\F_t(x)\,dx
      &= -\frac{1}{k}\Bigl( \tr T_k\bigl(G(t,2d(t))\bigr)
        - \E\bigl[\tr T_k\bigl(G(t,2d(t))\bigr)\,\big\vert\,N(t)\bigr]\Bigr).
  \end{align*}
\end{lem}
\begin{proof}
   As $x\to\pm\infty$, we have $\F_t(x)\to 0$ almost surely.
   Integrating by parts and using the relation $T'_k(x) = kU_{k-1}(x)$, 
  \begin{align*}
    \int_{-\infty}^{\infty} U_{k-1}(x)\F_t(x)\,dx
      &=-\frac{1}{k}\int_{-\infty}^{\infty}T_k(x)\,d\F_t(x)\\
      &= -\frac{1}{k}\sum_{i=1}^{N(t)} T_k(\lambda_i) + \frac{1}{k}\E\Biggl[
        \sum_{i=1}^{N(t)} T_k(\lambda_i) \;\bigg\vert\; N(t)\Biggr],
  \end{align*}
  where $\lambda_1\geq\cdots\geq\lambda_{N(t)}$ are the eigenvalues
  of $G(t)$ divided by $2\sqrt{2d(t)-1}$.
  This is equal to
  \begin{align*}
    \int_{-\infty}^{\infty} U_{k-1}(x)\F_t(x)\,dx
      &= -\frac{1}{k}\Bigl( \tr T_k\bigl(G(t,2d(t))\bigr)
        - \E\bigl[\tr T_k\bigl(G(t,2d(t))\bigr)\,\big\vert\,N(t)\bigr]\Bigr).
  \end{align*}
  Note that when $k$ is even, the $na_0$ term introduced by the trace (see
  \thref{def:tr}) is cancelled by the same term in the expectation.
\end{proof}
Combining this lemma with \thref{thm:diagonallimit}, 
integrals of the form $\int p(x)\F_{s+t}(x)\,dx$ converge jointly as $s\to\infty$
to a Gaussian field indexed by $t$ and by polynomials $p(x)$.
We now express this field in terms of the GFF.

\begin{proof}[Proof of \thref{thm:GFFconvergence}]
\thref{thm:diagonallimit} and \thref{lem:heighttrace} prove that
the integrals 
\begin{align*}
  \int_{-\infty}^{\infty} p_i(x)H_s(x,t_i)\,dx,\qquad i=1,\ldots,n
\end{align*}
converge jointly
to a centered multivariate normal distribution, which is also the distribution
of the integrals against the GFF. We just need to check that the covariances match up.
It suffices to confirm this on a polynomial basis.
By \thref{thm:diagonallimit},
\begin{align}
  \lim_{s\to\infty}
  \E\biggl[\biggl( \int_{-\infty}^{\infty} U_{j-1}(x)H_{s}(x,t_0)\,dx\biggr)
   \Biggl( \int_{-\infty}^{\infty} U_{k-1}(x)H_{s}(x,t_1)\,dx\biggr)\Biggr]
  = \delta_{jk}
    \frac{\pi}{4k}e^{k(t_0-t_1)}\label{eq:covval}
\end{align}
for $t_0\leq t_1$.
By \thref{prop:BG2}, the covariance of 
\begin{align*}
  \int_{-1}^1U_{j-1}(x)h(\Omega(x,t_0))\,dx \text{\quad and\quad}
\int_{-1}^1U_{k-1}(x)h(\Omega(x,t_1))\,dx
\end{align*}
is
\begin{align*}
  I \defeq  -\frac{1}{2\pi}\int_{-1}^1\int_{-1}^1 U_{j-1}(x)\log
    \biggl\lvert \frac{\Omega(x,t_0)-\Omega(y,t_1)}{\Omega(x,t_0)-\overline{\Omega}(y,t_1)}
    \biggr\rvert U_{k-1}(y)\,dx\,dy.
\end{align*}
Substituting $x=\cos u$ and $y=\cos v$, we have
\begin{align}\label{eq:Itobound}
  I = -\frac{1}{2\pi}\int_0^{\pi}\int_0^{\pi} U_{j-1}(\cos u)\sin u
      \log  \biggl\lvert \frac{e^{t_0+iu}-e^{t_1+iv}}{e^{t_0+iu}-e^{t_1-iv}}
    \biggr\rvert U_{k-1}(\cos v)\sin v\,du\,dv.
\end{align}
Assume that $t_0<t_1$.
For any constant $w\in\CC$ with $\abs{w}=t_1$, we can define
  functions $\log(z-w)$ and $\log(z-\overline{w})$ that are analytic on $\abs{z}<t_1$.
  For each $v$, we choose two such logarithm functions with $w=e^{t_1+iv}$ to get
  \begin{align*}
    \log\biggl\lvert \frac{e^{t_0+iu}-e^{t_1+iv}}{e^{t_0+iu}-e^{t_1-iv}} \biggr\rvert
     &= \frac12\Bigl( \log\bigl(e^{t_0+iu}-e^{t_1+iv}\bigr)
       +\log\bigl(e^{t_0-iu}-e^{t_1-iv}\bigr)\\
       &\qquad\qquad-\log\bigl(e^{t_0+iu}-e^{t_1-iv}\bigr)
       -\log\bigl(e^{t_0-iu}-e^{t_1+iv}\bigr)
       \Bigr).
  \end{align*}
  Using the relation $U_{n-1}(\cos x) = \sin(nx)/\sin x$, we then have
  \begin{align*}
    I &= -\frac{1}{4\pi}\int_0^\pi\int_0^{2\pi}\sin(ju)\sin(kv) \Bigl(
      \log\bigl(e^{t_0+iu}-e^{t_1+iv}\bigr) - \log\bigl(e^{t_0+iu}-e^{t_1-iv}\bigr)\Bigr)du\,dv,
  \end{align*}
  and by integrating by parts in $u$,
  \begin{align*}
    I
      &= -\frac{1}{4j\pi}\int_0^\pi\int_0^{2\pi}\cos(ju)\sin(kv)
      \biggl( \frac{ie^{t_0+iu}}{e^{t_0+iu}-e^{t_1+iv}} - \frac{ie^{t_0+iu}}{e^{t_0+iu}-e^{t_1-iv}}\biggr)du\,dv\\
     &= -\frac{1}{4j\pi}\int_0^{2\pi}\int_0^{2\pi}\cos(ju)\sin(kv)
      \frac{ie^{t_0+iu}}{e^{t_0+iu}-e^{t_1+iv}}\,du\,dv.
  \end{align*}
  We then integrate by parts in $v$ to get
  \begin{align*}
    I &= \frac{1}{4jk\pi}\int_0^{2\pi}\int_0^{2\pi}\cos(ju)\cos(kv)
      \frac{e^{t_0+t_1+i(u+v)}}{\bigl(e^{t_0+iu}-e^{t_1+iv}\bigr)^2}\,du\,dv.
  \end{align*}
  Let $\gamma$ denote a counterclockwise path around the unit disc.
  \begin{align*}
    I &= \frac{1}{4jk\pi}\int_0^{2\pi}\cos(kv)\int_\gamma
      \frac{z^{j}+z^{-j}}{2iz}\frac{e^{t_0+t_1+iv}z}{\bigl(e^{t_0}z - e^{t_1+iv}\bigr)^2}
        dz\,dv.
  \end{align*}
  The integrand of the path integral has a single pole in the unit disc at $0$, and the residue
  there is $je^{j(t_0-t_1-iv)}/2i$. This gives
  \begin{align*}
    I &= \frac{1}{4k}\int_0^{2\pi}\cos(kv)
       e^{j(t_0-t_1-iv)}\,
    dv\\
    &= \frac{1}{4k}\int_\gamma\frac{w^{k}+w^{-k)}}{2iw}
       e^{j(t_0-t_1)}w^{-j}\,
    dw \\
    &= \frac{e^{j(t_0-t_1)}}{8ik}
      \int_\gamma \bigl(w^{k-j-1}+w^{-k-j-3}\bigr)
       dw.
  \end{align*}
  By computing residues, this is $\pi e^{j(t_0-t_1)}/4k$ if $j=k$ and $0$ otherwise,
  agreeing with \eqref{eq:covval} for all $t_0<t_1$. 
  To extend this to $t_0=t_1$ by a limiting argument, we apply the dominated convergence
  theorem to the integral in \eqref{eq:Itobound}.
  One can show that
  \begin{align*}
    \Biggl\lvert\log  \biggl\lvert \frac{e^{t_0+iu}-e^{t_1+iv}}{e^{t_0+iu}-e^{t_1-iv}}
    \biggr\rvert\Biggr\rvert
    &\leq \log\biggl\lvert\frac{e^{iu} - e^{-iv}}{e^{iu}-e^{iv}}\biggr\rvert
  \end{align*}
  for all $t_0\leq t_1$.  The right-hand side of this equation
  is integrable over $0\leq u,v\leq \pi$. The other factors of the integrand in
  \eqref{eq:Itobound} are bounded there. Thus by the dominated convergence theorem
  we can compute $I$ when $t_0=t_1$ by letting $t_0\to t_1$ from below.
\end{proof}

%
%
\bibliographystyle{alpha}
\bibliography{thesis}
%
%
\appendix
\raggedbottom\sloppy
 

\end{document}